\documentclass[11pt,reqno]{amsart}
\usepackage{todonotes}
\title[]{On large periodic traveling wave solutions to the free boundary Stokes and Navier-Stokes equations }

\author{Seyed A. Banihashemi}
\address{Department of Mathematics, University of Maryland, College Park, MD 20742}
\email[S. Banihashemi]{sabani@umd.edu}

\author{Huy Q. Nguyen}
\address{Department of Mathematics, University of Maryland, College Park, MD 20742}
\email[H. Nguyen]{hnguye90@umd.edu}

\usepackage[margin=1in]{geometry}
\usepackage{amsmath, amsthm, amssymb, mathrsfs, stmaryrd}
\usepackage{times}
\usepackage{color}
\usepackage[colorlinks=true, pdfstartview=FitV, linkcolor=blue, citecolor=blue, urlcolor=blue]{hyperref}
\usepackage{hyperref}
\usepackage[nameinlink,noabbrev,capitalise]{cleveref}
\usepackage{mathtools}
\newcommand{\bq}{\begin{equation}}
\newcommand{\eq}{\end{equation}}
\newcommand{\bqa}{\begin{eqnarray*}}
\newcommand{\eqa}{\end{eqnarray*}}

\theoremstyle{plain}
\newtheorem{theo}{Theorem}[section]
\newtheorem{prop}[theo]{Proposition}
\newtheorem{lemm}[theo]{Lemma}
\newtheorem{coro}[theo]{Corollary}

\newtheorem{defi}[theo]{Definition}
\theoremstyle{definition}
\newtheorem{rema}[theo]{Remark}
\newtheorem{nota}[theo]{Notation}

\DeclareMathOperator{\Id}{I}

\DeclareSymbolFont{pletters}{OT1}{cmr}{m}{sl}
\DeclareMathSymbol{s}{\mathalpha}{pletters}{`s}

\def\tt{\theta}
\def\eps{\varepsilon}
\def\na{\nabla}
\def\la{\left\lvert}
\def\lA{\left\lVert}
\def\ra{\right\vert}
\def\rA{\right\Vert}

\def\mez{\frac{1}{2}}
\def\tdm{\frac{3}{2}}

\def\fdm{\frac{5}{2}}
\def\Rr{\mathbb{R}}
\def\T{\mathbb{T}}
\def\Nn{\mathbb{N}}
\def\Zz{\mathbb{Z}}

\def\Dd{\mathbb{D}}

\def\cF{\mathcal{F}}

\def\cN{\mathcal{N}}

\def\cP{\mathcal{P}}
\def\cT{\mathcal{T}}

\def\cG{\mathcal{G}}

\def\cH{\mathcal{H}}
\def\cL{\mathcal{L}}
\def\ld{\lambda}

\def\rH{\mathring{H}}
\def\p{\partial}
\def\na{\nabla}
\def\wc{\rightharpoonup}

\def\om{\omega}

\def\etanorm{\|\eta\|_{H^{s+\tdm}(\T^d)}}

\def\dv{\text{div }}

\numberwithin{equation}{section}

\begin{document}
\newcommand{\Huy}[1]{{\color{orange} \textbf{H:} #1}}
\newcommand{\Seyed}[2]{{\color{red} \textbf{S:} #2}}
\begin{abstract}
We study the free boundary problem for a finite-depth layer of viscous incompressible fluid in arbitrary dimension, modeled by the Stokes or Navier-Stokes equations. In addition to the gravitational field acting in the bulk, the free boundary  is acted upon by surface tension and an external stress tensor posited  to be in traveling wave form. We prove that for any isotropic  stress tensor with periodic profile, there exists a locally unique periodic traveling wave solution, which can have large amplitude.  Moreover, we prove that  the constructed traveling wave solutions are asymptotically stable for the dynamic free boundary Stokes equations. Our proofs rest on the analysis of the nonlocal normal-stress to normal-Dirichlet operators for the Stokes and Navier-Stokes equations in domains of Sobolev regularity. 
\end{abstract}

\keywords{traveling waves, free boundary, Stokes equations, Navier-Stokes equations, asymptotic stability}

\noindent\thanks{\em{ MSC Classification: 35R35, 35Q35, 35Q30, 35C07, 35B35.}}

\maketitle

\tableofcontents

\section{Introduction}
This paper is concerned with the existence and stability of traveling surface waves for {\it viscous} fluids. The corresponding theory for inviscid fluids, since the pioneering work \cite{Stokes} of Stokes, has flourished into a subject of immense impact in fluid mechanics and partial differential equations \cite{Ebbandflow}. The viscous theory, on the other hand, has only started recently after the work \cite{LeoniTice} of Leoni and Tice, providing  the first construction of solitary waves for the free boundary incompressible Navier-Stokes equations. A distinction between the two theories is that besides gravity and  surface tension, some types of viscous traveling waves cannot be generated without additional  external forces of traveling wave form. Experimental works  \cite{DCDA1, DCDA2, MasnadiDuncan, ParkCho1} have  observed  the appearance of traveling surface waves when a tube blowing air on a viscous fluid is translated uniformly above the surface. Motivated by this, \cite{LeoniTice} constructed  small solitary waves in the presence of  either a small stress tensor acting on the free surface or a small bulk force. This construction has been extended in various directions,  including periodic and tilted configurations \cite{KoganemaruTice1}, multi-layer configuration \cite{StevensonTice1}, the vanishing wave speed limit \cite{StevensonTice3}, the Navier-slip boundary conditions on the bottom \cite{KoganemaruTice2}, and the compressible Navier-Stokes equations \cite{StevensonTice2}. For Darcy flows, a similar construction was done in \cite{NguyenTice} which furthermore proved  the asymptotic stability of small periodic traveling waves.  In the absence of additional external forces, all of the above viscous free boundary problems admit a trivial solution with a flat free boundary. {\it Small} traveling wave solutions were  constructed perturbatively from the trivial one when {\it small} external forces are applied. We mention that the recent work  \cite{StevensonTice5} provides the first  construction of bore wave solutions to the free boundary Navier-Stokes equations on an inclined plane. This also provides the first type of viscous traveling surface waves sustained by gravity alone.

In view of the aforementioned results, a natural question is whether {\it large} traveling surface waves can be generated by {\it large} external forces. This is a large-data well-posedness issue for quasilinear problems. Large periodic traveling surface waves for Darcy flows have been constructed  in \cite{Nguyen2023-capillary} and \cite{BN} via two different methods.  By means of  global continuation based upon the Leray-Schauder degree theory, \cite{Nguyen2023-capillary}  produced large waves  in the presence of surface tension, and the waves can assume any given speed. Surface tension was instrumental in gaining the compactness needed for the global implicit function method in \cite{Nguyen2023-capillary}. On the other hand,  by discovering a hidden ellipticity, \cite{NguyenStevenson} extended the construction in \cite{Nguyen2023-capillary} to the 2D problem without surface tension. The second method introduced in \cite{BN} is based on perturbations of  large steady states to produce slowly traveling waves that are  of arbitrary size and  asymptotically stable. In the present paper,  we adopt the method in \cite{BN} to construct large periodic traveling wave solutions for the free boundary Stokes and Navier-Stokes equations with surface tension. Moreover, we  prove that the constructed solutions are asymptotically stable for the dynamic Stokes problem. This extension from Darcy flows to Stokes and Navier-Stokes flows  is not a  trivial undertaking and necessitates a number of  properties of the  normal-stress to normal-Dirichlet operator - the analog of the well studied Dirichlet to Neumann operator for potential flows. The  results obtained herein for this operator are of independent interest. Our  stability result for Stokes flows together with the one in \cite{BN} for Darcy flows shows a stark contrast between periodic traveling surface waves for viscous and inviscid fluids. Indeed, the classical Stokes wave solutions \cite{Stokes} of the water wave problem for inviscid fluids have been proven to be unstable with respect to various periodic perturbations in the longitudinal direction \cite{BenjaminFeir, BM, NguyenStrauss, ChenSu} and  the transverse direction \cite{CNS1, CNS2}.
\subsection{The free boundary Stokes and Navier-Stokes equations}
We consider a finite-depth layer of a viscous incompressible fluid above a flat bottom:
\[
\Omega_\eta = \{(x,y) \in \mathbb{T}^d\times \mathbb{R} : -b< y < \eta(x, t)\},
\]
where $\eta(x, t): \mathbb{T}^d\times [0, \infty) \to (-b,\infty)$ parametrizes the free boundary, and $\{y=-b\}$ is the flat bottom.  We  denote by
\[
   \Sigma_f = \{(x,y) \in \mathbb{T}^d\times \mathbb{R} : y= f(x)\}
   \]
the surface parametrized by $f: \T^d\to \Rr$. The fluid is acted upon by gravity $-g e_y$, surface tension,  and an external stress tensor $\cT(x, y, t): \Omega_{\eta}\to \Rr^{d+1}$ on free boundary $\Sigma_\eta$.  The  evolution of the fluid and the free boundary is governed by the system
   \begin{equation}\label{sys:S-NS}
    \begin{cases}
   \alpha(\p_t v+v\cdot\nabla v) - \Delta v + \nabla P = -ge_y \qquad & \text{in }\Omega_\eta, \\
    \nabla\cdot v= 0 \qquad & \text{in }\Omega_\eta, \\
    (PI - \mathbb{D}v)\mathcal{N} = (\sigma \cH(\eta)+ \cT)\mathcal{N} \qquad & \text{on }\Sigma_\eta, \\
    \p_t\eta=v\cdot\mathcal{N}  \qquad &\text{on } \Sigma_\eta, \\
    v=0 \qquad & \text{on }\Sigma_{-b},
    \end{cases}
\end{equation}
where $\mathcal{N} = (-\nabla_x\eta , 1)$ is the normal to $\Sigma_\eta$, $\mathbb{D}v = \nabla v + (\nabla v)^T$ is the strain rate  tensor, $\cH(\eta)= -\dv (\frac{\nabla_x\eta}{\sqrt{1+|\nabla_x \eta|^2}})$ is the mean curvature of $\Sigma_\eta$, and $\sigma>0$ is the surface tension coefficient.  Here $\alpha=0$ corresponds to the Stokes equations, and $\alpha=1$ the Navier-Stokes equations.   It is convenient to incorporate gravity into the pressure by setting  $p=P+gy$. Toward the construction of large traveling wave solutions, we consider external stress tensors of the traveling wave form 
\bq\label{form:stress}
\cT(x, y, t)=\cT_0(x- \gamma e_1 t),\quad \cT_0: \T^d\to \Rr^{d+1},
\eq
where $\gamma \in \Rr$ is the speed and we have assumed without loss of generality that the propagation  direction is $x_1$. Stress tensors of the  form  \eqref{form:stress} are uniform in the vertical direction. Traveling wave solutions of \eqref{sys:S-NS} are given by 
\[
\eta(x, t)=\eta(x-\gamma e_1t),\quad v(x, y, t)=v(x-\gamma e_1t, y),\quad p(x, t)=p(x-\gamma e_1t, y),
\]
where the profile functions solve the system 
\begin{equation}\label{sys:main}
    \begin{cases}
    \alpha(v-\gamma e_1)\cdot\nabla v - \Delta v + \nabla p = 0 \qquad & \text{in }\Omega_\eta, \\
    \nabla\cdot v= 0 \qquad & \text{in }\Omega_\eta, \\
    (pI - \mathbb{D}v)\mathcal{N} = (\sigma \cH(\eta)+ g\eta + \cT_0)\mathcal{N} \qquad & \text{on }\Sigma_\eta, \\
    v\cdot\mathcal{N} = - \gamma\partial_1 \eta \qquad &\text{on } \Sigma_\eta, \\
    v=0 \qquad & \text{on }\Sigma_{-b}.
    \end{cases}
\end{equation}
\subsection{Main results}
Our first main result concerns the existence of {\it large} periodic  traveling wave solutions to the Stokes and Navier-Stokes equations \eqref{sys:main}. To this end,  we observe that for any given scalar function  $\phi: \T^d\to \Rr$, the  system \eqref{sys:main} with isotropic stress tensor \bq\cT_0=\phi I \eq
admits the (steady) solution  
\bq\label{unperturbed}
\gamma=0,~v=0,~ p=0,~ \eta_*=(\sigma \cH + gI)^{-1}(-\phi),
\eq
 Our idea is to perturb the {\it large} solution \eqref{unperturbed} to obtain slowly traveling wave solutions. Since the strength $\phi$ of the stress tensor can be arbitrarily large, so can the size of the traveling wave profile $\eta$. 

\begin{theo}\label{theo1:intro}
Let $(d+1)/2<s\in \Nn$, and assume that $\phi\in H^{s-\mez}(\T^d)\cap C^{1}(\T^d)$ with  $\min_{\T^d} (-\phi)>-g b$. Then there exist $\delta_0=\delta_0(\| \phi\|_{H^{s-\mez}\cap C^{1}})>0$  such that for any $ \delta\in (0,  \delta_0)$, there is some $\gamma_\delta>0$ such that for all $|\gamma|< \gamma_\delta$, the traveling wave system \eqref{sys:main} with $\alpha \in \{0, 1\}$ and $\cT_0=\phi I$ has a unique solution $U=(v, p, \eta)$ in the ball $B(U_*, \delta)$ in ${}_0H^{s+1}_\sigma(\Omega_\eta)\times H^s(\Omega_\eta)\times H^{s+\tdm}(\T^d)$, where $U_*=(0, 0, \eta_*)$ and the space ${}_0H^{s+1}_\sigma(\Omega_\eta)$ is defined by \eqref{div free space}.
\end{theo}
In other words, in any sufficiently small neighborhood of $\eta_*$ there exists a unique traveling wave $\eta$ with sufficiently small speed $\gamma$. Some remarks on \cref{theo1:intro} are in order.  
\begin{itemize}
\item[(1)]  \cref{theo1:intro} combines \cref{large traveling wave for Stokes} and \cref{large traveling wave for NS} for the Stokes and Navier-Stokes problems, respectively.   It is a consequence of the proof of these theorems that the mapping $\gamma\mapsto U$ is Lipschitz continuous. 
\item[(2)] Our proof, which relies on physical space variables, can be  adapted to the case of non-flat bottoms $\{y=b(x)\}$, $b:\T^d\to \Rr$.
\item[(3)] For $\eta\in H^{s+\tdm}(\T^d)$, the condition  $s>(d+1)/2$ is minimal to guarantee that the free boundary has bounded curvature. 
\item[(4)] We will  construct traveling wave solutions using the contraction mapping method.  Surface tension is needed for the associated nonlinear map to be an endomap, as will be explained in \cref{sec:proofThm1}.  \cref{theo1:intro} thus leaves open the existence of large periodic traveling waves without surface tension. On the other hand, the existence of large solitary waves is open both  in the presence and absence of surface tension. 
\end{itemize}
Our  second main result concerns the stability of a class of large periodic traveling waves for the free boundary Stokes equations. 
\begin{theo}\label{theo2:intro} Let $(d+1)/2+1<s\in \Nn$ and suppose that  $\eta_*\in  H^{s+\frac52}(\T^d)$. Then there exist a positive constant  $\gamma^\dag$ and a nonincreasing function $\om: \Rr_+\to \Rr_+$, both depending on $\eta_*$,  such that the following holds.  Any traveling wave solution $\eta_w\in H^{s+\tdm}(\T^d)$ of the free boundary Stokes equations with speed $\gamma \in (-\gamma^\dag, \gamma^\dag)$ and satisfying 
\bq\label{cd:stab} 
\| \eta_w-\eta_*\|_{H^{s+\tdm}(\T^d)}\le \om(\|\eta_w\|_{H^{s+\tdm}})
\eq
is asymptotically stable with respect to perturbations in $\rH^{s+1}(\T^d)$. Moreover, the perturbed solution converges exponentially fast to $\eta_w$ as $t\to \infty$. 
 \end{theo}
We refer to \cref{theo:stab} for a more precise dependence of $\gamma^\dag$ and $\om$ on $\eta_*$. The condition \eqref{cd:stab} imposes a size restriction on  $\eta_w-\eta_*$ but not on $\eta_w$, thereby allowing for large $\eta_w$. In particular, \cref{theo1:intro} provides a class of  traveling waves  satisfying condition \eqref{cd:stab}. Indeed,  for any given  $\phi\in H^{s-\mez}(\T^d)$ with $(d+1)/2+1<s\in \Nn$, \cref{theo1:intro} yields  the existence of  a unique traveling wave $\eta_w$ with speed $\gamma\in (-\gamma_\delta, \gamma_\delta)$ provided   
\[
\| \eta_w-\eta_*\|_{H^{s+\tdm}}<\delta<\delta_0(\| \phi\|_{H^{s-\mez}}).
\]
It follows that $\| \eta_w\|_{H^{s+\tdm}}< \delta_0(\| \phi\|_{H^{s-\mez}})+ \| \eta_*\|_{H^{s+\tdm}}$, and hence \eqref{cd:stab} is satisfied if 
\bq\label{cd:stab:2}
\delta\le \om\big(\delta_0(\| \phi\|_{H^{s-\mez}})+ \| \eta_*\|_{H^{s+\tdm}}\big).
\eq
We note that the right-hand side of \eqref{cd:stab:2} depends only on $\phi$ since $\eta_*=(\sigma \cH + gI)^{-1}(-\phi)$. We thus obtain the following corollary. 
\begin{coro}\label{coro:intro} Let $(d+1)/2+1<s\in \Nn$ and  $\phi\in H^{s+\mez}(\T^d)$. Then any traveling wave solution $\eta_w$ given in \cref{theo1:intro} for the free boundary Stokes equations  is asymptotically stable in $\rH^{s+1}(\T^d)$ provided 
\bq\label{cd:stab:3}
\begin{cases}
\| \eta_w-\eta_*\|_{H^{s+\tdm}}<\delta< \min\left\{\delta_0(\| \phi\|_{H^{s-\mez}}),  \om\big(\delta_0(\| \phi\|_{H^{s-\mez}})+ \| \eta_*\|_{H^{s+\tdm}}\big)\right\},\\
|\gamma|<\min\{\gamma_\delta, \gamma^\dag\}. 
\end{cases}
\eq
\end{coro} 
In other words, upon shrinking the neighborhood of $\eta_*$ and the absolute speed $|\gamma|$ so as to satisfy \eqref{cd:stab:3}, all traveling waves  are asymptotically stable for dynamic free boundary Stokes equations. 
\subsection{Method of proof}  
The proofs of both \cref{theo1:intro} and  \cref{theo2:intro} rely on the analysis of  nonlocal  operators that map the normal stress  to the normal trace of the velocity for the Stokes and Navier-Stokes equations. For the flat free boundary $\eta=0$, \cite{LeoniTice} developed the asymptotics  of the symbol of this operator (for the Stokes equations) in the small frequency regime, leading to a tailored scale of Sobolev spaces in which small solitary waves exist. For our purpose of constructing large waves, we will develop a physical-space-based approach to analyze these operators in non-flat domains of Sobolev regularity.
\subsubsection{Construction of large traveling waves}\label{sec:proofThm1}
 We first  consider the $\gamma$-Stokes problem with prescribed normal stress that is {\it parallel} to $\cN$: 
\bq\label{sys:linOp}
\begin{cases}
    -\gamma \partial_1 v - \Delta v + \nabla p = 0,\quad  \nabla\cdot v= 0  &\quad\text{in~} \Omega_\eta, \\
    (pI - \mathbb{D}v)(\cdot, \eta(\cdot))\mathcal{N}(\cdot) = \chi(\cdot)\cN(\cdot) &\quad\text{on~}\T^d, \\
    v=0 &\quad \text{on~} \Sigma_{-b},
\end{cases}
\eq
where $\chi: \T^d\to \Rr$ is a scalar function. This leads to the {\it linear  normal-stress to normal-Dirichlet} operator 
\bq\label{Psi:intro}
\Psi_\gamma[\eta]\chi :=v\vert_{\Sigma_\eta}\cdot \cN.
\eq
Then, for the Stokes case,    the system \eqref{sys:main} with $\alpha=0$ can be recast into the following single equation for $\eta$:
\bq\label{eq:tw:intro}
\mathcal{G}(\gamma, \eta):= \Psi_0[\eta]\Big((\sigma \cH + gI)(\eta)+\phi\Big)+\gamma\p_1\eta=0
\eq
Our goal is then to construct a solution   $\eta$ near $\eta_*=(\sigma \cH + gI)(-\phi)$. This requires good understanding of the operator $\Psi_\gamma[\eta]$.  In order to obtain the continuity of $\Psi_\gamma[\eta]$ in Sobolev spaces, we will consider the more general problem 
\bq\label{sys:linOp}
\begin{cases}
    -\gamma \partial_1 v - \Delta v + \nabla p = 0, \quad  \nabla\cdot v= 0 &\quad\text{in~} \Omega_\eta, \\
    (pI - \mathbb{D}v)(\cdot, \eta(\cdot))\mathcal{N}(\cdot) = k&\quad\text{on~}\T^d, \\
    v=0 &\quad \text{on~} \Sigma_{-b},
\end{cases}
\eq
 where  $k: \T^d\to \Rr^{d+1}$ is a vector-valued function. Assume that 
 \bq\label{reg:eta:intro}
 \eta\in H^{s+\tdm}(\T^d)~\text{and~}\phi\in H^{s-\mez}(\T^d)\cap C^{1}(\T^d)~\text{with~}(d+1)/2<s\in \Nn.
 \eq
  We will prove in \cref{reg for stokes with stress} that 
    \bq\label{reg:intro}
    \|v\|_{H^{\sigma+1}(\Omega_\eta)} + \|p\|_{H^{\sigma}(\Omega_\eta)} \leq C(\|\eta\|_{H^{s+\tdm}(\T^d)}) \|k\|_{H^{\sigma-\mez}(\T^d)},\quad \sigma \in [0, s+1],
\eq
 provided $|\gamma|<\gamma^*=\gamma^*(\| \eta\|_{W^{1, \infty}})$. This is an improvement over the classical regularity result which assumes  $\eta\in W^{s+2, \infty}(\T^d)$. We note that $\sigma=s+1$ is the maximal regularity allowed for $k$ in \eqref{reg:intro} since $\cN\in H^{s+\mez}(\T^d)$. In particular, for $k=\chi \cN$ as in \eqref{sys:linOp} we can take $\chi\in H^{\sigma-\mez}(\T^d)$  with $\sigma \in [0, s+1]$. By virtue of the regularity estimate \eqref{reg:intro} and the trace inequality, we deduce that $\Psi_\gamma[\eta]\chi=v\vert_{\Sigma_\eta}\cdot \cN\in H^{\min\{\sigma+\mez, s+\mez\}}$, which is limited by the $H^{s+\mez}$ regularity of $\cN$ even though $\sigma$ can be taken in $(s, s+1]$. Thus 
 \bq\label{cont:Psi:intro}
 \Psi_\gamma[\eta]: H^{\sigma-\mez}(\T^d)\to H^{\sigma+\mez}(\T^d)\text{~ is continuous if ~}  \sigma \in [0, s]~\text{and~}|\gamma|\le \gamma^*.
 \eq

The continuity \eqref{cont:Psi:intro} with $\sigma=s$ this implies $\cG: \Rr\times H^{s+\tdm}(\T^d)\to H^{s+\tdm}(\T^d)$.  In order to construct a solution $(\gamma, \eta)$ of \eqref{Psi:intro} near $(0, \eta_*)$, $\eta_*=(\sigma \cH + gI)^{-1}(-\phi)\in H^{s+\tdm}(\T^d)$ via the implicit function theorem, the key condition  is the invertibility of the  Fr\'echet derivative $\p_\eta \cG(0, \eta_*)$. A formal calculation yields  $D_\eta \cG(0, \eta_*)f= \Psi_0[\eta_*]T_{\eta_*}f$, where $T_{\eta_*}:=D_\eta(\sigma \cH(\eta_*) + g\eta_*)$ is a second order elliptic operator (see \eqref{def: a_ij}). Therefore, the invertibility of $\p_\eta \cG(0, \eta_*)$ is equivalent to the invertibility of $\Psi_0[\eta_*]: H^{s-\mez}(\T^d)\to H^{s+\mez}(\T^d)$. More generally, given $\eta \in H^{s+\tdm}(\T^d)$ and $|\gamma|\le \gamma^*$, we aim to establish  the invertibility of  $\Psi_\gamma[\eta]: H^{s-\mez}(\T^d)\to H^{s+\mez}(\T^d)$. To this end, given a {\it scalar} function $h: \T^d\to \Rr$, we need to find a scalar function $\chi: \T^d\to \Rr$ such that the solution $(v, p)$ of \eqref{sys:linOp} satisfies $v\vert_{\Sigma_\eta}\cdot \cN=h$. This tacitly requires that the normal stress  $(pI - \mathbb{D}v)(\cdot, \eta(\cdot))\mathcal{N}(\cdot)=\chi \cN$ is {\it parallel} to $\cN$, a property not satisfied by all solutions of 
\begin{equation}\label{sys:invOp:0}
    \begin{cases}
    -\gamma \partial_1 v + \dv(pI - \Dd v) = 0,\quad  \nabla\cdot v= 0  \qquad &\text{in } \Omega_\eta, \\
  %  \cN^\perp(pI - \mathbb{D}v)(\cdot, \eta(\cdot))\mathcal{N}(\cdot) = l(\cdot) \qquad &\text{on }\T^d, \\
    v(\cdot,\eta(\cdot))\cdot \cN(\cdot) = h(\cdot) \qquad &\text{on }\T^d, \\
    v=0 \qquad &\text{on } \Sigma_{-b}.
    \end{cases}
\end{equation}
Fortunately, when supplementing the underdetermined problem \eqref{sys:invOp:0} with the desired condition 
\bq\label{parallel}
\cN^\perp\Big((pI - \mathbb{D}v)(\cdot, \eta(\cdot))\mathcal{N}(\cdot)\Big) = 0\qquad \text{on }\T^d,
\eq
we obtain and well-posed problem, where $\cN^\perp$ denotes the projection onto the hypersurface orthogonal to $\cN$. See \cref{weak sol for stokes with Dirichlet}. Moreover, we will prove in \cref{reg for stokes with Dirichlet} the regularity 
\bq
    \|v\|_{H^{\sigma+1}(\Omega_\eta)} + \|p\|_{H^{\sigma}(\Omega_\eta)} \leq C(\|\eta\|_{H^{s+\tdm}(\T^d)})\|h\|_{H^{\sigma+\mez}(\T^d)},\quad \sigma \in [0, s].
\eq
This in turn implies that 
\bq\label{invert:Psi:intro}
\Psi_\gamma[\eta]: \rH^{\sigma-\mez}(\T^d)\to \mathring{H}^{\sigma+\mez}(\T^d)
\eq
is an isomorphism for all $\sigma \in [0, s]$. In particular, the case $\sigma=s$ provides the desired invertibility of $D_\eta \cG(0, \eta_*): H^{s+\tdm}(\T^d)\to \rH^{s+\mez}(\T^d)$ when $\eta_*\in H^{s+\tdm}(\T^d)$. For the actual construction of  traveling waves, to avoid the technical issue of justifying the $C^1$ regularity of the nonlocal map $\cG$, we will employ the Banach fixed point  method. As a trade-off, this approach requires in addition to the invertibility \eqref{invert:Psi:intro} contraction estimates for $\Psi_\gamma[\eta_1]-\Psi_\gamma[\eta_2]$, which will be established in \cref{linearization and contraction for Psi}.

We remark that surface tension is necessary for the above proof. Indeed, without surface tension $\cG$ becomes 
\[
 \cG(\gamma, \eta)=\Psi_0[\eta](g\eta +\phi)+\gamma\p_1\eta
 \]
and $\cG: \Rr\times H^{s+\tdm}(\T^d)\to H^{s+\mez}(\T^d)$  by virtue of the continuity \eqref{invert:Psi:intro} . Moreover, we have  $\eta_*=-\frac{1}{g} \phi$ and $D_\eta\cG(0, \eta_*)=g\Psi_0[\eta_*]$ which fails to be an isomorphism from from $\rH^{s+\tdm}(\T^d)$ to $\rH^{s+\mez}(\T^d)$ in view of \eqref{invert:Psi:intro} with $\sigma=s$.

As for the free boundary Navier-Stokes problem, we consider the Navier-Stokes version of   \eqref{sys:linOp}:
\bq\label{sys:nonlinOp}
\begin{cases}
    -\gamma \partial_1 v - \Delta v + \nabla p = -v\cdot \na v,\quad  \nabla\cdot v= 0  &\quad\text{in~} \Omega_\eta, \\
    (pI - \mathbb{D}v)(\cdot, \eta(\cdot))\mathcal{N}(\cdot) = \chi(\cdot)\cN(\cdot) &\quad\text{on~}\T^d, \\
    v=0 &\quad \text{on~} \Sigma_{-b}.
\end{cases}
\eq
 The corresponding {\it nonlinear} normal-stress to normal-Dirichlet operator defined as in \eqref{Psi:intro} is denoted by $\Phi_\gamma[\eta]$. In order to establish the well-posedness of \eqref{sys:nonlinOp}, we consider the nonhomogeneous version of \eqref{sys:linOp} with a forcing term $f$ in the $\gamma$-Stokes equation. Then, by choosing $f=-v\cdot \na v$ for any given $v\in H^{s+1}(\Omega_\eta)$, we obtain via the Banach fixed point method that \eqref{sys:nonlinOp} has a unique small solution $v$ for  $\chi$ sufficiently small in $H^{s-\mez}(\T^d)$. This results in the continuity 
 \bq
 \Phi_\gamma[\eta]: B_{H^{s-\mez}(\T^d)}(0, \delta^0)\to H^{s+\mez}(\T^d)
 \eq
for some $\delta^0>0$. The smallness of $v$ is not an obstruction for our construction of traveling waves since we are   perturbing the solution   \eqref{unperturbed} which has $v=0$. By an analogous method, we can show that the inverse of $\Phi_\gamma[\eta]$ is well-defined on a small ball of $\rH^{s+\mez}(\T^d)$ and establish  contraction estimates for  $\Phi_\gamma[\eta_1]-\Phi_\gamma[\eta_2]$. 
\subsubsection{Stability of large traveling waves}
Using the operator $\Psi_\gamma[\eta]$ we can recast the dynamic free boundary Stokes equations as 
\bq
\p_t \eta=\gamma \p_1 \eta+\Psi_0[\eta]\Big((\sigma \cH + gI)(\eta)+\phi\Big).
\eq
Fixing a traveling wave solution $\eta_w$ with speed $\gamma$, we aim to establish its asymptotic stability. Denoting the perturbation by $f=\eta-\eta_w$ and discarding certain $O(f^2)$ terms, we consider   the ``almost linearized'' problem 
\bq\label{lin:intro:0}
\p_t f=\gamma \p_1 \eta+\Psi_0[\eta_w]T_{\eta_w}f+S(f), %\quad Af:= D_S \Psi_0[\eta_w]\big((\sigma \cH + gI)(\eta_w)+\phi\big)f,
\eq
\bq
S(f):= \Psi_0[\eta_w+f]\big((\sigma \cH + gI)(\eta_w)+\phi\big)-\Psi_0[\eta_w]\big((\sigma \cH + gI)(\eta_w)+\phi\big).
\eq
For the genuine linearized problem, $S$ must be replaced with the shape-derivative 
\[
\lim_{\eps \to 0} \frac{1}{\eps}\Big(\Psi_\gamma[\eta_w+\eps f](\chi) -\Psi_\gamma[\eta_w](\chi)\Big).
\]
We choose to work with the simpler operator  $S$ so as to exploit our contraction estimates for $\Psi_0[\eta_w+f]-\Psi_0[\eta_w]$.  Both $\Psi_0[\eta_w]T_{\eta_w}$ and $S$ are first-order  operators. Our strategy for proving  the ``linear'' asymptotic stability  of $f=0$  in $\rH^{s+1}(\T^d)$ relies on the following properties of $\Psi_0[\eta_w]$:

(i) $-\Psi_0[\eta_w]$ satisfies the coercive estimate  
\bq\label{coercive:intro}
 \langle -\Psi_0[\eta_w] \chi , \chi \rangle_{H^\mez(\T^d),H^{-\mez}(\T^d)}  \geq c_{\Psi_0[\eta_w] }\|\chi\|^2_{H^{-\mez}(\T^d)}.
\eq

(ii) The commutator estimate 
\bq\label{cmtt:intro}
 \|[\Psi_0[\eta_w],\partial^\alpha]\chi\|_{H^{\tdm}(\T^d)}\leq C(\|\eta_w\|_{H^{s+\tdm}(\T^d)})\|\chi\|_{H^{s+\mez}(\T^d)},\quad  |\alpha|= s+1,
\eq 
holds for $\eta_w\in H^{s+\frac52}$ with $(d+1)/2<s\in \Nn$.

 Assuming  above properties, let us demonstrate the stability.  First, we integrate  \eqref{lin:intro:0}  against $T_{\eta_w} f$ and substitute $\chi=T_{\eta_w}f$ in (i) to obtain  
\bq\label{stabest:low}
 \langle \p_t f , T_{\eta_w} f \rangle_{H^\mez,H^{-\mez}} \le -c_{\Psi_0[\eta_w] }\| T_{\eta_w} f\|_{H^{-\mez}}^2+  \langle S(f) , T_{\eta_w} f \rangle_{H^\mez,H^{-\mez}}.
 \eq
Next, for any $|\alpha|=s$, we commute \eqref{lin:intro:0} with $\p^\alpha$ and use (ii), then integrate the resulting equation against $T_{\eta_w} \p^\alpha f$. This yields 
 \bq\label{stabest:high}
 \begin{aligned}
 \langle \p_t \p^\alpha f , T_{\eta_w} \p^\alpha f \rangle_{H^\mez,H^{-\mez}}&\le -c_{\Psi_0[\eta_w] }\| T_{\eta_w} \p^\alpha f\|_{H^{-\mez}}^2+ \langle\Psi_0[\eta_w][\p^\alpha, T_{\eta_w}]f, T_{\eta_w} \p^\alpha f \rangle_{H^\mez,H^{-\mez}}\\
 &\qquad+  \langle \p^\alpha S(f) , T_{\eta_w} \p^\alpha f \rangle_{H^\mez,H^{-\mez}}\\
 &\le -c_{\Psi_0[\eta_w] }\| T_{\eta_w} \p^\alpha f\|_{H^{-\mez}}^2+ C(\| \eta_w\|_{H^{s+\frac52}})\| f\|_{H^{s+\mez}} \| T_{\eta_w} \p^\alpha f \|_{H^{-\mez}}\\
 &\qquad+  \langle \p^\alpha S(f) , T_{\eta_w} \p^\alpha f \rangle_{H^\mez,H^{-\mez}}.
 \end{aligned}
 \eq
 The product $\| f\|_{H^{s+\mez}} \| T_{\eta_w} \p^\alpha f \|_{H^{-\mez}}$ can be treated by interpolation and Young's inequality, where the lower regularity part can be absorbed by a large multiple $A$ of the good term $c_{\Psi_0[\eta_w] }\| T_{\eta_w} f\|_{H^{-\mez}}^2$ in \eqref{stabest:low}.  Then invoking the fact that $T_{\eta_w}$ is a self-adjoint second-order elliptic operator of divergence form, we deduce 
 \bq\label{stabest:sum}
 \mez \frac{d}{dt}\left(A\langle  f , T_{\eta_w} f \rangle_{H^\mez,H^{-\mez}}+ \langle \p^\alpha f , T_{\eta_w} \p^\alpha f \rangle_{H^\mez,H^{-\mez}} \right)\le -\frac{c_{\Psi_0[\eta_w] }}{C_1} \| f\|_{H^{s+\tdm}}^2+C_1\| S(f)\|_{H^{s+\mez}}^2,
 \eq
 where $C_1=C_1(\| \eta_w\|_{H^{s+\frac52}})$. As for $S$, our contraction estimates for $\Psi_0$ give 
\bq\label{est:S}
\| S(f)\|_{H^{s+\mez}}\le C(\| \eta_w\|_{H^{s+\tdm}})\| \eta_w-\eta_*\|_{H^{s+\tdm}}\| f\|_{H^{s+\tdm}}
\eq
so long as $\| f\|_{H^{s+\tdm}}<1$. Combining \eqref{stabest:sum} and \eqref{est:S}, we deduce the exponential decay of  $H^{s+1}$ provided the smallness condition 
\bq\label{stabcd:intro}
\| \eta_w-\eta_*\|_{H^{s+\tdm}}< \frac{c_{\Psi_0[\eta_w] }}{C_2(\| \eta_w\|_{H^{s+\frac52}})}
\eq
for some nondecreasing function $C_2: \Rr_+\to \Rr_+$.

We note however that \eqref{stabcd:intro}  does not necessarily imply the stability condition \eqref{cd:stab} unless the coercive constant $c_{\Psi_0[\eta_w]}$ depends only on the $H^{s+\tdm}$ norm of $\eta_w$ and not the entire function $\eta_w$. Unfortunately, our proof of the coercive estimate  (i) is indirect and only yields the latter.  It is worth mentioning that for  Darcy flows \cite{BN}, $ -\Psi_\gamma[\eta_w] $ is replaced with the Dirichlet to Neumann operator $G[\eta_w]$ which satisfies a  {\it quantitative} coercive estimate \cite{Nguyen2023-coercivity}. Here, we overcome the lack of quantitative dependence of $c_{\Psi_0[\eta_w]}$ on $\eta_w$ by replacing $\Psi_0[\eta_w]T_{\eta_w}f$ in \eqref{lin:intro:0} with $\Psi_0[\eta_*]T_{\eta_*}f$ at the expense of the additional linear operator $Rf:= (\Psi_0[\eta_w]T_{\eta_w}-\Psi_0[\eta_*]T_{\eta_*})f$. Invoking again our contraction estimates for $\Psi_0$, we see that $R$ obeys the same bound \eqref{est:S} as $S$ does. On the other hand, the stability condition \eqref{stabcd:intro} is then modified to 
\bq
\| \eta_w-\eta_*\|_{H^{s+\tdm}}< \frac{c_{\Psi_0[\eta_*] }}{C_3(\| \eta_w\|_{H^{s+\frac32}})}
\eq 
which is of the desired form \eqref{cd:stab} since $\eta_*$ is fixed. 
\subsection{Organization of the paper} In  \cref{sec:Stokes}, we develop the theory of weak solutions and regularity for the $\gamma$-Stokes equations with prescribed normal stress or normal trace in non-flat domains. Utilizing the results obtained in  \cref{sec:Stokes}, we analyze  the linear  and nonlinear normal-tress to normal-Drirchlet operators in \cref{sec:Psi}. We prove the bijectivity  and continuity, contraction, coercive, and commutator estimates for these operators in Sobolev spaces. After establishing  the Sobolev solvability of the capillary-gravity operator and properties of its  the linearization,  we prove  the main \cref{theo1:intro} in \cref{sec:construction}. \cref{sec:stab} is devoted to the stability analysis of traveling wave solutions to the free boundary Stokes equations. We first prove  the well-posedness of the linear dynamics, then establish  the nonlinear asymptotic stability, proving  \cref{theo2:intro}. \cref{appendix:lifting} contains results on liftings and Lagrange multipliers needed for the weak solution theory in  \cref{sec:Stokes}. \cref{appendix:productestimates} gathers product rules in various domains and spaces. Finally, we record basic features of paradifferential calculus in \cref{appendix:para}.

%%%%%%%%%%%%%%%%%%%%%%%%%%%%%%%%%%%%%%%%%%%%%%%%%%%%%%%%%%%%%%%%%%%%%%%%%%%%%%%%%%%%%%%%%%%%%%%%%%%%%%%%%%%%

\section{$\gamma-$Stokes problems in non-flat domains}\label{sec:Stokes}
We  assume throughout this section  that $\eta \in L^{\infty}(\T^d)$ satisfies
\bq \label{eta lower bound}
\inf_{x\in \T^d}(\eta(x)+b)\ge c^0>0.
\eq
For any $s\geq 1$, we define the following Sobolev spaces
\bq\label{zero on the bottom space}
{}_0H^s(\Omega_\eta) := \{u\in H^s (\Omega_\eta): \quad u\vert_{\Sigma_{-b}} =0\},
\eq 
\bq\label{div free space}
{}_0H^s_\sigma(\Omega_\eta) := \{u\in {}_0H^s (\Omega_\eta; \Rr^{d+1}): \quad \dv u=0 \},
\eq
\bq
{}_0H^s_\cN(\Omega_\eta) := \left\{u\in {}_0H^s(\Omega_\eta; \Rr^{d+1}): \quad u(\cdot,\eta(\cdot))\cdot\cN(\cdot) =0 \text{ 
 on }\T^d\right\},
\eq
\bq
\mathring{H}^s(\T^d):=\left \{u\in H^s(\T^d): \quad \int_{\T^d} u =0 \right\}.
\eq
\begin{nota}
\begin{itemize}
\item $\na$ is a row vector. If $F$ is vector field, then $(\na F)_{ij}=\p_jF_i$.
\end{itemize}
\end{nota}

%%%%%%%%%%%%%%%%%%%%%%%%%%%%%%%%%%%%%%%%%%%%%%%%%%%%%%

\subsection{Weak solutions}
We first consider the following $\gamma-$Stokes system with prescribed normal stress condition:
\begin{equation}\label{sys:stokes with stress}
    \begin{cases}
    -\gamma \partial_1 v + \dv(pI - \Dd v) = f \qquad & \text{in }\Omega_\eta, \\
    \nabla\cdot v= g \qquad & \text{in }\Omega_\eta, \\
    (pI - \mathbb{D}v)(\cdot, \eta(\cdot))\mathcal{N}(\cdot) = k(\cdot) \qquad &\text{on }\T^d, \\
    v=0 \qquad & \text{on }\Sigma_{-b},
    \end{cases}
\end{equation}
where $f\in \bigl( {}_0H^1(\Omega_\eta)\bigr)^*$, $g\in L^2(\Omega_\eta)$, and $k\in (H^{-\mez}(\T^d))^{d+1}$.

A simple computation shows that a weak formulation for (\ref{sys:stokes with stress}) is to find a $v\in {}_0H^1(\Omega_\eta) $ and $p\in L^2(\Omega_\eta)$ such that $\dv v = g$ and 
\bq\label{weak stokes with stress}
\int_{\Omega_\eta } \mez \Dd v: \Dd u - p\dv u - \gamma \partial_1 v\cdot u = \langle f,u\rangle_{\Omega_\eta} - \langle k,u(\cdot,\eta(\cdot))\rangle_{\Sigma_\eta} 
\eq
for all $u\in {}_0H^1(\Omega_\eta) $, where $\langle \cdot, \cdot \rangle_{\Omega_\eta}$ denotes the pairing between ${}_0H^1(\Omega_\eta)$ and its dual $\bigl(  {}_0H^1(\Omega_\eta)\bigr)^*$, and $\langle \cdot, \cdot \rangle_{\Sigma_\eta}$ means the pairing between $H^{-\mez}(\T^d)$ and $H^{\mez}(\T^d)$.
\begin{prop}\label{weak sol for stokes with stress}
    Suppose $\eta\in W^{1,\infty}(\T^d)$. Then there exists a $\gamma^*= \gamma^*(d, b, c^0, \|\eta\|_{W^{1,\infty}(\T^d)})>0$ such that for any $f\in \bigl(  {}_0H^1(\Omega_\eta)\bigr)^*$, $g\in L^2(\Omega_\eta)$, and $k\in (H^{-\mez}(\T^d))^{d+1}$, provided that $|\gamma|\le \gamma^*$ there exists a unique pair of $v\in {}_0H^1(\Omega_\eta) $ and $p\in L^2(\Omega_\eta)$ such that $\dv v = g$ and (\ref{weak stokes with stress}) is satisfied. Moreover, we have
    \bq \label{base energy estimate for stokes with stress}
    \|v\|_{{}_0H^1(\Omega_\eta)} + \|p\|_{L^2(\Omega_\eta)} \leq C(\|\eta\|_{W^{1,\infty}(\T^d)}) \bigl( \|f\|_{\bigl( {}_0H^1(\Omega_\eta)\bigr)^*} + \|g\|_{L^2(\Omega_\eta)} + \|k\|_{H^{-\mez}(\T^d)}\bigr)
    \eq
    for some $C:\Rr^+ \to \Rr^+$ depending on $(d,b,c^0)$. 
\end{prop}
\begin{proof}
    First, consider the case where $g=0$. Let $B:{}_0H^1_\sigma(\Omega_\eta) \times {}_0H^1_\sigma(\Omega_\eta) \to \Rr$ be the bilinear form 
    \bq\label{bilinearB}
    B(v, u) = \int_{\Omega_\eta } \mez \Dd v: \Dd u - \gamma \partial_1 v\cdot u . 
    \eq
    Obviously, $B$ is a bounded operator and due to Korn's and Poincare's inequalities we have
$$
\begin{aligned}
\|u\|^2_{{}_0H^1(\Omega_\eta)}&\leq c_*(\|\eta\|_{W^{1,\infty}(\T^d)})\|\Dd u\|^2_{L^2(\Omega_\eta)},\\  
 \|u\|^2_{{}_0H^1(\Omega_\eta)}&\leq c_*(\|\eta\|_{W^{1,\infty}(\T^d)})\|\nabla u\|^2_{L^2(\Omega_\eta)},
\end{aligned}
$$
    where $c_*:\Rr^+ \to \Rr^+$ depends only on $(d,b,c^0)$. It follows that 
    \[
    B(u, u)\ge \frac{1}{2c_*}(1-|\gamma|2c_*)\| u\|_{{}_0H^1(\Omega_\eta)}^2,
    \]
    and hence 
    \bq\label{def:gamma*}
   B(u, u)\ge \frac{1}{4c_*}\| u\|_{{}_0H^1(\Omega_\eta)}^2\quad\text{if~} |\gamma|\le \gamma^*:= (4c_*(\|\eta\|_{W^{1,\infty}(\T^d)}))^{-1}.
    \eq
     Moreover, the mapping
    \[
 {}_0H^1_\sigma(\Omega_\eta) \ni     u\mapsto  \langle f,u\rangle_{\Omega_\eta} - \langle k,u(\cdot,\eta(\cdot))\rangle_{\Sigma_\eta} 
    \]
    is bounded by means of the trace theorem.  Therefore, by virtue of the Lax-Milgram theorem, there exists a unique $v\in {}_0H^1_\sigma(\Omega_\eta)  $ such that 
    $$B(v,u) =  \langle f,u\rangle_{\Omega_\eta} - \langle k,u(\cdot,\eta(\cdot))\rangle_{\Sigma_\eta}\quad \forall u\in {}_0H^1_\sigma(\Omega_\eta).$$
Then, the functional  $\Lambda\in \bigl(  {}_0H^1(\Omega_\eta)\bigr)^* $ defined by  
\[
 \langle \Lambda ,u\rangle_{\Omega_\eta}=B(v,\cdot) -\langle f,\cdot\rangle_{\Omega_\eta} + \langle k,\cdot\rangle_{\Sigma_\eta}
 \]
 vanishes on ${}_0H^1_\sigma(\Omega_\eta)$. By  Corollary \ref{coro: lagrange multiplier for pressure}, there exists a  unique $p\in L^2(\Omega_\eta)$ such that 
    $$\int_{\Omega_\eta} p \dv u = \langle \Lambda ,u\rangle_{\Omega_\eta}\quad\forall u\in {}_0H^1(\Omega_\eta);$$
   moreover, $p$ satisfies 
    $$\|p\|_{L^2(\Omega_\eta)} \leq c(\|\eta\|_{W^{1,\infty}(\T^d)}) \|\Lambda\|_{\bigl(  {}_0H^1(\Omega_\eta)\bigr)^*}.$$
  Thus  $(v,p)$ satisfy (\ref{weak stokes with stress}) and (\ref{base energy estimate for stokes with stress}) with $g=0$.

    For the inhomogeneous case $g\not\equiv 0$, using Proposition \ref{solving div problem} we find some $w\in {}_0H^1(\Omega_\eta)$ with $\dv w =g$ satisfying 
    \bq \label{bound for w}
    \|w\|_{{}_0H^1(\Omega_\eta)} \leq c(\|\eta\|_{W^{1,\infty}(\T^d)}) \|g\|_{L^2(\Omega_\eta)}.
    \eq
     From the  homogeneous case, there exists a unique pair of  $v_0\in {}_0H^1_\sigma(\Omega_\eta) $ and $p\in L^2(\Omega_\eta)$ satisfying (\ref{weak stokes with stress}) with $f$ replaced with  $\bar{f}:= f-B(w,\cdot)$. Consequently $v=v_0+w$ satisfies (\ref{weak stokes with stress}) and the bound (\ref{base energy estimate for stokes with stress}) follows from (\ref{bound for w}).  

\end{proof}
Next, we consider the $\gamma-$Stokes system with Navier boundary condition:
\begin{equation}\label{sys:stokes with Dirichlet}
    \begin{cases}
    -\gamma \partial_1 v + \dv(pI - \Dd v) = f \qquad &\text{in } \Omega_\eta, \\
    \nabla\cdot v= g \qquad & \text{in }\Omega_\eta, \\
    \cN^\perp(pI - \mathbb{D}v)(\cdot, \eta(\cdot))\mathcal{N}(\cdot) = l(\cdot) \qquad &\text{on }\T^d, \\
    v(\cdot,\eta(\cdot))\cdot \cN(\cdot) = h(\cdot) \qquad &\text{on }\T^d, \\
    v=0 \qquad &\text{on } \Sigma_{-b},
    \end{cases}
\end{equation}
where $f\in \bigl( {}_0H^1(\Omega_\eta)\bigr)^*$, $g\in L^2(\Omega_\eta)$, $l\in H^{-\mez}(\T^d)$, $h\in H^{\mez}(\T^d)$, and $\cN^\perp:= I - \langle \cdot, \cN \rangle\frac{\cN}{|\cN|^2}$ is the projection on the hyperplane orthogonal to $\cN$. 
Moreover, assume that  $l$ and $h$ satisfy the compatibility conditions
\begin{align}\label{compatibility for l}
l\cdot \cN =0 \qquad \text{on } \T^d,\\
\label{compatibility for h and g}
\int_{\T^d} h=\int_{\Omega_\eta} g.
\end{align}
 Multiplying  \eqref{sys:stokes with Dirichlet}$_1$ by an arbitrary $u\in{}_0H^1_\cN(\Omega_\eta) $ and formally integrating by parts, we obtain
$$\int_{\Omega_\eta } \mez \Dd v: \Dd u - p\dv u - \gamma \partial_1 v\cdot u = \langle f,u\rangle_{\Omega_\eta} - \langle l+ \cN^T(pI - \Dd v)\cN \frac{\cN}{|\cN|^2},u(\cdot,\eta(\cdot))\rangle_{\Sigma_\eta}.$$
Since $u\cdot \cN =0$,  $(v,p)$ formally satisfy 
\bq \label{weak stokes with Dirichlet}
\int_{\Omega_\eta } \mez \Dd v: \Dd u - p\dv u - \gamma \partial_1 v\cdot u = \langle f,u\rangle_{\Omega_\eta} - \langle l,u(\cdot,\eta(\cdot))\rangle_{\Sigma_\eta}\quad\forall u\in{}_0H^1_\cN(\Omega_\eta).
\eq

Therefore, we define $v\in {}_0H^1(\Omega_\eta)$ and $p\in L^2(\Omega_\eta)$ to be a weak solutions to (\ref{sys:stokes with Dirichlet}) if $v(\cdot,\eta(\cdot))\cdot \cN(\cdot) = h(\cdot) $, $\dv v = g$, and \eqref{weak stokes with Dirichlet} is satisfied. 

\begin{prop}\label{weak sol for stokes with Dirichlet}
    Suppose $\eta \in W^{2,\infty}(\T^d)$. Let $f\in \bigl(  {}_0H^1(\Omega_\eta)\bigr)^*$, $g\in L^2(\Omega_\eta)$, $l\in H^{-\mez}(\T^d)$, and $h\in H^{\mez}(\T^d)$, where $l$, $g$, and $h$ satisfy (\ref{compatibility for l}) and (\ref{compatibility for h and g}). If  $|\gamma|\le \gamma^*$, given by \eqref{def:gamma*}, then there exists a unique pair of $v\in {}_0H^1(\Omega_\eta) $ and $p\in \mathring{L}^2(\Omega_\eta)$ with $\dv v = g$, and $v(\cdot,\eta(\cdot))\cdot \cN(\cdot) = h(\cdot)$, and satisfying (\ref{weak stokes with Dirichlet}). Moreover, we have
    \bq \label{base energy estimate for stokes with Dirichlet}
    \|v\|_{{}_0H^1(\Omega_\eta)} + \|p\|_{L^2(\Omega_\eta)} \leq C(\|\eta\|_{W^{2,\infty}(\T^d)}) \bigl( \|f\|_{\bigl(  {}_0H^1(\Omega_\eta)\bigr)^*} + \|g\|_{L^2(\Omega_\eta)} + \|l\|_{H^{-\mez}(\T^d)}+ \|h\|_{H^{\mez}(\T^d)}\bigr)
    \eq
    for some $C:\Rr^+ \to \Rr^+$ depending on $(d,b,c^0)$.
\end{prop}
\begin{proof}
    By the virtue of \cref{lifting neumann}, there exists $v_h \in {}_0H^1(\Omega_\eta)$ such that $v_h(\cdot,\eta(\cdot))\cdot\cN = h$, and 
    $$\|v_h\|_{H^1(\Omega_\eta)} \leq c(\|\eta\|_{W^{2,\infty}(\T^d)})\|h\|_{H^{\mez}(\T^d)}.$$
    Let $g_h = \dv v_h$. Then the divergence theorem implies $\int_{\T^d} h = \int_{\Omega_\eta} g_h$, 
    which together with the compatibility condition  \eqref{compatibility for h and g} implies $g-g_h \in \mathring{L}^2(\Omega_\eta)$.
  By \cref{lifting div in HcN}, there exists $v_g\in {}_0H^1_\cN(\Omega_\eta)$ with $\dv v_g = g-g_h $ and
    $$\|v_g\|_{H^1(\Omega_\eta)} \leq c(\|\eta\|_{W^{2,\infty}(\T^d)})\bigl(\|h\|_{H^{\mez}(\T^d)}+ \|g\|_{L^2(\Omega_\eta)}\bigr).$$
        
    Let $B: {}_0H^1(\Omega_\eta) \times {}_0H^1(\Omega_\eta) \to \Rr$ be the bilinear form 
    $$B(v,u) = \int_{\Omega_\eta} \mez \Dd v: \Dd u - \gamma \partial_1 v \cdot u$$
which is continuous and coercive by  the proof of \cref{weak sol for stokes with stress}. The Lax-Milgram theorem implies that  there exists a unique $v_0\in {}_0H^1_\cN(\Omega_\eta)\cap {}_0H^1_\sigma(\Omega_\eta) $ such that
    $$B(v_0,u) = \langle f,u\rangle_{\Omega_\eta} - \langle l,u(\cdot,\eta(\cdot))\rangle_{\Sigma_\eta} - B(v_h+v_g,u)\quad\forall u\in {}_0H^1_\cN(\Omega_\eta)\cap {}_0H^1_\sigma(\Omega_\eta).$$
    Then, the functional $\Lambda\in ({}_0H^1_\cN(\Omega_\eta))^*$ defined by
    $$\langle \Lambda, u \rangle_{\Omega_\eta}  = B(v_0+v_h+v_g,u) -  \langle f,u\rangle_{\Omega_\eta} + \langle l,u(\cdot,\eta(\cdot))\rangle_{\Sigma_\eta}$$
    vanishes on ${}_0H^1_\cN(\Omega_\eta)\cap {}_0H^1_\sigma(\Omega_\eta)$. By \cref{Lagrange for pressure in HcN}, there exists a unique $p\in \mathring{L}^2(\Omega_\eta)$ such that
    $$\int_{\Omega_\eta} p\dv u = \langle \Lambda, u \rangle_{\Omega_\eta} \quad \forall u\in {}_0H^1_\cN(\Omega_\eta); $$
    moreover, $p$ satisfies
    \begin{multline*}
        \|p\|_{L^2(\Omega_\eta)} \leq c(\|\eta\|_{W^{2,\infty}(\T^d)}) \|\Lambda\|_{\bigl({}_0H^1_\cN(\Omega_\eta)\bigr)^*} \\
        \leq c(\|\eta\|_{W^{2,\infty}(\T^d)})\bigl( \|f\|_{\bigl(  {}_0H^1(\Omega_\eta)\bigr)^*} + \|g\|_{L^2(\Omega_\eta)} + \|l\|_{H^{-\mez}(\T^d)}+ \|h\|_{H^{\mez}(\T^d)}  \bigr)  .
    \end{multline*}
    This shows that $v=v_0+v_h+v_g$ and $p$ are the desired solution pair.
    
\end{proof}
\begin{rema}
    Since the system (\ref{sys:stokes with Dirichlet}) is invariant under $p\mapsto p+c$, the uniqueness statement in \cref{weak sol for stokes with Dirichlet} holds for  $p\in \mathring{L}^2(\Omega_\eta)$. 
\end{rema}

%%%%%%%%%%%%%%%%%%%%%%%%%%%%%%%%%%%%%%%%%%%%%%%%%%%%%%

\subsection{Regularity of weak solutions}
Next, we prove regularity results for weak solutions to the Stokes systems \eqref{sys:stokes with stress} and \eqref{sys:stokes with Dirichlet}. To this end we first  flatten the fluid domain via a Lipschitz diffeomorphism. 
Let $\Omega:=\T^d\times(-b,0)\ni (x, z)$, $x=(x_1,...,x_d)\in \T^d$, and $z\in (-b,0)$.
Define the map $\cF_\eta:\Omega\to \Omega_\eta$ by
\bq \label{cF definitoin}
\cF_\eta(x,z) = (x, \varrho(x,z)),
\eq
where
\bq\label{def:varrho}
\varrho(x,z) = \frac{b+z}{b}e^{\delta z|D|}\eta(x)+z
\eq
and $\delta>0$ is a small constant to be chosen later.

With the convention of $(\nabla F)_{ij}= \partial_j F_i$, we have
\bq
        \nabla\cF_{\eta}(x,z) = \begin{bmatrix}
            I_{d\times d}& 0_{d\times 1} \\
            \frac{b+z}{b}e^{\delta z|D|}\nabla \eta(x) & \frac{1}{b}e^{\delta z|D|}\eta(x)+ \frac{b+z}{b}\delta|D|e^{\delta z|D|}\eta(x) +1
        \end{bmatrix}
    \eq
    \bq
    J(x,z) := \det(\nabla_{x,z} \cF_\eta) = \partial_z \varrho(x,z) = \frac{1}{b}e^{\delta z|D|} \eta(x) +   \frac{b+z}{b}\delta|D|e^{\delta z|D|}\eta(x) + 1.
    \eq
For any $f: \Omega_\eta \to \Rr^m$ and $\Tilde{f} = f\circ \cF_\eta$, we have
\bq\label{change:nabla}
(\nabla_{x,y} f)\circ \cF_\eta(x, z) = \nabla_{x,z} \Tilde{f}(x, z)A(x,z),
\eq
where 
\bq\label{definition of A}
        A(x,z):= (\nabla \cF_\eta(x,z))^{-1} = \begin{bmatrix}
            I_{d\times d} & 0_{d\times 1} \\
            -\frac{1}{J(x,z)}\frac{b+z}{b}e^{\delta z|D|}\nabla \eta(x) & \frac{1}{J(x,z)}
        \end{bmatrix}.
\eq
Moreover, if $v: \Omega_\eta\to \Rr^{d+1}$ and $\tilde{v}=v\circ \cF_\eta$, then 
\bq\label{change:div}
(\na_{x, y}\cdot v)\circ \cF_\eta(x, z)=(\na_{x, z} \tilde{v}):A^T.
\eq
\begin{lemm} \label{lower bound for J}
    Suppose   $\eta \in H^{s}(\T^d)$ with $s>1+d/2$. There exists $\delta^*=\delta^*(\|\eta\|_{H^{s}(\T^d)})>0$ such that if $\delta< \delta^*$, then $\p_z\varrho(x,z)\geq M=M(b,c^0)>0$ for all $(x,z)\in \Omega$. In particular, $\cF_\eta$ is a diffeomorphism.
\end{lemm}
\begin{proof}
We have
 \begin{align*}
 \p_z\varrho(z, z)&= 1+ \frac{1}{b}e^{\delta z|D|}\eta(x)+\delta \frac{b+z}{b}e^{\delta z|D|}|D|\eta(x)\\
 &= \frac{\eta+b}{b}+  \frac{1}{b}(e^{\delta z|D|}\eta(x)-\eta(x))+\delta \frac{b+z}{b}e^{\delta z|D|}|D|\eta(x),
 \end{align*}
 where $\eta+b\ge c_0$ by \eqref{eta lower bound}. Clearly 
 \[
\left| \frac{b+z}{b}e^{\delta z|D|}|D|\eta(x)\right|\le c_d |\widehat{|D|\eta}\|_{L^1}.
 \]
In addition,  using the mean value theorem, we deduce 
 \[ 
 \left|e^{\delta z|D|}\eta(x)-\eta(x)\right|=\left |\delta \int_0^z  e^{\delta r|D|}|D|\eta(x) dr\right|\le \delta bc_d |\widehat{|D|\eta}\|_{L^1}.
 \]
 Then, invoking the embedding $|\widehat{|D|\eta}\|_{L^1}\le c(d, s) \| \eta\|_{H^s}$ for $s>1+\frac{d}{2}$, we obtain 
 \[
 \p_z\varrho(x, z)\ge \frac{c_0}{b}-\delta c(d, s)\| \eta\|_{H^s}.
 \]
 It follows that $\p_z\varrho\ge M:= \frac{c_0}{2b}>0$ if $\delta\le \delta_*:=  \frac{c_0}{2bc(d, s)\| \eta\|_{H^s}}$.  This in turn implies that $\varrho(x,\cdot)$ is increasing in $z$ for any fixed $x\in \T^d$, so that $\cF_\eta$ is a bijection. On the other hand, since $J(x,z) = \partial_z \varrho\geq M>0$, $\cF_\eta$ is a local diffeomorphism, hence it must be a global diffeomorphism.   
\end{proof}
Combining the  standard smoothing effect of the kernel $e^{z|D|}$ ($z\le 0$) and linear interpolation, we have 
\begin{lemm}\label{diffeo regularity}
   Suppose  $\eta\in H^{s}(\T^d)$ with $s\ge -\mez$. Then $\cF_\eta \in H^{s+\mez}(\Omega)$ and $J\in H^{s-\mez}(\Omega)$. Moreover, we have the bounds
    \bq\label{regularity bound for cF}
    \|\cF_\eta\|_{H^{s+\mez}(\Omega)} \leq C(\|\eta\|_{H^s}+1),
    \eq
    \bq\label{regularity bound for J}
    \|J\|_{H^{s-\mez}(\Omega)} \leq C(\|\eta\|_{H^{s}}+1),
    \eq
    where the positive constant $C$ depends only on $(d,b,s)$.    
\end{lemm}

\begin{lemm}\label{regularity for A}
     Suppose $\eta\in H^{s+\tdm}(\T^d)$ with $s>(d+1)/2$. Recall that $A$ is defined in \eqref{definition of A}.
     
    (i) We have  $A\in H^{s+1}(\Omega)$, with 
    \bq\label{regularity bound for A}
    \|A\|_{H^{s+1}(\Omega)}\leq C(\|\eta\|_{H^{s+\tdm}(\T^d)}),
    \eq
    where $C:\Rr^+\to \Rr^+$  depends only on $(d,b,c^0,s)$.
    
    (ii)  Define $B:\Omega\to\Rr$ by
    \bq\label{B definition}
    B(x,z) :=  \sum_{k=1}^{d+1} A^2_{(d+1)k}(x,z).
    \eq
For all  $l \in \Nn$ and $\sigma \in [0, s+1]$, if $F\in H^\sigma(\Omega)$ then
    \bq\label{product estimate for B}
        \|B^{-l}F\|_{H^{\sigma}(\Omega)}\leq C(\|\eta\|_{H^{s+\tdm}(\T^d)})\|F\|_{H^{\sigma(\Omega)}},
    \eq
    where $C:\Rr^+\to \Rr^+$ depends only on $(d,b,c^0,s,\sigma,l)$.
\end{lemm}
\begin{proof}
  By Lemmas \ref{lower bound for J} and  \ref{diffeo regularity}, we have $J\in H^{s+1}(\Omega)$ and  $\inf J \geq M(b, c^0) > 0$.    Moreover,  combing \eqref{definition of A} and \eqref{regularity bound for J} yields  
      $$\inf_{(x,z)\in \Omega} B \geq \inf_{(x,z)\in \Omega} \frac{1}{J^2} \geq \|J\|_{L^\infty}^{-2}\geq C\|J\|_{H^{s+1}}^{-2} \geq C(\|\eta\|_{H^{s+\tdm}}+1)^{-2}>0.$$
    Therefore, both (i) and  (ii) follow from  Lemma \ref{product estimate for inverse}.
\end{proof}
\begin{lemm}\label{composition regularity}
    Suppose $\eta\in H^{s+\tdm}(\T^d)$ with $(d+1)/2<s\in \Nn$.  Then for all  $\sigma \in [0, s+2]$, if $F\in H^\sigma(\Omega_\eta)$ then
    \bq \label{flat to non-flat bound}
        \|F\circ \cF_\eta\|_{H^{\sigma}(\Omega)} \leq C(\|\eta\|_{H^{s+\tdm}(\T^d)})\|F\|_{H^{\sigma}(\Omega_\eta)},
    \eq
    \bq \label{non-flat to flat bound}
        \|F\|_{H^{\sigma}(\Omega_\eta)} \leq C(\|\eta\|_{H^{s+\tdm}(\T^d)}) \|F\circ \cF_\eta\|_{H^{\sigma}(\Omega)},
    \eq
    where $C:\Rr^+\to \Rr^+$ depends only on $(d,b,c^0,s,\sigma)$. 
\end{lemm}
\begin{proof}
    Since  $\sigma\le s+2\in \Nn$, it suffices to prove \eqref{flat to non-flat bound} and \eqref{non-flat to flat bound} for $\sigma \in \Nn$, the non-integer case following by linear interpolation. We first prove  \eqref{flat to non-flat bound} by induction on $\sigma$. The case  $\sigma=0$ follows easily from \cref{lower bound for J}:
    $$\|F\circ \cF_\eta\|_{L^2(\Omega)}^2 = \int_{\Omega_\eta} F(x,y)^2 \frac{1}{J} \leq \frac{2b}{c^0}\|F\|^2_{L^2(\Omega_\eta)}.$$
Assume that \eqref{flat to non-flat bound} holds for some $0\leq\sigma\leq s+1$ and let  $F\in H^{\sigma+1}(\Omega_\eta) $. Using Lemma \ref{product estimate for domain}, the induction hypothesis, and \cref{diffeo regularity}, we have
    \begin{align*}
        \|\partial_i (F\circ \cF_\eta)\|_{H^{\sigma}(\Omega)}& \leq C(\| \eta\|_{W^{1, \infty}(\T^d)}) \sum_{j=1}^{d+1} \|\partial_jF\circ \cF_\eta\|_{H^{\sigma}(\Omega)} \|\partial_i \cF_\eta^j\|_{H^{s+1}(\Omega)} \\
       & \leq C(\| \eta\|_{W^{1, \infty}(\T^d)})C(\etanorm) \sum_{j=1}^{d+1} \|\partial_jF\|_{H^{\sigma}(\Omega_\eta)} \|\cF_\eta^j\|_{H^{s+2}(\Omega)} \\
        & \leq C(\etanorm)\|F\|_{H^{\sigma+1}(\Omega_\eta)}    
    \end{align*}
    which completes the proof of \eqref{flat to non-flat bound}. \eqref{non-flat to flat bound} can be proven analogously  using the estimate \eqref{regularity bound for A} for $\| (\na \cF_\eta)^{-1}\|_{H^{s+1}(\Omega)}$. 
\end{proof}
\begin{rema}
\cref{composition regularity} is valid for all real numbers  $s>(d+1)/2$. See Propositions 2.9 and 2.10 in \cite{LatoccaNguyen}. 
\end{rema}
Now we are ready to prove regularity for weak solutions to the $\gamma$-Stokes system \eqref{sys:stokes with stress}.   If a function $w$ is defined on $\Omega_\eta$, we denote $\Tilde{w}(x, z)=w(\cF_\eta(x, z))$.  
Formally, we may use \eqref{change:nabla} and \eqref{change:div} to rewrite \eqref{sys:stokes with stress} in the flattened domain as
\bq \label{sys: flattened stokes with stress}
     \begin{cases}
        -\nabla_{x,z}\left(\nabla_{x,z}\Tilde{v}_i A+ A^T(\partial_i\Tilde{v})^T \right):A^T + \nabla_{x,z}\Tilde{p}A\cdot e_i - \gamma\nabla_{x,z}\Tilde{v}_iA\cdot e_1 = \Tilde{f}_i & \quad \text{in } \Omega, \\
         \nabla_{x,z}\Tilde{v} : A^T = \Tilde{g} & \quad \text{in } \Omega, \\
         \left[\Tilde{p}I-\left(\nabla_{x,z}\Tilde{v}A+A^T(\nabla_{x,z}\Tilde{v})^T\right)\right] \cN = k& \quad \text{on } \T^{d}, \\ %write down the var
         \Tilde{v}=0& \quad \text{on } \Sigma_{-b},
     \end{cases}
     \eq
where the first equation is written for each component. We note that $(v, p)\in {}_0H^1(\Omega_\eta)\cap L^2(\Omega_\eta)$ is a weak solution of \eqref{sys:stokes with stress} if and only if $(\tilde{v}, \tilde{p})\in {}_0H^1(\Omega)\cap L^2(\Omega)$ is a weak solution of \eqref{sys: flattened stokes with stress}.

 In \cref{reg for stokes with stress} below,  we establish a Sobolev regularity result assuming Sobolev regularity of the boundary function $\eta$, namely, $\eta \in H^{s+\tdm}(\T^d)$ with $s>(d+1)/2$. This condition is weaker than the classical condition $\eta \in C^{s+1, 1}(\T^d)$ (see Theorem IV.$7.4$ in \cite{BoyerFabrie} and \cref{remark: classic regularity} below). 
\begin{theo}\label{reg for stokes with stress}
    Let $|\gamma|\leq \gamma^*$ (see \eqref{def:gamma*}), $(d+1)/2<s\in \Nn$, and $ \sigma \in [0, s+1]\cap \Nn$. Suppose $\eta\in H^{s+\tdm}(\T^d)$, $f\in H^{\sigma-1}(\Omega_\eta) \cap ({}_0H^1(\Omega_\eta))^*$, $g\in H^{\sigma}(\Omega_\eta)$ and $k\in H^{\sigma-\mez}(\T^d)$. Then the  weak solution of \eqref{sys:stokes with stress} satisfies 
    \bq\label{high reg estimate for stokes with stress}
    \|v\|_{H^{\sigma+1}(\Omega_\eta)} + \|p\|_{H^{\sigma}(\Omega_\eta)} \leq C(\|\eta\|_{H^{s+\tdm}(\T^d)}) \bigl(\|f\|_{H^{\sigma-1}(\Omega_\eta)\cap ({}_0H^1(\Omega_\eta))^*} + \|g\|_{H^{\sigma}(\Omega_\eta)} + \|k\|_{H^{\sigma-\mez}(\T^d)}\bigr)
    \eq
    for some $C:\Rr^+ \to \Rr^+$ depending only on $(d,b,c^0,s,\sigma)$. %Moreover, if $f=0$, then the preceding assertion holds for $0\le \sigma \le s+1$.
\end{theo}
\begin{proof} 
   For $s>(d+1)/2$, we have $\eta \in H^{s+\tdm}(\T^d)\hookrightarrow W^{2, \infty}(\T^d)$. By virtue of \cref{weak sol for stokes with stress},  there exists a unique weak  solution $(v,p)\in H^1(\Omega_\eta)\times L^2(\Omega_\eta)$ satisfying 
 \bq \label{variest:vp}
    \|v\|_{{}_0H^1(\Omega_\eta)} + \|p\|_{L^2(\Omega_\eta)} \leq C(\|\eta\|_{W^{1,\infty}(\T^d)}) \bigl( \|f\|_{\bigl( {}_0H^1(\Omega_\eta)\bigr)^*} + \|g\|_{L^2(\Omega_\eta)} + \|k\|_{H^{-\mez}(\T^d)}\bigr).
    \eq
    Since $\| f\|_{H^{-1}}\le \| f\|_{({}_0H^1)'}$, this yields  \eqref{high reg estimate for stokes with stress} for $\sigma=0$. For $\sigma \in [1, s+1]$, the norm of $f$ in \eqref{high reg estimate for stokes with stress} reduces to the $H^{\sigma-1}$ norm.  We will prove \eqref{high reg estimate for stokes with stress} for $\sigma \in [1, s+1]$ by induction. 
    
  \underline{ Step 1.} We first consider the case $\sigma=1$. Since $f\in L^2(\Omega_\eta)$, $g\in H^1(\Omega_\eta)$, and $k\in H^{\mez}(\T^d)$, the pairings in the weak formulation (\ref{weak stokes with stress}) are  $L^2$ inner products  of  functions. Thus, we can make the change of variables $(x, y)=\cF_\eta(x, z)$ to write the weak formulation in the flattened domain as
     \begin{multline}\label{weak stokes with stress in flattened domain}
             \int_\Omega \Bigl[ \mez \left(\nabla_{x,z}\tilde{v}A + A^T(\nabla_{x,z}\tilde{v})^T\right) : \left(\nabla_{x,z}\tilde{u}A + A^T(\nabla_{x,z}\tilde{u})^T\right) \\
             \quad- \tilde{p}(\nabla_{x,z}\tilde{u}: A^T) - \gamma(\nabla_{x,z}\tilde{v}Ae_1)\cdot \tilde{u}   \Bigr]J(x,z) dxdz = \\ \int_\Omega (\tilde{f} \cdot\tilde{u}) J(x,z) dxdz - \int_{\T^d} k(x)\cdot \tilde{u}(x,0) dx     
     \end{multline}
    for all $\tilde{u}\in {}_0H^1(\Omega)$, together with the divergence condition $\nabla_{x,z}\tilde{v}:A^T=\tilde{g}$. 
    
    Let $1\leq i\leq d$, and $h\in \Rr\setminus\{0\}$. Then for any $w:\Omega\to \Rr^m$, define $w_i^h(x, z) = w(x+he_i, z)$ and $\delta_i^hw(x, z) = \frac{w_i^h(x,z) - w(x,z)}{h}$. We have the product rule 
    \bq\label{finitequotientproduct}
    \delta^h_i(f_1f_2\dots f_n)=\delta^h_if_1f_2\dots f_n+f_{1i}^h\delta^h_if_2\dots f_n+\dots +\big(\Pi_{j=1}^{n-1}f_{ji}^h\big)\delta^h_if_n
    \eq
    for difference quotients. 
    
Now, in \eqref{weak stokes with stress in flattened domain} we replace $\tilde{u}\in {}_0H^1(\Omega)$ with $\tilde{u}^{-h}_i\in {}_0H^1(\Omega)$  and make change of  variables $(x, z)\mapsto (x+he_i, z)$. After that, we take the difference of the resulting equation and \eqref{weak stokes with stress in flattened domain}, and invoke the formula \eqref{finitequotientproduct}. We obtain 
    \begin{multline}\label{quotient diff integral eq}
         \int_\Omega \Bigl[ \mez\big(\nabla_{x,z}\delta_i^h\tilde{v}A + A^T(\nabla_{x,z}\delta_i^h\tilde{v})^T\big) : \big(\nabla_{x,z}\tilde{u}A + A^T(\nabla_{x,z}\tilde{u})^T\big) - \delta_i^h\tilde{p}(\nabla_{x,z}\tilde{u}: A^T) - \gamma(\nabla_{x,z}\delta_i^h\tilde{v}Ae_1)\cdot \tilde{u} \Bigr]J  \\
         + \Bigl[ \mez \big(\nabla_{x,z}\tilde{v}_i^h\delta_i^hA + \delta_i^hA^T(\nabla_{x,z}\tilde{v}_i^h)^T\big) : \big(\nabla_{x,z}\tilde{u}A + A^T(\nabla_{x,z}\tilde{u})^T\big) - \tilde{p}_i^h(\nabla_{x,z}\tilde{u}: \delta_i^hA^T) - \gamma(\nabla_{x,z}\tilde{v}_i^h\delta_i^hAe_1)\cdot \tilde{u} \Bigr]J   \\
         +  \Bigl[ \mez\big(\nabla_{x,z}\tilde{v}_i^h A_i^h + (A_i^h)^T(\nabla_{x,z}\tilde{v}_i^h)^T\big) : \big(\nabla_{x,z}\tilde{u}\delta_i^h A + \delta_i^hA^T(\nabla_{x,z}\tilde{u})^T\big)   \Bigr]Jdxdz \\
         +  \int_\Omega \Bigl[ \mez\big(\nabla_{x,z}\tilde{v}_i^h A_i^h + (A_i^h)^T(\nabla_{x,z}\tilde{v}_i^h)^T\big): \big(\nabla_{x,z}\tilde{u} A_i^h + (A_i^h)^T(\nabla_{x,z}\tilde{u})^T\big)\\ - \tilde{p}_i^h(\nabla_{x,z}\tilde{u}: (A_i^h)^T) - \gamma(\nabla_{x,z}\tilde{v}_i^h A_i^he_1)\cdot \tilde{u}   \Bigr]\delta_i^hJ dxdz \\
         = \int_\Omega (\delta_i^h\tilde{f} \cdot\tilde{u}) Jdxdz +  \int_\Omega (\tilde{f}_i^h \cdot\tilde{u}) \delta_i^hJ dxdz - \int_{\T^d} \delta_i^h k(x)\cdot \tilde{u}(x,0) dx
    \end{multline}
    for all $\Tilde{u}\in {}_0H^1(\Omega)$. 
    
{\it Bound for $\|\delta_i^h \tilde{p}\|_{L^2(\Omega)}$}. The idea is to prove the existence of a test function $\tilde{u}\in {}_0H^1(\Omega)$ satisfying $\nabla_{x,z}\tilde{u}: A^T=\delta^h_i \tilde{p}$,  so that the main pressure term $\delta^h_i \tilde{p}(\nabla_{x,z}\tilde{u}: A^T)$ on the first line of \eqref{quotient diff integral eq} becomes $|\delta^h_i \tilde{p}|^2$. To this end, we note that since  $q:=(\delta_i^h \tilde{p}) \circ \cF_\eta^{-1} \in L^2(\Omega_\eta) $, by \cref{solving div problem}, there exists $w\in {}_0H^1(\Omega_\eta)$ such that $\dv w = q$. Then $\tilde{w}:= w\circ \cF_\eta \in {}_0H^1(\Omega)$ satisfies 
    \[
    (\nabla_{x,z}\tilde{w}:A^T) = (\dv w)\circ\cF_\eta=  \delta_i^h \tilde{p}
    \]
    as desired.  Moreover, combining \eqref{div problem estimate} and  \cref{composition regularity} yields
    \bq\label{w tilde bound}
    \|\tilde{w}\|_{H^1(\Omega)}\leq C(\|\eta\|_{H^{s+\tdm}(\T^d)}) \|\delta_i^h\tilde{p}\|_{L^2(\Omega)},
    \eq
    where $C$ depends only on $(d,b,c^0)$.  Then we choose $\tilde{u}=\tilde{w}$  in (\ref{quotient diff integral eq}) and successively apply  H\"older's inequality. The highest norm  needed for $A$ and $J$ is    $W^{1, \infty}$  which is finite in view of \eqref{regularity bound for A}, \eqref{regularity bound for J},  and the fact that $s>(d+1)/2$. We obtain
    \[
    \|\delta_i^h\tilde{p}\|^2_{L^2(\Omega)} \leq C(\etanorm)\|\tilde{w}\|_{H^1(\Omega)}\Bigl( \|\delta_i^h\tilde{v}\|_{H^1(\Omega)} + \|\tilde{p}\|_{L^2(\Omega)} + \|\tilde{v}\|_{H^1(\Omega)} + \|\tilde{f}\|_{L^2(\Omega)} + \|k\|_{H^{\mez}(\T^d)}
 \Bigr).
    \]
Invoking  the base estimate \eqref{variest:vp},   \cref{composition regularity}, and \eqref{w tilde bound}, we deduce
    \bq \label{delta p tilde bound}
    \|\delta_i^h\tilde{p}\|_{L^2(\Omega)} \leq C(\etanorm)\Bigl( \|\delta_i^h\tilde{v}\|_{H^1(\Omega)} + \|\tilde{f}\|_{L^2(\Omega)} + \|\tilde{g}\|_{H^1(\Omega)} + \|k\|_{H^{\mez}(\T^d)}
 \Bigr).
    \eq
  {\it Bound for $\|\delta_i^h  \tilde{v}\|_{H^1}$}.   We choose  $\tilde{u}=\delta_i^h \tilde{v}$ in \eqref{quotient diff integral eq}. Arguing as before, we have
    \begin{multline}\label{estv:100}
        |\int_\Omega \Bigl[\mez|(\nabla_{x,z}\delta_i^h\tilde{v}A + A^T(\nabla_{x,z}\delta_i^h\tilde{v})^T|^2 - \gamma(\nabla_{x,z}\delta_i^h\tilde{v}Ae_1)\cdot \delta_i^h\tilde{v} \Bigr]Jdxdz|  \\
        \leq |\int_\Omega \delta_i^h\tilde{p}(\nabla_{x,z}\delta_i^h\tilde{v}: A^T)Jdxdz| + C(\etanorm) \|\delta_i^h \tilde{v}\|_{H^1(\Omega)}\Bigl( \|\tilde{f}\|_{L^2(\Omega)} + \|\tilde{g}\|_{H^1(\Omega)} + \|k\|_{H^{\mez}(\T^d)}
 \Bigr).
    \end{multline}
   In addition,  the divergence condition gives $\nabla_{x,z}\delta_i^h\tilde{v}: A^T = \delta_i^h \tilde{g} - (\nabla_{x,z}\tilde{v}_i^h:\delta_i^hA^T)$, whence 
   \[
   \| \nabla_{x,z}\delta_i^h\tilde{v}: A^T\|_{L^2(\Omega)}\le \| \tilde{g}\|_{H^1(\Omega)}+C(\etanorm)\| v\|_{H^1(\Omega)}.
   \]
   Hence, using Holder's  inequality and  the estimates  \eqref{variest:vp} and \eqref{delta p tilde bound}, we obtain
    \begin{multline}\label{estv:101}
        |\int_\Omega \delta_i^h\tilde{p}(\nabla_{x,z}\delta_i^h\tilde{v}: A^T)Jdxdz| \\ 
        \leq C(\etanorm)\Bigl( \|\delta_i^h\tilde{v}\|_{H^1(\Omega)} + \|\tilde{f}\|_{L^2(\Omega)} + \|\tilde{g}\|_{H^1(\Omega)} + \|k\|_{H^{\mez}(\T^d)}
 \Bigr)\Bigl( \|\tilde{g}\|_{H^1(\Omega)}+ \| f\|_{({}_0H^1(\Omega_\eta))^*}  + \|k\|_{H^{-\mez}(\T^d)}
 \Bigr).
    \end{multline}
    On the other hand, if we set $w=\delta_i^h\tilde{v}\circ \cF_\eta^{-1}$, then 
    \begin{multline}\label{estv:102}
        \int_\Omega \Bigl[\mez|(\nabla_{x,z}\delta_i^h\tilde{v}A + A^T(\nabla_{x,z}\delta_i^h\tilde{v})^T|^2 - \gamma(\nabla_{x,z}\delta_i^h\tilde{v}Ae_1)\cdot \delta_i^h\tilde{v} \Bigr]Jdxdz = \int_{\Omega_\eta}\mez |\Dd w|^2 - \gamma \partial_1 w \cdot w \\
        \geq C\|w\|^2_{H^1(\Omega_\eta)} \geq C(\etanorm)\|\delta_i^h\tilde{v}\|_{H^1(\Omega)}^2,
    \end{multline}
    where in the first inequality above we have used Poincar\'e's and Korn's inequalities, as well as the smallness assumption on $\gamma\le  \gamma^*$, and in the second inequality we have used  \cref{composition regularity}. It follows from \eqref{estv:100}, \eqref{estv:101}, and \eqref{estv:102} that
        \bq
    \|\delta_i^h\tilde{v}\|_{H^1(\Omega)}^2 \leq C(\etanorm) \Bigl( \|\delta_i^h\tilde{v}\|_{H^1(\Omega)} +D
 \Bigr)D,   
    \eq
   where $D= \|\tilde{f}\|_{L^2(\Omega)} + \|\tilde{g}\|_{H^1(\Omega)} + \|k\|_{H^{\mez}(\T^d)}$.
   
After an application of Young's inequality, we obtain the desired estimate for $\delta_i^h\tilde{v}$: 
 \bq\label{estv:H2}
 \|\delta_i^h\tilde{v}\|_{H^1(\Omega)} \leq C(\etanorm) \Bigl( \|\tilde{f}\|_{L^2(\Omega)} + \|\tilde{g}\|_{H^1(\Omega)} + \|k\|_{H^{\mez}(\T^d)}
 \Bigr).
 \eq

Since  \eqref{delta p tilde bound} and \eqref{estv:H2} are uniform in $h$, we deduce that $\partial_i\tilde{v}\in H^1(\Omega)$ and $\partial_i\tilde{p}\in L^2(\Omega)$, with 
\bq\label{estv:H2tan}
\|\partial_i\tilde{v}\|_{H^1(\Omega)} +\|\partial_i\tilde{p}\|_{L^2(\Omega)} \leq C(\etanorm) \Bigl( \|\tilde{f}\|_{L^2(\Omega)} + \|\tilde{g}\|_{H^1(\Omega)} + \|k\|_{H^{\mez}(\T^d)}
 \Bigr),\quad 1\le i\le d.
\eq
{\it Bounds for vertical derivatives}. It remains  to prove that $\partial_z\tilde{v}\in H^1(\Omega)$ and $\partial_z\tilde{p}\in L^2(\Omega)$.  For convenience, we regard $z$ as the $(d+1)^{th}$ independent variable.  Using the first equation in (\ref{sys: flattened stokes with stress}) in the distributional sense, for any $1\leq j\leq d+1$ we can write
     \bq\label{eq for partial d+1 p}
     \partial_{d+1} \tilde{p}A_{(d+1)j} = \partial_{d+1}^2 \tilde{v}_j B + \cP^j,
     \eq
     where $B$ is defined as in \eqref{B definition} and
     \begin{multline*}
         \cP^j = \tilde{f}_j + \sum_{\substack{r,i,k\in \{1,...,d+1\} \\ \{r,i\}\neq \{d+1\}}} \partial_r \partial_i \tilde{v}_jA_{ik}A_{rk}  + \sum_{r,i,k\in \{1,...,d+1\}}  \partial_i\tilde{v}_j\partial_rA_{ik}A_{rk} +\gamma\sum_{k=1}^{d+1}\partial_k\tilde{v}_jA_{k1}-\sum_{k=1 }^d\partial_k\tilde{p}A_{kj}.
     \end{multline*}
     Using the  estimates  \eqref{regularity bound for A}, \eqref{variest:vp}, and \eqref{estv:H2tan}, we find
     \bq\label{bound for cPj}
     \|\cP^j\|_{L^2(\Omega)}\leq C(\|\eta\|_{H^{s+\tdm}(\T^d)})D.
     \eq
    On the other hand, taking $\partial_{d+1}$ of the second equation in \eqref{sys: flattened stokes with stress} yields
    \bq\label{identity:dz:2}
    \sum_{j=1}^{d+1} \partial_{d+1}^2\tilde{v}_jA_{(d+1)j} = \cP,
    \eq
     where 
     \[
     \cP = \partial_{d+1}\tilde{g} - \sum_{\substack{k\in \{1,...,d\} \\ j\in\{1,...,d+1\} }} \partial_{d+1}\partial_{k}\tilde{v}_jA_{kj}- \sum_{k,j=1}^{d+1} \partial_k\tilde{v}_j\partial_{d+1}A_{kj}
     \]
  satisfies
     \bq\label{bound for cP}
     \|\cP\|_{L^2(\Omega)}\leq C(\|\eta\|_{H^{s+\tdm}(\T^d)})D.
     \eq
Now, we multiply (\ref{eq for partial d+1 p}) by $B^{-1}A_{(d+1)j}$, sum over $j$, and add the resulting equation to  \eqref{identity:dz:2} in order to cancel out $ \sum_{j=1}^{d+1} \partial_{d+1}^2\tilde{v}_jA_{(d+1)j}$.  We obtain
     \bq
        \partial_{d+1}\tilde{p} = \cP  +B^{-1} \sum_{j=1}^{d+1}A_{(d+1)j} \cP^j.
     \eq
Consequently,  the estimates  \eqref{bound for cPj}, \eqref{bound for cP},  \eqref{product estimate for B}, and \eqref{regularity bound for A} imply 
    \bq\label{estdzp} 
       \|\partial_{d+1}\tilde{p}\|_{L^2(\Omega)}\leq C(\|\eta\|_{H^{s+\tdm}(\T^d)})D.
    \eq
Combing this and    \eqref{eq for partial d+1 p}  yields
       \bq\label{estdz2v}
    \|\partial_{d+1}^2\tilde{v}_j\|_{L^2(\Omega)}\leq C(\|\eta\|_{H^{s+\tdm}(\T^d)})D.  
    \eq
Combining the estimates \eqref{estv:H2tan}, \eqref{estdzp}, \eqref{estdz2v},  \eqref{flat to non-flat bound}, and \eqref{non-flat to flat bound}, we obtain  \eqref{high reg estimate for stokes with stress} for  $\sigma=1$.

 \underline{ Step 2.}  Next, suppose  that for $1\leq\sigma\leq s$ and for any weak solution $(v,p)$ and any data $(f,g,k)$ as in the hypothesis, we have the estimate
    \bq\label{inductionhypo:reg}
        \|v\|_{H^{\sigma+1}(\Omega_\eta)} + \|p\|_{H^{\sigma}(\Omega_\eta)} \leq C(\|\eta\|_{H^{s+\tdm}(\T^d)}) \bigl( \|f\|_{H^{\sigma-1}(\Omega_\eta)} + \|g\|_{H^{\sigma}(\Omega_\eta)} + \|k\|_{H^{\sigma-\mez}(\T^d)}\bigr).
    \eq
The estimates (\ref{flat to non-flat bound}) and (\ref{non-flat to flat bound}) then imply
    \bq\label{inductionhypo:reg:2}
    \|\Tilde{v}\|_{H^{\sigma+1}(\Omega)} + \|\Tilde{p}\|_{H^{\sigma}(\Omega)} \leq C(\|\eta\|_{H^{s+\tdm}(\T^d)})(\|\Tilde{f}\|_{H^{\sigma-1}(\Omega)} + \|\Tilde{g}\|_{H^{\sigma}(\Omega)} + \|k\|_{H^{\sigma-\mez}(\T^d)}).
    \eq
  
  For any fixed $1\leq j\leq d$, formally taking $\partial_j$ of  (\ref{sys: flattened stokes with stress}) yields
    \bq \label{djsystem}
     \begin{cases}
        -\nabla_{x,z}\big(\nabla_{x,z}\partial_j\Tilde{v}_i A + A^T(\partial_i\partial_j \tilde{v})^T\big):A^T + \nabla_{x,z}\partial_j\Tilde{p}A\cdot e_i - \gamma\nabla_{x,z}\partial_j\Tilde{v}_iA\cdot e_1 = \bar{f}_i & \quad \text{in } \Omega, \\
         \nabla_{x,z}\partial_j\Tilde{v} : A^T = \bar{g}  & \quad \text{in } \Omega, \\
         [\partial_j\Tilde{p}I-(\nabla_{x,z}\partial_j\Tilde{v}A+A^T(\nabla_{x,z}\partial_j\Tilde{v})^T)] \cN = \bar{k}& \quad \text{on } \T^{d}, \\ %write down the var
         \partial_j\Tilde{v}=0 & \quad \text{on } \Sigma_{-b},
     \end{cases}
     \eq
     where
     \bq 
     \begin{cases}
        \bar{f}_i = \partial_j \Tilde{f}_i +\nabla_{x,z}\big(\nabla_{x,z}\Tilde{v}_i \partial_jA + \partial_jA^T(\partial_i\tilde{v})^T\big):A^T + \nabla_{x,z}\big(\nabla_{x,z}\Tilde{v}_i A+ A^T(\partial_i\Tilde{v})^T\big):\partial_jA^T, \\
        \qquad \qquad \qquad\qquad \qquad \qquad - \nabla_{x,z}\Tilde{p}\partial_jA\cdot e_i + \gamma\nabla_{x,z}\Tilde{v}_i\partial_jA\cdot e_1  \\
         \bar{g} = \partial_j \Tilde{g} -\nabla_{x,z}\Tilde{v} : \partial_jA^T, \\
         \bar{k}= \partial_j k + [\nabla_{x,z}\Tilde{v}\partial_jA+\partial_jA^T(\nabla_{x,z}\Tilde{v})^T]\cN  -  [\Tilde{p}I-(\nabla_{x,z}\Tilde{v}A+A^T(\nabla_{x,z}\Tilde{v})^T)] \partial_j\cN.
     \end{cases}
     \eq
     Since $\sigma \ge 1$ in \eqref{inductionhypo:reg:2}, we have $(\Tilde{v}, \Tilde{p}) \in H^2(\Omega)\times H^1(\Omega)$. Consequently,  it can be shown that $(\p_j \Tilde{v}, \p_j\Tilde{p})$ is a weak solution of \eqref{djsystem}.   The induction hypothesis \eqref{inductionhypo:reg} then implies 
    \bq\label{estv:103}
    \|\partial_j\Tilde{v}\|_{H^{\sigma+1}(\Omega)} + \|\partial_j\Tilde{p}\|_{H^{\sigma}(\Omega)} \leq C(\|\eta\|_{H^{s+\tdm}(\T^d)})(\|\bar{f}\|_{H^{\sigma-1}(\Omega)} + \|\bar{g}\|_{H^{\sigma}(\Omega)} + \|\bar{k}\|_{H^{\sigma-\mez}(\T^d)}).
    \eq
    Since $\sigma \le s$ and $s>(d+1)/2$, we can apply Lemma \ref{product estimate for domain} to have
    \begin{equation*}
         \begin{split}
             \|(\nabla_{x,z}(\nabla_{x,z}\Tilde{v}_i \partial_jA+\partial_jA^T(\partial_i\tilde{v})^T):A^T)\|_{H^{\sigma-1}} +  \|(\nabla_{x,z}(\nabla_{x,z}\Tilde{v}_i A):\partial_jA^T)\|_{H^{\sigma-1}} \leq & c \|\Tilde{v}\|_{H^{\sigma+1}} \|A\|_{H^{s+1}}^2, \\
             \|\nabla_{x,z}\Tilde{p}\partial_jA\cdot e_i\|_{H^{\sigma-1}}\leq & c\|\Tilde{p}\|_{H^{\sigma}}\|A\|_{H^{s+1}}, \\ 
                 \| \nabla_{x,z}\Tilde{v} : \partial_jA\|_{H^{\sigma}} \leq & c \|\Tilde{v}\|_{H^{\sigma+1}} \|A\|_{H^{s+1}}.
         \end{split}
     \end{equation*}
 Hence,  \eqref{regularity bound for A} and the induction hypothesis yield
     \bq\label{estv:104}
     \|\bar{f}\|_{H^{\sigma-1}(\Omega)}+ \|\bar{g}\|_{H^{\sigma}(\Omega)}\leq C(\|\eta\|_{H^{s+\tdm}(\T^d)})(\|\Tilde{f}\|_{H^{\sigma}(\Omega)} + \|\Tilde{g}\|_{H^{\sigma+1}(\Omega)} + \|k\|_{H^{\sigma+\mez}(\T^d)}).
     \eq
   On the other hand, invoking Lemma \ref{product estimate on torus} and the trace inequality, we find 
     \bq\label{estv:105}
        \|\bar{k}\|_{H^{\sigma-\mez}(\T^d)}\leq C(\|\eta\|_{H^{s+\tdm}(\T^d)})(\|\Tilde{f}\|_{H^{\sigma}(\Omega)} + \|\Tilde{g}\|_{H^{\sigma+1}(\Omega)} + \|k\|_{H^{\sigma+\mez}(\T^d)}).
     \eq
     Inserting \eqref{estv:104} and \eqref{estv:105} in \eqref{estv:103} yields
     \bq\label{bound for horizonthal derivative}
        \|\partial_j\Tilde{v}\|_{H^{\sigma+1}(\Omega)} + \|\partial_j\Tilde{p}\|_{H^{\sigma}(\Omega)} \leq C(\|\eta\|_{H^{s+\tdm}(\T^d)})(\|\Tilde{f}\|_{H^{\sigma}(\Omega)} + \|\Tilde{g}\|_{H^{\sigma+1}(\Omega)} + \|k\|_{H^{\sigma+\mez}(\T^d)}),\quad 1\le j\le d.
     \eq
Using this and the first two equations in \eqref{sys: flattened stokes with stress}, we can    argue similarly  to the case $\sigma=1$ and obtain that  $\partial_z\Tilde{v}\in H^{\sigma+1}(\Omega)$  and $\partial_z\Tilde{p}\in H^{\sigma}(\Omega)$ with the same bound as \eqref{bound for horizonthal derivative}. Therefore, invoking  \eqref{flat to non-flat bound} and \eqref{non-flat to flat bound} again, we complete the induction. 
\end{proof}
\begin{rema}\label{rema:interpolation}
By the linearity of \eqref{sys:stokes with stress} and interpolation between Sobolev spaces, \eqref{high reg estimate for stokes with stress} holds for all $\sigma \in [0, s+1]$ if $f=0$ and for all $\sigma \in [1, s+1]$ otherwise, provided $(d+1)/2<s\in \Nn$. 
\end{rema}
The  proof of \cref{reg for stokes with stress} yields the following regularity statement for the system \eqref{sys: flattened stokes with stress}:
\begin{coro}\label{coro:regularityinflat}
Let $|\gamma|\le \gamma^*$, $(d+1)/2<s\in \Nn$, and $ \sigma\in [0, s+1]\cap \Nn$. Suppose $\eta\in H^{s+\tdm}(\T^d)$, $\tilde{f}\in H^{\sigma-1}(\Omega)\cap ({}_0H^1(\Omega))^*$, $\tilde{g}\in H^{\sigma}(\Omega)$ and $k\in H^{\sigma-\mez}(\T^d)$. Then the system \eqref{sys: flattened stokes with stress} has a unique solution $(\tilde{v},\tilde{p})\in H^{\sigma+1}(\Omega)\times H^\sigma(\Omega)$ satisfying the estimate 
\bq\label{energy est for flattened Stokes with stress}
    \|\Tilde{v}\|_{H^{\sigma+1}(\Omega)} + \|\Tilde{p}\|_{H^{\sigma}(\Omega)} \leq C(\|\eta\|_{H^{s+\tdm}(\T^d)})(\|\Tilde{f}\|_{H^{\sigma-1}(\Omega)\cap ({}_0H^1(\Omega))^*} + \|\Tilde{g}\|_{H^{\sigma}(\Omega)} + \|k\|_{H^{\sigma-\mez}(\T^d)}).
    \eq
     for some $C:\Rr^+ \to \Rr^+$ depending only on $(d,b,c^0,s,\sigma)$.
\end{coro}
\begin{rema}\label{remark: classic regularity}
   The classical regularity result in Theorem IV.$7.4$ in \cite{BoyerFabrie} requires $\eta\in W^{s+2, \infty}(\T^d)$. Precisely, if  $s\in \Nn$, $ \sigma\in [1,  s+1]$, and $\eta\in W^{s+2,\infty}(\T^d)$, then the weak solution of of \eqref{sys:stokes with stress} satisfies  
    \bq\label{reg:classical} \|v\|_{H^{\sigma+1}(\Omega_\eta)} + \|p\|_{H^{\sigma}(\Omega_\eta)} \leq C(\|\eta\|_{W^{s+2,\infty}(\T^d)}) \bigl( \|f\|_{H^{\sigma-1}(\Omega_\eta)} + \|g\|_{H^{\sigma}(\Omega_\eta)} + \|k\|_{H^{\sigma-\mez}(\T^d)}\bigr).
    \eq
    Moreover, if $f=0$ then $\sigma$ can be taken in $[0, s+1]$. This estimate is weaker than \eqref{high reg estimate for stokes with stress} and can be proven by modifying the above proof as follows. The regularity of $A(x, z)=  (\nabla \cF_\eta(x,z))^{-1} $ is  put in $W^{k, \infty}(\Omega)$ upon replacing  $\varrho$ in \eqref{def:varrho} with $\varrho(x, z)=\frac{b+z}{b}\eta(x)+z$, so that 
    \[
    \| A\|_{W^{k, \infty}(\Omega)}\le C(1+\|\eta\|_{W^{k+1}(\T^d)}).
    \] 
\end{rema}
Next, we prove a regularity result for the $\gamma$-Stokes system \eqref{sys:stokes with Dirichlet} which can be formally written in  the flattened domain as
\bq \label{sys: flattened stokes with Dirichlet}
     \begin{cases}
        -\nabla_{x,z}\big(\nabla_{x,z}\Tilde{v}_i A + A^T(\partial_i\tilde{v})^T\big):A^T + \nabla_{x,z}\Tilde{p}A\cdot e_i - \gamma\nabla_{x,z}\Tilde{v}_iA\cdot e_1 = \Tilde{f}_i & \quad \text{in } \Omega,\quad 1\le i\le d, \\
         \nabla_{x,z}\Tilde{v} : A^T = \Tilde{g}  & \quad \text{in } \Omega, \\
         \cN^\perp[\Tilde{p}I-(\nabla_{x,z}\Tilde{v}A+A^T(\nabla_{x,z}\Tilde{v})^T)] \cN = l& \quad \text{on } \T^{d},\\ 
         \tilde{v}\cdot \cN = h & \quad \text{on } \T^{d},\\ 
         \Tilde{v}=0 & \quad \text{on } \Sigma_{-b}.
     \end{cases}
     \eq
 
    \begin{theo}\label{reg for stokes with Dirichlet}
    Let $|\gamma|\leq \gamma^*$, $(d+1)/2<s\in \Nn$, and $ \sigma \in [0, s]\cap \Nn$. Suppose $\eta\in H^{s+\tdm}(\T^d)$, $f\in H^{\sigma-1}(\Omega_\eta)  \cap ({}_0H^1(\Omega_\eta))^*$,  $g\in H^{\sigma}(\Omega_\eta)$, $l\in H^{\sigma-\mez}(\T^d)$, and $h\in H^{\sigma+\mez}(\T^d)$, where $l$, $g$, and $h$ satisfy (\ref{compatibility for l}) and (\ref{compatibility for h and g}). Then the weak solution   of \eqref{sys:stokes with Dirichlet} satisfies
    \bq\label{high reg estimate for stokes with Dirichlet}
    \begin{aligned}
    \|v\|_{H^{\sigma+1}(\Omega_\eta)} + \|p\|_{H^{\sigma}(\Omega_\eta)} \leq C(\|\eta\|_{H^{s+\tdm}(\T^d)}) \bigl(& \|f\|_{H^{\sigma-1}(\Omega_\eta) \cap ({}_0H^1(\Omega_\eta))^* } \\
    &\quad+ \|g\|_{H^{\sigma}(\Omega_\eta)} + \|l\|_{H^{\sigma-\mez}(\T^d)}+ \|h\|_{H^{\sigma+\mez}(\T^d)}\bigr)
    \end{aligned}
    \eq
    for some $C:\Rr^+ \to \Rr^+$ depending only on $(d,b,c^0,s,\sigma)$. %Moreover, if $f=0$ then the preceding assertion holds for $0\le \sigma\le s$. 
\end{theo}
\begin{proof}
 
 Since  $\eta\in H^{s+\tdm}(\T^d)\hookrightarrow W^{2,\infty}(\T^d)$ for $s>(d+1)/2$,  by \cref{weak sol for stokes with Dirichlet}, \eqref{sys:stokes with Dirichlet} has a unique weak solution $(v, p)$   satisfying \eqref{base energy estimate for stokes with Dirichlet}, i.e. 
  \bq \label{variest:vp:N}
    \|v\|_{{}_0H^1(\Omega_\eta)} + \|p\|_{L^2(\Omega_\eta)} \leq C(\|\eta\|_{W^{2,\infty}(\T^d)}) \bigl( \|f\|_{\bigl(  {}_0H^1(\Omega_\eta)\bigr)^*} + \|g\|_{L^2(\Omega_\eta)} + \|l\|_{H^{-\mez}(\T^d)}+ \|h\|_{H^{\mez}(\T^d)}\bigr).
    \eq
  This yields \eqref{high reg estimate for stokes with Dirichlet} for $\sigma=0$. The case $\sigma \in [1, s]$ will be proven by induction. 
    
\underline{Step 1.} For the case $\sigma=1$, we  follow a scheme similar to that of Theorem \ref{reg for stokes with stress} for $\sigma=1$.  However, before flattening the domain, we first need to find a weak formulation in terms of test functions in ${}_0H^1(\Omega_\eta)$, as the weak formulation \eqref{weak stokes with Dirichlet} relies on the more restricted space of test functions ${}_0H^1_\cN(\Omega_\eta)$, which lacks invariance under taking difference quotients. To this aim, we define the functional $\Lambda_{(v, p)}\in ({}_0H^1(\Omega_\eta))^*$ via 
\bq\label{def:Ldup}
u\mapsto \int_{\Omega_\eta } \mez \Dd v: \Dd u - p\dv u - \gamma \partial_1 v\cdot u - \langle f,u\rangle_{\Omega_\eta}  + \langle l,u(\cdot,\eta(\cdot))\rangle_{\Sigma_\eta}.
\eq
It follows from \eqref{variest:vp:N} and  Cauchy-Schwarz's inequality that
\bq\label{Lambda bound}
\|\Lambda_{(v, p)}\|_{({}_0H^1(\Omega_\eta))^*}\leq C(\etanorm) \bigl( \|f\|_{({}_0H^1(\Omega_\eta))^*} + \|g\|_{L^2} + \|l\|_{H^{-\mez}} + \|h\|_{H^\mez} \bigr).
\eq
Moreover, the weak formulation (\ref{weak stokes with Dirichlet}) means  that $\Lambda_{(v, p)}\vert_{{}_0H^1_\cN(\Omega_\eta)} =0$. Thus, by Corollary \ref{Lambda representation by k}, there exists a scalar distribution  $\chi\in H^{-\mez}(\T^d)$ such that 
\bq\label{weakform:N:e}
\langle \Lambda_{(v, p)} , u \rangle_{\Omega_\eta} = \langle \chi, u(\cdot,\eta(\cdot))\cdot \cN(\cdot)\rangle_{H^{-\mez},H^\mez}\quad\forall u\in {}_0H^1(\Omega_\eta).
\eq
Moreover,  \eqref{Lambda bound} implies 
    \bq \label{k bound}
    \|\chi\|_{H^{-\mez}}\leq c(\etanorm) \|\Lambda\|_{({}_0H^1(\Omega_\eta))^*}\leq C(\etanorm) \bigl( \|f\|_{({}_0H^1(\Omega_\eta))^*} + \|g\|_{L^2} + \|l\|_{H^{-\mez}} + \|h\|_{H^\mez} \bigr).
    \eq
    We have shown that each weak solution $(v, p)$ of \eqref{sys:stokes with Dirichlet} satisfies the extended weak form \eqref{weakform:N:e} in terms of test functions in ${}_0H^1(\Omega_\eta)$. Since $\sigma=1$, we have $f\in L^2(\Omega_\eta)$, $g\in H^1(\Omega_\eta)$, $l\in H^\mez(\T^d)$ and $h\in H^{\tdm}(\T^d)$, and hence the pairings  \eqref{weakform:N:e}  are  $L^2$ inner products of functions. Thus, we can change variables to obtain the following weak formulation in the flattened domain:
\begin{multline}\label{weak stokes with Dirichlet in flattened domain}
             \int_\Omega \Bigl[ \mez (\nabla_{x,z}\tilde{v}A + A^T(\nabla_{x,z}\tilde{v})^T : \nabla_{x,z}\tilde{u}A + A^T(\nabla_{x,z}\tilde{u})^T ) - \tilde{p}(\nabla_{x,z}\tilde{u}: A^T) - \gamma(\nabla_{x,z}\tilde{v}Ae_1)\cdot \tilde{u}   \Bigr]J dxdz = \\
             +\int_\Omega (\tilde{f} \cdot\tilde{u}) J dxdz - \int_{\T^d} l(x)\cdot \tilde{u}(x,0) dx   + \langle \chi, u(\cdot,\eta(\cdot))\cdot \cN(\cdot)\rangle_{H^{-\mez},H^\mez}    
     \end{multline}
for all $u\in {}_0H^1_\cN(\Omega_\eta) $, together with the divergence condition $(\nabla_{x,z}\tilde{v}:A^T)=\tilde{g}$ and the normal Dirichlet condition $\tilde{v}(x,0)\cdot \cN(x,\eta(x)) = h(x)  $. 
Proceeding as in the proof of Theorem \ref{reg for stokes with stress}, for any horizontal index $1\le i \leq d$, we may take the difference quotient of (\ref{weak stokes with Dirichlet in flattened domain}) to obtain
\begin{multline}\label{quotient diff integral eq'}
         \int_\Omega \Bigl[ \mez \big(\nabla_{x,z}\delta_i^h\tilde{v}A + A^T(\nabla_{x,z}\delta_i^h\tilde{v})^T\big) : \big(\nabla_{x,z}\tilde{u}A + A^T(\nabla_{x,z}\tilde{u})^T\big) - \delta_i^h\tilde{p}(\nabla_{x,z}\tilde{u}: A^T) - \gamma(\nabla_{x,z}\delta_i^h\tilde{v}Ae_1)\cdot \tilde{u}   \Bigr]J \\
         + \Bigl[ \mez \big(\nabla_{x,z}\tilde{v}_i^h\delta_i^hA + \delta_i^hA^T(\nabla_{x,z}\tilde{v}_i^h)^T\big) : \big(\nabla_{x,z}\tilde{u}A + A^T(\nabla_{x,z}\tilde{u})^T\big) - \tilde{p}_i^h(\nabla_{x,z}\tilde{u}: \delta_i^hA^T) - \gamma(\nabla_{x,z}\tilde{v}_i^h\delta_i^hAe_1)\cdot \tilde{u}   \Bigr] J  \\
         +  \Bigl[ \mez \big(\nabla_{x,z}\tilde{v}_i^h A_i^h + (A_i^h)^T(\nabla_{x,z}\tilde{v}_i^h)^T\big) : \big(\nabla_{x,z}\tilde{u}\delta_i^h A + \delta_i^hA^T(\nabla_{x,z}\tilde{u})^T\big)   \Bigr]J dxdz \\
         +  \int_\Omega \Bigl[ \mez\big(\nabla_{x,z}\tilde{v}_i^h A_i^h + (A_i^h)^T(\nabla_{x,z}\tilde{v}_i^h)^T\big) : \big(\nabla_{x,z}\tilde{u} A_i^h + (A_i^h)^T(\nabla_{x,z}\tilde{u})^T\big)\\ - \tilde{p}_i^h(\nabla_{x,z}\tilde{u}: (A_i^h)^T) - \gamma(\nabla_{x,z}\tilde{v}_i^h A_i^he_1)\cdot \tilde{u}   \Bigr]\delta_i^hJ dxdz \\
         = \int_\Omega (\delta_i^h\tilde{f} \cdot\tilde{u}) Jdxdz +  \int_\Omega (\tilde{f}_i^h \cdot\tilde{u}) \delta_i^hJdxdz - \int_{\T^d} \delta_i^h l(x)\cdot \tilde{u}(x,0) dx  \\
         + \langle \delta_i^h \chi, \tilde{u}(\cdot,\eta(\cdot))\cdot \cN(\cdot)\rangle_{H^{-\mez},H^\mez} + \langle \chi^h_i, \tilde{u}(\cdot,\eta(\cdot))\cdot \delta_i^h \cN(\cdot)\rangle_{H^{-\mez},H^\mez}
    \end{multline}
       for all $\Tilde{u}\in {}_0H^1(\Omega)$.

{\it  Bounds for $\|\delta_i^h \tilde{p}\|_{L^2(\Omega)}$ and $\|\delta_i^h\chi\|_{H^{-\mez}}$}. We recall from the proof of \cref{reg for stokes with stress}  that there exists  $w\in {}_0H^1(\Omega_\eta)$ such that  $\tilde{w}:= w\circ \cF_\eta \in {}_0H^1(\Omega)$ satisfies $\nabla_{x,z}\tilde{w}:A^T = \delta_i^h \tilde{p}$ and 
    \bq\label{w tilde bound'}
    \begin{split}
        \|\tilde{w}\|_{H^1(\Omega)}&\leq C(\|\eta\|_{H^{s+\tdm}(\T^d)}) \|\delta_i^h\tilde{p}\|_{L^2(\Omega)} \\
        \|\tilde{w}(\cdot,\eta(\cdot))\cdot \cN(\cdot)\|_{H^\mez} &\leq C(\|\eta\|_{H^{s+\tdm}(\T^d)}) \|\delta_i^h\tilde{p}\|_{L^2(\Omega)}
    \end{split}
    \eq
where $C$ depends only on $(d,b,c^0)$. Choosing $\tilde{u}=\tilde{w}$ in (\ref{quotient diff integral eq'}) and arguing as in the proof of Theorem \ref{reg for stokes with stress}, we obtain
\bq \label{delta p tilde bound'}
    \|\delta_i^h\tilde{p}\|_{L^2(\Omega)} \leq C(\etanorm)\Bigl( \|\delta_i^h\tilde{v}\|_{H^1(\Omega)} + \|\delta^h_i \chi\|_{H^{-\mez}}+ \|\tilde{f}\|_{L^2(\Omega)} + \|\tilde{g}\|_{H^1(\Omega)} + \|l\|_{H^{\mez}(\T^d)} + \|h\|_{H^\tdm}
 \Bigr).
\eq
It remains to bound $\|\delta_i^h\chi\|_{H^{-\mez}}$ on the right-hand side of \eqref{delta p tilde bound'}.  By virtue of Riesz's representation theorem and \cref{lifting neumann},  there exists $\tilde{u}\in{}_0H^1(\Omega)$ such that 
$$\langle \delta_i^h \chi, \tilde{u}(\cdot,\eta(\cdot))\cdot \cN(\cdot)\rangle_{H^{-\mez},H^\mez} = \|\delta_i^h \chi \|_{H^{-\mez}}^2$$
and 
$$\|\tilde{u}\|_{H^1} \leq c(\|\eta\|_{W^{2, \infty}})\|\delta_i^h \chi \|_{H^{-\mez}}\le C(\etanorm)\|\delta_i^h \chi \|_{H^{-\mez}}.$$
With this choice of $\Tilde{u}$ in \eqref{quotient diff integral eq'}, we obtain
\bq\label{delta k bound:0}
    \|\delta_i^h \chi \|_{H^{-\mez}} \leq C(\etanorm)\Bigl( \|\delta_i^h\tilde{v}\|_{H^1(\Omega)} + \|\chi\|_{H^{-\mez}}+ \|\tilde{f}\|_{L^2(\Omega)} + \|\tilde{g}\|_{H^1(\Omega)} + \|l\|_{H^{\mez}(\T^d)} + \|h\|_{H^\tdm}
 \Bigr).
\eq
We note in particular that for the last term in \eqref{quotient diff integral eq'} we have used \cref{product estimate on torus} and the fact that $s>(d+1)/2$ to bound 
\begin{align*}
\left|\langle \chi^h_i, \tilde{u}(\cdot,\eta(\cdot))\cdot \delta_i^h \cN(\cdot)\rangle_{H^{-\mez},H^\mez} \right|&\le C\|\chi\|_{H^{-\mez}} \|\tilde{u}\vert_{\Sigma_\eta}\|_{H^\mez(\T^d)} \| \delta_i^h \cN(\cdot)\|_{H^{s-\mez}}\\
&\le C(\| \eta\|_{H^{s+\tdm}}) \|\chi\|_{H^{-\mez}} \| \tilde{u}\|_{H^1(\Omega_\eta)}.
\end{align*}
Inserting  $H^{-\mez}$ bound \eqref{k bound} for $\chi$ in \eqref{delta k bound:0}, we deduce 
\bq\label{delta k bound}
    \|\delta_i^h \chi \|_{H^{-\mez}} \leq C(\etanorm)\Bigl( \|\delta_i^h\tilde{v}\|_{H^1(\Omega)} + \|\tilde{f}\|_{L^2(\Omega)} + \|\tilde{g}\|_{H^1(\Omega)} + \|l\|_{H^{\mez}(\T^d)} + \|h\|_{H^\tdm}
 \Bigr).
\eq
{\it Bound  for $\|\delta_i^h \tilde{v}\|_{H^1(\Omega)}$}.  We choose $\tilde{u}= \delta_i^h \tilde{v}\in {}_0H^1(\Omega)$ in \eqref{quotient diff integral eq'} and use the fact that the divergence and the normal Dirichlet data imply $\nabla_{x,z}\delta_i^h\tilde{v}: A^T = \delta_i^h \tilde{g} - \nabla_{x,z}\tilde{v}_i^h:\delta_i^hA^T$ and $\delta_i^h\tilde{v}(\cdot,\eta(\cdot))\cdot \cN(\cdot) = \delta_i^h h(\cdot) - \tilde{v}_i^h(\cdot,\eta(\cdot)) \cdot \delta_i^h \cN(\cdot)$. Using in addition \eqref{variest:vp:N}, \eqref{delta p tilde bound'} and  \eqref{delta k bound}, we  then obtain
\bq\label{estfqv:N}
    \|\delta_i^h\tilde{v}\|_{H^1(\Omega)}^2 \leq C(\etanorm)( \|\delta_i^h\tilde{v}\|_{H^1(\Omega)} +D)D,   
\eq
where $D= \|\tilde{f}\|_{L^2(\Omega)} + \|\tilde{g}\|_{H^1(\Omega)} + \|l\|_{H^{\mez}(\T^d)} + \|h\|_{H^\tdm}$. After an application of Young's inequality to \eqref{estfqv:N} and returning to  (\ref{delta p tilde bound'}), we obtain
\bq
\|\partial_i\tilde{v}\|_{H^1(\Omega)} +\|\partial_i\tilde{p}\|_{L^2(\Omega)} \leq C(\etanorm) D,\quad 1\le i\le d.
\eq
The same bounds for the vertical derivatives $\p_z\tilde{v}$ and $\p_z\tilde{p}$ can be obtained as the proof of Theorem \ref{reg for stokes with stress} by using the first two equations in \eqref{sys: flattened stokes with Dirichlet}. This completes the proof of \eqref{high reg estimate for stokes with Dirichlet} for $\sigma=1$.

\underline{Step 2.} We suppose (\ref{high reg estimate for stokes with Dirichlet}) holds for $1\leq \sigma\leq s-1$. 
We fix $1\le j\le d$. By formally taking  $\partial_j$ of \eqref{sys: flattened stokes with Dirichlet},  we find that
\bq\label{Stokes:N:dj}
     \begin{cases}
        -\nabla_{x,z}\big(\nabla_{x,z}\partial_j\Tilde{v}_i A+ A^T(\partial_i\partial_j \tilde{v})^T\big):A^T + \nabla_{x,z}\partial_j\Tilde{p}A\cdot e_i - \gamma\nabla_{x,z}\partial_j\Tilde{v}_iA\cdot e_1 = \bar{f}_i& \quad \text{in } \Omega, \\
         \nabla_{x,z}\partial_j\Tilde{v} : A^T = \bar{g} & \quad \text{in } \Omega, \\
         \cN^\perp[\partial_j\Tilde{p}I-(\nabla_{x,z}\partial_j\Tilde{v}A+A^T(\nabla_{x,z}\partial_j\Tilde{v})^T)] \cN = \bar{k} & \quad \text{on } \T^{d}, \\ 
         \partial_j \tilde{v}\cdot \cN = \bar{h}& \quad \text{on } \T^{d},\\ 
         \partial_j\Tilde{v}=0 & \quad \text{on } \Sigma_{-b},
     \end{cases}
     \eq
     where
      \bq
     \begin{cases}
        \bar{f}_i = \partial_j \Tilde{f}_i +\nabla_{x,z}\big(\nabla_{x,z}\Tilde{v}_i \partial_jA + \partial_j A^T(\partial_i\tilde{v})^T\big):A^T + \nabla_{x,z}\big(\nabla_{x,z}\Tilde{v}_i A+ A^T(\partial_i\tilde{v})^T\big):\partial_jA^T, \\
        \qquad\qquad \qquad \qquad\qquad \qquad \qquad \qquad - \nabla_{x,z}\Tilde{p}\partial_jA\cdot e_i + \gamma\nabla_{x,z}\Tilde{v}_i\partial_jA\cdot e_1,  \\
         \bar{g} = \partial_j \Tilde{g} -\nabla_{x,z}\Tilde{v} : \partial_jA^T, \\
         \bar{l}= \partial_j l +\cN^\perp [\nabla_{x,z}\Tilde{v}\partial_jA+\partial_jA^T(\nabla_{x,z}\Tilde{v})^T]\cN,  \\
         \qquad \qquad \qquad -  \cN^\perp[\Tilde{p}I-(\nabla_{x,z}\Tilde{v}A+A^T(\nabla_{x,z}\Tilde{v})^T)] \partial_j\cN - \partial_j\cN^\perp[\Tilde{p}I-(\nabla_{x,z}\Tilde{v}A+A^T(\nabla_{x,z}\Tilde{v})^T)] \cN, \\ 
         \bar{h} = \partial_jh -\tilde{v}\partial_j\cN.
     \end{cases}
     \eq
     Since $\tilde{v}\in H^{\sigma+1}(\Omega)\subset H^2(\Omega)$ and $\tilde{p}\in H^\sigma(\Omega)\subset H^1(\Omega)$, it can be shown that $(\p_j\tilde{v}, \p_j\tilde{p})$ is a weak solution of \eqref{Stokes:N:dj}. Consequently, the induction hypothesis implies 
\[
\| \p_j\tilde{v}\|_{H^{\sigma+1}(\Omega)}+\| \p_j\tilde{p}\|_{H^\sigma(\Omega)}\le C(\|\eta\|_{H^{s+\tdm}})\bigl(\|\bar{f}\|_{H^{\sigma-1}(\Omega_\eta)} + \|\bar{g}\|_{H^{\sigma}(\Omega_\eta)} + \|\bar{l}\|_{H^{\sigma-\mez}(\T^d)}+ \|\bar{h}\|_{H^{\sigma+\mez}(\T^d)}\bigr).
\]
  For $s>(d+1)/2$, we can apply Sobolev estimates in Section \ref{appendix:productestimates} to estimate the right-hand side, thereby obtaining

    \bq
    \|\partial_j\Tilde{v}\|_{H^{\sigma+1}(\Omega)} + \|\partial_j\Tilde{p}\|_{H^{\sigma}(\Omega)} \leq C(\|\eta\|_{H^{s+\tdm}(\T^d)}) \bigl( \|\tilde{f}\|_{H^{\sigma}(\Omega_\eta)} + \|\tilde{g}\|_{H^{\sigma+1}(\Omega_\eta)} + \|l\|_{H^{\sigma+\mez}(\T^d)}+ \|h\|_{H^{\sigma+\tdm}(\T^d)}\bigr).
    \eq
  We note that since $ \bar{h}$ involves $\p_j\cN\in H^{s-\mez}$, its $H^{\sigma+\mez}$ estimate requires $\sigma \le s-1$.   Finally, the vertical derivatives $\partial_z\Tilde{v}$ and $\p_z\tilde{p}$ can be bounded as before.
\end{proof}

%%%%%%%%%%%%%%%%%%%%%%%%%%%%%%%%%%%%%%%%%%%%%%%%%%%%%%

\section{Normal-stress to normal-Dirichlet operators}\label{sec:Psi}
\subsection{The linear normal-stress to normal-Dirichlet operator $\Psi_\gamma[\eta]$} 

For any given $\eta\in W^{1, \infty}(\T^d)$ satisfying \eqref{eta lower bound}, we define the linear {\it normal-stress to normal-Dirichlet} operator $\Psi_\gamma[\eta]$ as follows. For  {\it scalar functions} $\chi: \T^d\to \Rr$,
\bq
\Psi_\gamma[\eta]\chi=v\vert_{\Sigma_\eta}\cdot \cN,
\eq
where $(v, p)$ solves the $\gamma$-Stokes problem with prescribed normal stress condition:
\bq\label{sys: hom stokes with normal stress}
\begin{cases}
    -\gamma \partial_1 v - \Delta v + \nabla p = 0 &\quad\text{in~} \Omega_\eta, \\
    \nabla\cdot v= 0 &\quad\text{in~}  \Omega_\eta, \\
    (pI - \mathbb{D}v)(\cdot, \eta(\cdot))\mathcal{N}(\cdot) = \chi(\cdot)\cN(\cdot) &\quad\text{on~}\T^d, \\
    v=0 &\quad \text{on~} \Sigma_{-b}.
\end{cases}
\eq
It is readily seen that $\Psi_\gamma[\eta](\chi+c)=\Psi_\gamma[\eta]\chi$ for all $c\in \Rr$. The next proposition shows that $\Psi_\gamma[\eta]$ is an isomorphism of order $-1$ between  homogeneous Sobolev spaces in a range  determined by the regularity of $\eta$. 
\begin{prop} \label{bound for Psi and Psi inverse}
     Let $|\gamma|\le \gamma^*$, $(d+1)/2<s\in \Nn$, and suppose $\eta\in H^{s+\tdm}(\T^d)$. Then for all $\sigma\in [0, s]$, $\Psi_\gamma[\eta]: \mathring{H}^{\sigma-\mez}(\T^d)\to \mathring{H}^{\sigma+\mez}(\T^d)$ is an isomorphism with 
    \bq\label{bound for Psi}
    \|\Psi_\gamma[\eta]\|_{H^{\sigma-\mez}(\T^d)\to \mathring{H}^{\sigma+\mez}(\T^d)}\leq C(\|\eta\|_{H^{s+\tdm}(\T^d)})
    \eq
   and 
    \bq\label{bound for Psi inverse}
    \|(\Psi_\gamma[\eta])^{-1}\|_{\mathring{H}^{\sigma+\mez}(\T^d)\to \mathring{H}^{\sigma-\mez}(\T^d)}\leq C(\|\eta\|_{H^{s+\tdm}(\T^d)}),
    \eq
    where $C:\Rr^+\to \Rr^+$ depends only on $(d,b,c^0,s)$.
    
    Moreover, if $(d+1)/2+1<s\in \Nn$ and $\sigma \in [s-1, s]$, then for any $\chi \in H^{\sigma-\mez}(\T^d)$, we have
    \bq\label{tame:contPsi}
    \| \Psi_\gamma[\eta]\chi\|_{H^{\sigma+\mez}(\T^d)}\le C(\|\eta\|_{H^{s+\mez}})\big(\| \chi \|_{H^{\sigma-\mez}} +\| \chi \|_{H^{s-\frac52}}\| \eta\|_{H^{\sigma+\tdm}}\big).
    \eq
    \end{prop}
\begin{proof}
1. Let $\sigma \in [0, s]$ and $\chi \in H^{\sigma-\mez}(\T^d)$. We proceed to prove \eqref{bound for Psi}. We have $\sigma-\mez\in [-\mez, s-\mez]$ and \cref{product estimate on torus} (with $s_0=s_1=\sigma-\mez$, $s_2=s-\mez$) implies 
\bq\label{est:chiN}
\|\chi \cN\|_{H^{\sigma-\mez}}\le C\| \chi\|_{H^{\sigma-\mez}}\| \cN\|_{H^{s-\mez}}\le C(\| \eta\|_{H^{s+\mez}})\| \chi\|_{H^{\sigma-\mez}}.
\eq
By virtue of \cref{reg for stokes with stress} (see also \cref{rema:interpolation}), \eqref{sys: hom stokes with normal stress} has a unique solution $(v, p)\in H^{\sigma+1}(\Omega_\eta)\cap H^\sigma(\Omega_\eta)$ which satisfies the bound \eqref{high reg estimate for stokes with stress}.  Consequently, $v\vert_{\Sigma_\eta}\in H^{\sigma+\mez}(\T^d)$ by the trace theorem. Since $\sigma+\mez\le s+\mez$ and $\cN\in H^{s+\mez}$, \cref{product estimate on torus} implies that $\Psi_\gamma[\eta]\chi=v\vert_{\Sigma_\eta}\cdot \cN\in H^{\sigma+\mez}$ together with the estimate \eqref{bound for Psi}. It is the $H^{s+\mez}$ regularity of $\eta$ that limits  $\sigma$  to $[0, s]$ even though \cref{reg for stokes with stress} allows for $\sigma\in [0, s+1$].  Since $\na \cdot v=0$ and $v\vert_{\Sigma_{-b}}=0$, we have $\int_{\T^d}v\vert_{\Sigma_\eta}\cdot \cN=0$ by the divergence theorem. We have proven that $\Psi_\gamma[\eta]: H^{\sigma-\mez}(\T^d)\to \rH^{\sigma+\mez}(\T^d)$ for $\sigma \in [0, s]$, together with the norm estimate \eqref{bound for Psi}. 
 
2. Next, we prove that $\Psi_\gamma[\eta]$ is injective on $\rH^{-\mez}(\T^d) $. Suppose that  $\chi \in \rH^{-\mez}(\T^d)$ and $\Psi_\gamma[\eta](\chi)=0$. Using  \cref{weak sol for stokes with stress} and the weak form \eqref{weak stokes with stress}, we find
\bq
\int_{\Omega_\eta}\mez|\Dd v|^2 - \int_{\Omega_\eta} \gamma \partial_1 v \cdot v = -\int_{\T^d} \chi\cN\cdot v  = -\int_{\T^d} \chi\Psi_\gamma[\eta](\chi) =0.
\eq
The left-hand side is the bilinear form \eqref{bilinearB} which is coercive if $|\gamma|\le \gamma^*$. 

It follows that  $v=0$, and hence $p=c$ is a constant. The normal stress condition then implies $\chi =c$. But  $\chi$ has zero mean, so $\chi=0$. Therefore,  $\Psi_\gamma[\eta]$ is injective if  $|\gamma|\le \gamma^*$.

3. For the surjectivity of $\Psi_\gamma[\eta]$, we consider any $\sigma \in [0, s]$ and  an arbitrary function  $h\in \mathring{H}^{\sigma+\mez}(\T^d)$. Then, according to Proposition \ref{weak sol for stokes with Dirichlet}, the $\gamma$-Stokes system \eqref{sys:stokes with Dirichlet} with $f=g=l=0$ has a unique weak solution $(v,p)$ satisfying (\ref{weak stokes with Dirichlet}),  that is, 
\[
\int_{\Omega_\eta } \mez \Dd v: \Dd u - p\dv u - \gamma \partial_1 v\cdot u=0\quad \forall u\in {}_0H^1_\cN(\Omega_\eta).
\]
 \cref{Lambda representation by k} then implies that there exists a unique  {\it  scalar} distribution   $\chi\in H^{-\mez}(\T^d)$ such that 
$$\int_{\Omega_\eta } \mez \Dd v: \Dd u - p\dv u - \gamma \partial_1 v\cdot u = -\langle \chi, u(\cdot,\eta(\cdot))\cdot \cN(\cdot)\rangle_{H^{-\mez},H^\mez} \quad \forall u\in {}_0H^1(\Omega_\eta).$$
Comparing this with the weak form \eqref{weak stokes with stress}, we find that $(v, p)$ is the solution of \eqref{sys:stokes with stress} with $f=g=0$ and $k=(pI - \Dd v)\cN=\chi\cN$. Therefore, $(v, p)$ is the  solution of \eqref{sys: hom stokes with normal stress}. 
Moreover,  according to Theorem \ref{reg for stokes with Dirichlet}, we actually have $(v,p)\in H^{\sigma+1}\times H^\sigma$. Invoking the trace theorem, \cref{prop: composition regularity}, \cref{product estimate on torus}, and the condition that $\sigma\in [0, s]$, we  deduce that $\chi\cN= (pI - \Dd v)\cN \in H^{\sigma-\mez}(\T^d)$ and $\chi=\frac{\chi\cN\cdot \cN}{|\cN|^2}\in H^{\sigma-\mez}(\T^d)$. Upon subtracting  the mean of $\chi$ from $\chi$ and $p$, we conclude that  $\Psi_\gamma[\eta]\chi=h$ with $\chi\in \rH^{\sigma-\mez}(\T^d)$. Therefore, $\Psi_\gamma: \rH^{\sigma-\mez}(\T^d)\to \rH^{\sigma+\mez}(\T^d)$ is surjective and \eqref{bound for Psi inverse} holds.

4. Now we assume $1+(d+1)/2<s\in \Nn$ and $\sigma \in [s-1, s]$ and prove the tame estimate \eqref{tame:contPsi} for any $\chi \in H^{\sigma-\mez}(\T^d)$. We will apply  \cref{reg for stokes with stress} with $s$ replaced by $\Nn\ni s':=s-1>(d+1)/2$. For $\chi \in H^{\sigma-\mez}(\T^d)$, we have $\chi \cN\in H^{\sigma -\mez}(\T^d)$. Since $\sigma \in [s-1, s]=[s', s'+1]$ and $s-2> 0$, \cref{reg for stokes with stress} (with $s'$ in place of $s$) then gives 

 \begin{align*}
&\| v\|_{H^{\sigma+1}(\Omega_\eta)}\le C(\|\eta\|_{H^{s+\mez}})\| \chi \cN\|_{H^{\sigma-\mez}},\\
&\| v\|_{H^{s-1}(\Omega_\eta)}\le C(\|\eta\|_{H^{s+\mez}})\| \chi \cN\|_{H^{s-\frac52}}.
\end{align*}
Since $H^{s+\mez}(\T^d)\subset W^{1, \infty}(\T^d)$ for $s>(d+1)/2$, the trace theorem implies  
\[
\| v\vert_{\Sigma_\eta}\|_{H^{r-\mez}(\T^d)}\le C(\|\eta\|_{H^{s+\mez}})\| v\|_{H^{r}(\Omega_\eta)},\quad r\ge 1.
\]
For $s>(d+1)/2+1$ and $\sigma \in [s-1, s]$, we can apply \cref{product estimate on torus} to deduce 
\bq\label{tame:Psi:vest}
\begin{aligned}
&\| v\vert_{\Sigma_\eta} \|_{H^{\sigma+\mez}(\T^d)}\le  C(\|\eta\|_{H^{s+\mez}})\| \chi \|_{H^{\sigma-\mez}},\\
 &\| v\vert_{\Sigma_\eta} \|_{H^{s-\tdm}(\T^d)}\le C(\|\eta\|_{H^{s+\mez}})\| \chi \|_{H^{s-\frac52}}.
 \end{aligned}
 \eq
Hence,  applying the tame product estimate \eqref{tame:product} to $\Psi_\gamma[\eta]\chi=v\vert_{\Sigma_\eta}\cdot\cN$ and invoking \eqref{tame:Psi:vest}, we obtain 
\begin{align*}
\| \Psi_\gamma[\eta]\chi\|_{H^{\sigma+\mez}}&\le C\| v\vert_{\Sigma_\eta}\|_{H^{\sigma+\mez}}\| \cN\|_{L^\infty}+C\| v\vert_{\Sigma_\eta}\|_{L^\infty}\| \cN\|_{H^{\sigma+\mez}}\\
&\le C(\|\eta\|_{H^{s+\mez}})\big(\| \chi \|_{H^{\sigma-\mez}} +\| \chi \|_{H^{s-\frac52}}(1+\| \eta\|_{H^{\sigma+\tdm}})\big)\\
&\le C(\|\eta\|_{H^{s+\mez}})\big(\| \chi \|_{H^{\sigma-\mez}} +\| \chi \|_{H^{s-\frac52}}\| \eta\|_{H^{\sigma+\tdm}}\big),
\end{align*}
where we have used the embedding $H^{s-\tdm}(\T^d)\subset L^\infty(\T^d)$. This completes the proof of \eqref{tame:contPsi}.
\end{proof}
\begin{rema}
The estimate \eqref{tame:contPsi} with $\sigma=s-1$ is a consequence of the estimate \eqref{bound for Psi} with $s$ replaced by $s':=s-1$ and $\sigma=s'$ - the maximum allowed value of $\sigma$. On the other hand, \eqref{tame:contPsi} with $\sigma \in (s-1, s]$ is a  tame estimate  with respect to $\chi$ and $\eta$. 
\end{rema} 
The next lemma records the fact that the linear operator $\Psi_0[\eta]$ is self-adjoint.
\begin{lemm}\label{lemm: Psi_0 is self-adjoint}
  For any $\chi_1, \chi_2 \in \mathring{H}^{-\mez}(\T^d)$ we have 
    \bq
        \langle \Psi_0[\eta] \chi_1 , \chi_2 \rangle_{H^{\mez},H^{-\mez}}  = \langle \Psi_{0}[\eta] \chi_2 , \chi_1 \rangle_{H^{\mez},H^{-\mez}} .
    \eq
\end{lemm}
\begin{proof}
    Let $(v_i,p_i)\in H^1(\Omega_\eta)\times L^2(\Omega_\eta)$ be the corresponding solutions to the problem (\ref{sys: hom stokes with normal stress}) with $\gamma = 0$ and $\chi=\chi_i$, $i=1,2$. By the weak form \eqref{weak stokes with stress} for $v_1$, we have 
    $$\int_{\Omega_\eta} \mez \Dd v_1 : \Dd v_2 = - \langle \chi_1\cN, v_2 \rangle_{H^{-\mez}, H^\mez} = -\langle  \chi_1, v_2\cdot \cN  \rangle_{H^{-\mez},H^{\mez}} = -\langle  \chi_1, \Psi_{0}[\eta] \chi_2 \rangle_{H^{-\mez},H^{\mez}}.
    $$
  Similarly, the weak formulation  \eqref{weak stokes with stress} for $v_2$ yields
    $$\int_{\Omega_\eta} \mez \Dd v_2 : \Dd v_1 = -\langle  \chi_2, \Psi_{0}[\eta] \chi_1 \rangle_{H^{-\mez},H^{\mez}}$$
    Comparing the above  equalities yields the result.
\end{proof}
Next, we prove contraction estimates for $\Psi_\gamma[\eta_1]-\Psi_\gamma[\eta_2]$.  
\begin{prop} \label{linearization and contraction for Psi}
    Let $(d+1)/2<s\in \Nn$, $\sigma \in [0, s]$, and consider $\eta_1, \eta_2\in H^{s+\tdm}(\T^d)$ with $\inf_{x\in\T^d}(\eta_i +b)\geq c^0_i>0$.   Let 
    \[
   |\gamma|\le \min\left\{ \gamma^*(d, b, c_1^0, \|\eta_1\|_{W^{1, \infty}}), \gamma^*(d,b,c_2^0, \|\eta_2\|_{W^{1, \infty}})\right\},
    \]
where $\gamma^*$ is given by \eqref{def:gamma*}.  Then there exists $C:\Rr^+ \to \Rr^+$ depending only on $(d,b,c_1^0,c_2^0,s,\sigma)$ such that  for any $\chi \in H^{\sigma-\mez}(\T^d)$, we have
        \bq\label{contraction estimate for R eta}
        \|\Psi_\gamma[\eta_1](\chi)-\Psi_\gamma[\eta_2](\chi) \|_{H^{\sigma+\mez}(\T^d)} \leq C(\|(\eta_1,\eta_2)\|_{H^{s+\tdm}(\T^d)})\|\eta_1-\eta_2\|_{H^{s+\tdm}(\T^d)}\|\chi\|_{H^{\sigma-\mez}(\T^d)}. 
    \eq
Moreover, if $(d+1)/2+1< s \in \Nn$, $\sigma \in [s-1,s]$, and $\chi \in H^{\sigma-\mez}(\T^d)$, then we have 
\begin{multline}\label{tame contraction estimate for R eta}
        \|\Psi_\gamma[\eta_1](\chi)-\Psi_\gamma[\eta_2](\chi) \|_{H^{\sigma+\mez}(\T^d)} \\
        \leq C(\|(\eta_1,\eta_2)\|_{H^{s+\mez}(\T^d)})\Bigl\{\|\eta_1-\eta_2\|_{H^{s+\mez}(\T^d)}\|\chi\|_{H^{\sigma-\frac12}(\T^d)}  
        + \|\eta_1-\eta_2\|_{H^{\sigma+\tdm}(\T^d)}\|\chi\|_{H^{s-\frac52}(\T^d)} \\
        + \|\eta_1-\eta_2\|_{H^{s+\mez}(\T^d)}\|\chi\|_{H^{s-\frac52}(\T^d)}\big(\|\eta_1\|_{H^{\sigma+\tdm}(\T^d)} + \|\eta_2\|_{H^{\sigma+\tdm}(\T^d)}\big) 
        \Bigr\} .
\end{multline}
\end{prop}
\begin{proof}
We prove \eqref{contraction estimate for R eta} for the case when $\sigma$ is an integer. The non-integer case follows from linear interpolation.   

We treat the case $\sigma = 0$ separately.  
Let $(v_i,p_i)\in H^{1}(\Omega_{\eta_i})\times L^2(\Omega_{\eta_i})$ be the weak solutions to \eqref{sys: hom stokes with normal stress} in the sense of \eqref{weak stokes with stress} with  $\eta=\eta_i$, $i\in \{1,2\}$. Set $\tilde{v}_i=v_i\circ \cF_{\eta_i}$, $\tilde{p}_i=p_i\circ \cF_{\eta_i}$, $A_i = (\nabla_{x,z}\cF_{\eta_i})^{-1}$, and $J_i(x,z) = \det(\nabla_{x,z}\cF_{\eta_i})$. In the flatten domain $\Omega$, \eqref{quotient diff integral eq} implies
\begin{multline}
    \int_\Omega \Bigl[ \mez \left(\nabla_{x,z}\tilde{v}_iA_i + A_i^T(\nabla_{x,z}\tilde{v}_i)^T\right) : \left(\nabla_{x,z}\tilde{u}A_i + A_i^T(\nabla_{x,z}\tilde{u})^T\right) \\
             \quad- \tilde{p}_i(\nabla_{x,z}\tilde{u}: A_i^T) - \gamma(\nabla_{x,z}\tilde{v}_iA_ie_1)\cdot \tilde{u}   \Bigr]J_i(x,z) dxdz = - \langle \chi(x)\cN_i,  \tilde{u}(x,0) \rangle_{H^{-\mez}(\T^d), H^\mez(\T^d)}
\end{multline}
for all $\tilde{u}\in {}_0H^1(\Omega)$, and  $\nabla_{x,z}\tilde{v}_i:A_i^T=0$.   
Then  $\tilde{v}=\tilde{v}_1 - \tilde{v}_2$ and $\tilde{p}= \tilde{p}_1- \tilde{p}_2$ satisfy
\begin{multline}
    \int_\Omega \Bigl[ \mez \left(\nabla_{x,z}\tilde{v}A_1 + A_1^T(\nabla_{x,z}\tilde{v})^T\right) : \left(\nabla_{x,z}\tilde{u}A_1 + A_1^T(\nabla_{x,z}\tilde{u})^T\right) \\ - \tilde{p}(\nabla_{x,z}\tilde{u}: A_1^T) - \gamma(\nabla_{x,z}\tilde{v}A_1e_1)\cdot \tilde{u}   \Bigr]J_1(x,z) dxdz = - K(\tilde{u})
\end{multline}
for all  $\tilde{u}\in {}_0H^1(\Omega)$,
where 
\begin{multline}
    K(\tilde{u}) =  \int_\Omega \Bigl[ \mez \left(\nabla_{x,z}\tilde{v}_2(A_1-A_2) + (A_1^T-A_2^T)(\nabla_{x,z}\tilde{v}_2)^T\right) : \left(\nabla_{x,z}\tilde{u}A_1 + A_1^T(\nabla_{x,z}\tilde{u})^T\right) \\ - \tilde{p}_2(\nabla_{x,z}\tilde{u}: (A_1^T-A_2^T) - \gamma(\nabla_{x,z}\tilde{v}_2(A_1-A_2)e_1)\cdot \tilde{u}   \Bigr]J_1(x,z) dxdz \\
    + \int_\Omega \Bigl[ \mez \left(\nabla_{x,z}\tilde{v}_2A_2 + A_2^T(\nabla_{x,z}\tilde{v}_2)^T\right) : \left(\nabla_{x,z}\tilde{u}(A_1-A_2) + (A_1^T-A_2^T)(\nabla_{x,z}\tilde{u})^T\right)   \Bigr]J_1(x,z) dxdz \\
    + \int_\Omega \Bigl[ \mez \left(\nabla_{x,z}\tilde{v}_2A_2 + A_2^T(\nabla_{x,z}\tilde{v}_2)^T\right) : \left(\nabla_{x,z}\tilde{u}A_2 + A_2^T(\nabla_{x,z}\tilde{u})^T\right) \\ - \tilde{p}_2(\nabla_{x,z}\tilde{u}: A_2^T) - \gamma(\nabla_{x,z}\tilde{v}_2A_2e_1)\cdot \tilde{u}   \Bigr](J_1-J_2)(x,z) dxdz \\
    + \langle \chi(x)(\cN_1-\cN_2),  \tilde{u}(x,0) \rangle_{H^{-\mez}(\T^d), H^\mez(\T^d)}.
\end{multline}
We have
\bq\label{bound for A1-A2}
\begin{split}
&\|\cN_1-\cN_2\|_{H^{s+\mez}(\T^d)} \leq \|\eta_1- \eta_2\|_{H^{s+\tdm}(\T^d)},\\
&\|A_1- A_2\|_{H^{s+1}(\Omega)}\leq C(\|\eta_1\|_{H^{s+\tdm}(\T^d)},\|\eta_2\|_{H^{s+\tdm}(\T^d)})\|\eta_1-\eta_2\|_{H^{s+\tdm}(\T^d)},\\
&\|J_1 - J_2\|_{H^{s+1}(\Omega)} \leq  C(\|\eta_1\|_{H^{s+\tdm}(\T^d)},\|\eta_2\|_{H^{s+\tdm}(\T^d)})\|\eta_1-\eta_2\|_{H^{s+\tdm}(\T^d)}.
\end{split}
\eq
Using \cref{solving div problem} we argue as in the proof of \cref{reg for stokes with stress} that  there exists  $\tilde{w} \in  {}_0H^1(\Omega) $ satisfying  $(\nabla_{x,z}\tilde{w}:A_1^T) =   \tilde{p}$, with
\bq\label{tilde:w:contra}
    \|\tilde{w}\|_{H^1(\Omega)}\leq C(\|\eta\|_{H^{s+\tdm}(\T^d)}) \|\tilde{p}\|_{L^2(\Omega)}.
\eq
We choose $\tilde{u} = \tilde{w}$ and apply H\"older's inequality with the aid of \eqref{tilde:w:contra}, \eqref{base energy estimate for stokes with stress} and \cref{composition regularity}. Then we obtain the following bound for $\|\tilde{p}\|_{L^2(\Omega)}$
\bq\label{bound for p tilde in contraction}
    \|\tilde{p}\|_{L^2(\Omega)} \leq C(\|\eta_1\|_{H^{s+\tdm}(\T^d)},\|\eta_2\|_{H^{s+\tdm}(\T^d)})\|\eta_1-\eta_2\|_{H^{s+\tdm}(\T^d)}\|\chi\|_{H^{-\mez}(\T^d)} + C(\|\eta_1\|_{H^{s+\tdm}(\T^d)})\|\tilde{v}\|_{H^1(\Omega)}. 
\eq
Next, we choose $\tilde{u} = \tilde{v}$. Since $\nabla_{x,z}\tilde{v}: A_1^T = \nabla_{x,z}\tilde{v}_2: (A_1^T-A_2^T)$, we have
$$\|\nabla_{x,z}\tilde{v}: A_1^T\|_{L^2(\Omega)}\leq C(\|\eta_1\|_{H^{s+\tdm}(\T^d)},\|\eta_2\|_{H^{s+\tdm}(\T^d)}) \|\eta_1-\eta_2\|_{H^{s+\tdm}(\T^d)}\|\chi\|_{H^{-\mez}(\T^d)}.$$
Using the bound above  in combination with Korn's and Poincar\'e's inequalities and the smallness assumption $|\gamma|\le  \gamma^*$, we find
\bq
    \|\tilde{v}\|^2_{H^1(\Omega)} \leq C(\|\eta_1\|_{H^{s+\tdm}(\T^d)},\|\eta_2\|_{H^{s+\tdm}(\T^d)}) \|\eta_1-\eta_2\|_{H^{s+\tdm}(\T^d)}\|\chi\|_{H^{-\mez}(\T^d)}(\|\tilde{p}\|_{L^2(\Omega)+ \|\tilde{v}\|_{H^1(\Omega)}}).
\eq
It follows from  \eqref{bound for p tilde in contraction} and Young's inequality that
\bq\label{bound for v tilde in contraction}
     \|\tilde{v}\|_{H^1(\Omega)} \leq C(\|\eta_1\|_{H^{s+\tdm}(\T^d)},\|\eta_2\|_{H^{s+\tdm}(\T^d)}) \|\eta_1-\eta_2\|_{H^{s+\tdm}(\T^d)}\|\chi\|_{H^{-\mez}(\T^d)}.
\eq
Finally, since  
\bq\label{contraction in flattened domain}
(\Psi_\gamma[\eta_1] - \Psi_\gamma[\eta_2]) \chi = \tilde{v}_1\vert_{z=0}\cdot (\cN_1 - \cN_2) + \tilde{v}\vert_{z=0}\cdot  \cN_2,
\eq
\eqref{bound for v tilde in contraction}, \eqref{bound for A1-A2}, \eqref{base energy estimate for stokes with stress}, and \cref{composition regularity} imply \eqref{contraction estimate for R eta} with $\sigma=0$.

For $\sigma\geq 1$,  we have  $(v_i,p_i)\in H^{\sigma+1}(\Omega_{\eta_i})\times H^{\sigma}(\Omega_{\eta_i})$. Then $(\tilde{v}_i, \tilde{p}_i)$ satisfy  \eqref{sys: flattened stokes with stress}  in the strong sense (for $\sigma\geq 1$). Hence, 
 $\tilde{v}$ and $\tilde{p}$ satisfy (in the strong sense)
\bq
\begin{cases}
   -\nabla_{x,z}(\nabla_{x,z}\tilde{v} A_1):A_1^T + \nabla_{x,z} \tilde{p} A_1 - \gamma \nabla_{x,z}\tilde{v}A_1 e_1  = \bar{f} &\quad\text{in~} \Omega=\T^d\times (-b, 0), \\
    \nabla_{x,z}\tilde{v}:A_1^T= \bar{g} &\quad\text{in~} \Omega, \\
    (\tilde{p}I - \left(\nabla_{x,z}\tilde{v}A_1 +A_1^T (\nabla_{x,z}\tilde{v})^T)\right)\cN_1= \bar{k} &\quad\text{on~}\Sigma_0, \\
    \tilde{v}=0&\quad\text{on~} \Sigma_{-b},
\end{cases}
\eq
where
\bq
\begin{cases}
    \bar{f}= \nabla_{x,z}(\nabla_{x,z}\tilde{v}_2 (A_1-A_2)):A_1^T+ \nabla_{x,z}(\nabla_{x,z}\tilde{v}_2 A_2):(A_1-A_2)^T - \nabla_{x,z}\tilde{p}_2 (A_1-A_2) \\
    \qquad \qquad \qquad \qquad \qquad \qquad \qquad \qquad  \qquad \qquad \qquad \qquad \qquad \qquad  + \gamma\nabla_{x,z}\tilde{v}_2(A_1-A_2)e_1 
    \\
    \bar{g}=-\nabla_{x,z}\tilde{v}_2: (A_1-A_2)^T \\
    \\
    \bar{k}= \chi(\cN_1- \cN_2)  - \tilde{p}_2\vert_{z=0}(\cN_1 - \cN_2) 
     + \left(\nabla_{x,z}\tilde{v}_2 (A_1-A_2) + (A_1- A_2)^T(\nabla_{x,z}\tilde{v}_2)^T\right)\vert_{z=0}\cN_1 \\
    \qquad \qquad \qquad \qquad \qquad \qquad \qquad \qquad \qquad \qquad 
    + \left(\nabla_{x,z}\tilde{v}_2 A_2 + A_2^T(\nabla_{x,z}\tilde{v}_2)^T\right)\vert_{z=0}(\cN_1 - \cN_2).
\end{cases}
\eq
Applying \cref{coro:regularityinflat} gives
\bq
   \|\tilde{v}\|_{H^{\sigma+1}(\Omega)} + \|\tilde{p}\|_{H^\sigma(\Omega)} \leq C(\|\eta_1\|_{H^{s+\tdm}(\T^d)})(\|\bar{f}\|_{H^{\sigma-1}(\Omega)} + \|\bar{g}\|_{H^{\sigma}(\Omega)} + \|\bar{k}\|_{H^{\sigma-\mez}(\T^d)})
\eq
for $\sigma\in [1, s+1]$.  It follows from \eqref{bound for A1-A2} and  Lemmas   \ref{product estimate on torus} and \ref{product estimate for domain} that 
$$
\|\bar{f}\|_{H^{\sigma-1}(\Omega)} + \|\bar{g}\|_{H^{\sigma}(\Omega)} + \|\bar{k}\|_{H^{\sigma-\mez}(\T^d)} \leq C(\|\eta_1\|_{H^{s+\tdm}(\T^d)},\|\eta_2\|_{H^{s+\tdm}(\T^d)})\|\eta_1-\eta_2\|_{H^{s+\tdm}(\T^d)} \|\chi\|_{H^{\sigma-\mez}(\T^d)},
$$
whence
\bq\label{bound for u}
\|\tilde{v}\|_{H^{\sigma+1}(\Omega)}  \leq \tilde{C}(\|\eta_1\|_{H^{s+\tdm}(\T^d)},\|\eta_2\|_{H^{s+\tdm}(\T^d)})\|\eta_1-\eta_2\|_{H^{s+\tdm}(\T^d)} \|\chi\|_{H^{\sigma-\mez}(\T^d)},\quad \sigma \in [1, s+1].
\eq
Finally, when $\sigma\in [1, s]$, \eqref{contraction estimate for R eta} follows from \eqref{contraction in flattened domain} and \eqref{bound for u}. The restriction $\sigma \le s$ is due to the fact that $\cN_j\in H^{s+\mez}(\T^d)$.

{\bf 2.} We now turn to prove the tame contraction estimate \eqref{tame contraction estimate for R eta} under the stronger regularity condition  $1+(d+1)/2< s \in \Nn$ and $\sigma \in [s-1,s]$. Using \eqref{contraction in flattened domain} and the tame product estimate \eqref{tame:product}, we have
\begin{align*}
    \|\Psi_\gamma[\eta_1](\chi)-\Psi_\gamma[\eta_2](\chi) \|_{H^{\sigma+\mez}(\T^d)}& \leq C(d,s) \Bigl( \|\tilde{v}_1\vert_{z=0}\|_{L^\infty(\T^d)}\|\cN_1-\cN_2\|_{H^{\sigma+\mez}(\T^d)} \\
     \qquad&+ \|\tilde{v}_1\vert_{z=0}\|_{H^{\sigma+\mez}(\T^d)}\|\cN_1-\cN_2\|_{L^\infty(\T^d)} \\
    \qquad&+ \|\tilde{v}\vert_{z=0}\|_{L^\infty(\T^d)}\|\cN_2\|_{H^{\sigma+\mez}(\T^d)}  + \|\tilde{v}\vert_{z=0}\|_{H^{\sigma+\mez}(\T^d)}\|\cN_2\|_{L^\infty(\T^d)}
    \Bigr).
\end{align*}
For $s>(d+1)/2$, we have
\begin{align*}
&  \| \cN_j\|_{L^\infty(\T^d)}\le C(1+\| \eta_j\|_{H^{s+\mez}(\T^d)}),\quad  \| \cN_j\|_{H^{\sigma+\mez}(\T^d)}\le C(1+\| \eta_j\|_{H^{\sigma+\tdm}(\T^d)}),\\
& \|\cN_1-\cN_2\|_{L^\infty(\T^d)} \le C \|\eta_1-\eta_2\|_{H^{s+\mez}(\T^d)},\quad \|\cN_1-\cN_2\|_{H^{\sigma+\mez}(\T^d)} \le C \|\eta_1-\eta_2\|_{H^{\sigma+\tdm}(\T^d)}. 
\end{align*}
The fact that $(d+1)/2 < s-1 \in \Nn$ enables us to apply \cref{coro:regularityinflat} with $s$ replaced by  $s'=s-1$. In particular, the estimate \eqref{energy est for flattened Stokes with stress} with $\sigma=s-2>0$ and  $\sigma\in [s-1, s]=[s', s'+1]$  gives
\begin{align*}
   & \|\tilde{v}_1\vert_{z=0}\|_{L^\infty(\T^d)} \leq C \|\tilde{v}_1\vert_{z=0}\|_{H^{s-\tdm}(\T^d)} \leq C \|\tilde{v}_1\|_{H^{s-1}(\Omega)}  \leq C(\|\eta_1\|_{H^{s+\mez}(\T^d)}) \|\chi\cN_1\|_{H^{s-\frac52}(\T^d)},\\
   & \|\tilde{v}_1\vert_{z=0}\|_{H^{\sigma+\mez}(\T^d)} \leq C \|\tilde{v}_1\|_{H^{\sigma+1}(\Omega)}  \leq C(\|\eta_1\|_{H^{s+\mez}(\T^d)}) \|\chi\cN_1\|_{H^{\sigma-\mez}(\T^d)}.
\end{align*}

Invoking the product rule in \cref{product estimate on torus} yields 
\begin{align*}
   & \|\tilde{v}_1\vert_{z=0}\|_{L^\infty(\T^d)}  \leq C(\|\eta_1\|_{H^{s+\mez}(\T^d)}) \|\chi \|_{H^{s-\frac52}(\T^d)},\\
   & \|\tilde{v}_1\vert_{z=0}\|_{H^{\sigma+\mez}(\T^d)}  \leq C(\|\eta_1\|_{H^{s+\mez}(\T^d)}) \|\chi\|_{H^{\sigma-\mez}(\T^d)}.
\end{align*}
Similarly, applying  \eqref{bound for u} with $s-1$ in place of  $s$, we obtain
\begin{align*}
   &  \|\tilde{v}\vert_{z=0}\|_{L^\infty(\T^d)}\leq C(\|(\eta_1,\eta_2)\|_{H^{s+\mez}(\T^d)})\|\eta_1 - \eta_2\|_{H^{s+\mez}(\T^d)} \|\chi\|_{H^{s-\frac52}(\T^d)},\\
   &  \|\tilde{v}\vert_{z=0}\|_{H^{\sigma+\mez}(\T^d)} \leq C(\|(\eta_1,\eta_2)\|_{H^{s+\mez}(\T^d)})\|\eta_1 - \eta_2\|_{H^{s+\mez}(\T^d)} \|\chi\|_{H^{\sigma-\mez}(\T^d)}.
\end{align*}
Gathering the above estimates leads to \eqref{tame contraction estimate for R eta}.

\end{proof}

%%%%%%%%%%%%%%%%%%%%%%%%%%%%%%%%%%%%%%%%%%%%%%%%%%%%%%%%%
\subsection{The nonlinear normal-stress to normal-Dirichlet operator $\Phi_\gamma[\eta]$}
The linear normal-stress to normal-Dirichlet operator $\Psi_\gamma[\eta]$ defined above is associated to the linear Stokes problem \eqref{sys: hom stokes with normal stress}. Our goal in this subsection is to study the corresponding operator for the nonlinear Navier-Stokes problem in the regime of small data while the free boundary $\eta$ can still be large. We assume that $\eta$ satisfies \eqref{eta lower bound} throughout this section. 

We first establish the well-posedness of the $\gamma$-Navier-Stokes problem with prescribed normal stress: 
\begin{equation}\label{sys:navier stokes with stress}
    \begin{cases}
    -\gamma \partial_1 v - \Delta v+ v\cdot \nabla v + \nabla p = f \qquad & \Omega_\eta, \\
    \nabla\cdot v= g \qquad & \Omega_\eta, \\
    (pI - \mathbb{D}v)(\cdot, \eta(\cdot))\mathcal{N}(\cdot) = k(\cdot) \qquad &\T^d, \\
    v=0 \qquad & \Sigma_{-b}.
    \end{cases}
\end{equation}
This will be achieved by a fixed point argument. We define the nonlinear operator  $Tv=u$, where $(u, p)$  is the unique solution to the $\gamma$-Stokes problem
\begin{equation}\label{sys: T definition}
    \begin{cases}
    -\gamma \partial_1 u - \Delta u + \nabla p = f-v\cdot \nabla v  \qquad & \Omega_\eta, \\
    \nabla\cdot u= g \qquad & \Omega_\eta, \\
    (pI - \mathbb{D}u)(\cdot, \eta(\cdot))\mathcal{N}(\cdot) = k(\cdot) \qquad &\T^d, \\
    u=0 \qquad & \Sigma_{-b}.
    \end{cases}
\end{equation}
As such, $v$ is a solution  of \eqref{sys:navier stokes with stress} if and only if $v$ is a fixed point of $T$, the existence of which is the content of the next proposition. 
\begin{prop}\label{T has a fixed point}
    Let $|\gamma|\le \gamma^*$, $(d+1)/2<s\in \Nn$, and suppose $\eta\in H^{s+\tdm}(\T^d)$. There exists 
    \[
    \delta^*=\delta^*(\|\eta\|_{H^{s+\tdm}(\T^d)}, d,b,c^0, s)>0
    \]
     such that the following holds. For all $0< \delta \le \delta^*$, there exists $\delta'=\delta'(\|\eta\|_{H^{s+\tdm}(\T^d)}, d,b,c^0,s)>0$ such that if 
    \bq\label{nonlPsi:cddata}
    \max\{\|f\|_{H^{s-1}(\Omega_\eta)},\|g\|_{H^{s}(\Omega_\eta)},\|k\|_{H^{s-\mez}(\T^d)}\}\leq \delta',
    \eq
    then $T$ maps $\overline{B_{H^{s+1}(\Omega_\eta)}(0,\delta)}$ to itself and is a contraction.  Consequently, the $\gamma$-Navier-Stokes problem \eqref{sys:navier stokes with stress} has a unique solution $(v, p)\in {}_0H^{s+1}(\Omega_\eta)\times H^{s}(\Omega_\eta)$ with $\| v\|_{H^{s+1}(\Omega_\eta)}\le \delta$. Moreover,  $v$ and $p$ satisfy
    \bq\label{estimate for navier stokes with stress}
 \| v\|_{H^{s+1}(\Omega_\eta)}\le  C(\|\eta\|_{H^{s+\tdm}(\T^d)}) \bigl( \|f\|_{H^{s-1}(\Omega_\eta)} + \|g\|_{H^{s}(\Omega_\eta)} + \|k\|_{H^{s-\mez}(\T^d)}\bigr). 
   \eq
\end{prop}
\begin{proof}
    Let $v\in {}_0H^{s+1}(\Omega_\eta)$.  Since $H^s(\Omega_\eta)$ is an algebra for $s>(d+1)/2$,  we have
    \[
    \|v\cdot \nabla v\|_{H^{s-1}(\Omega_\eta)}\le \|v\cdot \nabla v\|_{H^{s}(\Omega_\eta)}\le C(\|\eta\|_{W^{1, \infty}(\T^d)}) \|v\|^2_{H^{s+1}(\Omega_\eta)}.
    \]
By virtue of \cref{reg for stokes with stress},  $(v, p)\in  {}_0H^s_\sigma(\Omega_\eta)\times H^{s-1}(\Omega_\eta)$. We set   $u=Tv$. The regularity estimate  (\ref{high reg estimate for stokes with stress}) implies
    \begin{multline}\label{est:T:NS}
        \|u\|_{H^{s+1}(\Omega_\eta)} \leq C(\|\eta\|_{H^{s+\tdm}(\T^d)}) \bigl( \|f\|_{H^{s-1}(\Omega_\eta)} +\|v\cdot \nabla v\|_{H^{s-1}(\Omega_\eta)}+ \|g\|_{H^{s}(\Omega_\eta)} + \|k\|_{H^{s-\mez}(\T^d)}\bigr)\\
        \leq C(\|\eta\|_{H^{s+\tdm}(\T^d)}) \bigl( \|f\|_{H^{s-1}(\Omega_\eta)} +\|v\|^2_{H^{s+1}(\Omega_\eta)}+ \|g\|_{H^{s}(\Omega_\eta)} + \|k\|_{H^{s-\mez}(\T^d)}\bigr),
    \end{multline}
    where $C$ depends only on $(d,b,c^0,s,\sigma)$.    In particular,  $T: H^{s+ 1}(\Omega_\eta)\to H^{s+ 1}(\Omega_\eta)$. We set $\delta^* = (4C(\|\eta\|_{H^{s+\tdm}(\T^d)}))^{-1}$.  Then, for all $0<\delta\le \delta^*$, choosing  $\delta' \leq \delta (6C(\|\eta\|_{H^{s+\tdm}(\T^d)}))^{-1}$, we deduce that $T$ maps $\overline{B_{H^{s+1}(\Omega_\eta)}(0,\delta)}$ to itself for all data $(f, g, k)$ satisfying \eqref{nonlPsi:cddata}.

    For contraction, suppose $u_1= Tv_1$ and $u_2=Tv_2$. Then $u:=u_1-u_2$ satisfies (\ref{sys:stokes with stress}) with $f=v_2\cdot \nabla v_2 - v_1 \cdot \nabla v_1 = v_1\nabla (v_2-v_1) + (v_2-v_1)\cdot \nabla v_2$, $g=0$, and $k=0$. Appealing again to (\ref{high reg estimate for stokes with stress}) yields
    \begin{multline}
         \|u\|_{H^{s+1}(\Omega_\eta)} \leq C(\|\eta\|_{H^{s+\tdm}(\T^d)}) \bigl( \| v_1\nabla (v_2-v_1)\|_{H^{s-1}(\Omega_\eta)} +\| (v_2-v_1)\cdot \nabla v_2\|_{H^{s-1}(\Omega_\eta)}\bigr) \\
         \leq C(\|\eta\|_{H^{s+\tdm}(\T^d)}) \bigl( \| v_1\|_{H^{s+1}(\Omega_\eta)} +\| v_2\|_{H^{s+1}(\Omega_\eta)}\bigr) \| v_1-v_2\|_{H^{s+1}(\Omega_\eta)}.
    \end{multline}
     Since $\delta\le \delta^*=(4C(\|\eta\|_{H^{s+\tdm}(\T^d)}))^{-1}$, we have 
     \bq
      \|u\|_{H^{s+1}(\Omega_\eta)} \leq \mez  \| v_1-v_2\|_{H^{s+1}(\Omega_\eta)}
     \eq
    which proves that $T$ is a contraction. The unique fixed point $v$ of $T$ is the unique solution of \eqref{sys:navier stokes with stress} in the ball  and satisfies \eqref{estimate for navier stokes with stress} in view of \eqref{est:T:NS} and the fact that $\delta C(\|\eta\|_{H^{s+\tdm}(\T^d)})\le \frac14$. 
\end{proof}

Using the above argument to the Navier-Stokes version of  \eqref{sys:stokes with Dirichlet}, we obtain
\begin{prop}\label{well-posedness for navier stokes with small data with Dirichlet}
    Let $(d+1)/2<s\in \Nn$, $|\gamma|\le \gamma^*$, and suppose $\eta\in H^{s+\tdm}(\T^d)$. There exists 
    \[
    \tilde{\delta}^*=\tilde{\delta}^*(\|\eta\|_{H^{s+\tdm}(\T^d)}, d,b,c^0, s)>0
    \]
     such that the following holds. For all $0< \delta \le \tilde{\delta}^*$, there exists $\tilde{\delta}'=\tilde{\delta}'(\|\eta\|_{H^{s+\tdm}(\T^d)}, d,b,c^0,s)>0$ such that if $(f, g, l, h)$ satisfies the compatibility conditions  \eqref{compatibility for l}, \eqref{compatibility for h and g}, and 
    \bq
    \max\{\|f\|_{H^{s-1}(\Omega_\eta)},\|g\|_{H^{s}(\Omega_\eta)},\|l\|_{H^{s-\mez}(\T^d)},\|h\|_{H^{s+\mez}(\T^d)}\}\leq \tilde{\delta}',
    \eq 
then the system
    \begin{equation}\label{sys:navier stokes with Dirichlet}
    \begin{cases}
    -\gamma \partial_1 v - \Delta v + v\cdot \nabla v+ \nabla p = f& \quad \text{in~}\Omega_\eta, \\
    \nabla\cdot v= g & \quad \text{in~} \Omega_\eta, \\
    \cN^\perp(pI - \mathbb{D}v)(\cdot, \eta(\cdot))\mathcal{N}(\cdot) = l(\cdot) \qquad & \quad \text{on~} \T^d,\\
    v(\cdot,\eta(\cdot))\cdot \cN(\cdot) = h(\cdot) & \quad \text{on~}\T^d, \\
    v=0 \qquad & \quad \text{on~}\Sigma_{-b}.
    \end{cases},
\end{equation}
has a unique solution $(v,p)\in {}_0H^{s+1}(\Omega_\eta)\times \rH^{s}(\Omega_\eta)$ with $\| v\|_{H^{s+1}(\Omega_\eta)} \le \delta$. Moreover, $v$ and $p$ satisfy
\begin{multline}\label{estimate for navier stokes with Dirichlet}
\|v\|_{H^{s+1}(\Omega_\eta)}+\|p\|_{H^{s}(\Omega_\eta)}\leq C(\|\eta\|_{H^{s+\tdm}(\T^d)}) \left(\|f\|_{H^{s-1}(\Omega_\eta)}+\|g\|_{H^{s}(\Omega_\eta)}+\|l\|_{H^{s-\mez}(\T^d)}+\|h\|_{H^{s+\mez}(\T^d)}\right),
\end{multline}
where $C:\Rr^+\to \Rr^+$  depends only on $(d,b,c^0,s)$.  
\end{prop}
\begin{defi}\label{def:PhiXi} For $\chi: \T^d\to \Rr$, suppose that  $(v, p)$ is the  solution of \eqref{sys:navier stokes with stress} with $f=g=0$ and  $k=\chi \cN$, we define the (nonlinear) normal-stress to normal-Dirichlet operator for the case of Navier-Stokes by
\bq
    \Phi_\gamma[\eta](\chi) = v\vert_{\Sigma_\eta}\cdot \cN.
\eq
For $h: \T^d\to \Rr$, suppose that  $(v,p)\in H^{s+1}(\Omega_\eta)\times \rH^s(\Omega_\eta)$ is the  solution of \eqref{sys:navier stokes with Dirichlet} with $f=g=l=0$.  Then the normal stress $(pI - \mathbb{D}v)\cN$ is parallel to the normal vector $\cN$, so $(pI - \mathbb{D}v)\cN=\chi \cN$ for some scalar function $\chi: \T^d\to \Rr$.   We define the (nonlinear) normal-Dirichlet to normal-stress operator for the case of Navier-Stokes by
\bq
    \Xi_\gamma[\eta](h) =\chi \equiv (pI-\Dd v)\cN\cdot \frac{\cN}{|\cN|^2}.
\eq
\end{defi}
It can be readily checked that $\Phi_\gamma[\eta](\chi+c)=\Phi_\gamma[\eta](\chi)$ for all $c\in \Rr$. In the next proposition, we address the well-definedness of these operators in suitable domains. % as well as some of their elementary properties.
\begin{prop} \label{prop of Phi and Xi}
    Let $|\gamma|\le \gamma^*$, $(d+1)/2<s\in \Nn$, and suppose  $\eta \in H^{s+\tdm}(\T^d)$. %  There exists $\delta_i=\delta_i(\|\eta\|_{H^{s+\tdm}(\T^d)})$, $i=1,2$, such that the following assertions hold.
    \begin{enumerate}
        \item[(i)]  There exist positive numbers $\delta^0$ and $\tilde{\delta}^0$ depending only on  $(\|\eta\|_{H^{s+\tdm}(\T^d)}, d,b,c^0, s)$ such that the operators $\Phi_\gamma[\eta]: B_{H^{s-\mez}(\T^d)}(0, \delta^0)\to \mathring{H}^{s+\mez}(\T^d)$ and $\Xi_\gamma[\eta]: B_{\rH^{s+\mez}(\T^d)}(0, \tilde{\delta}^0)\to H^{s-\mez}(\T^d)$ as described in \cref{def:PhiXi} are well-defined. Moreover,  there exist $M_\Phi$,  $M_\Xi :\Rr^+ \to \Rr^+$ depending only on $(d,b,c^0,s)$ such that 
        \begin{align} \label{bound for Phi}
        \|\Phi_\gamma[\eta](\chi)\|_{H^{s+\mez}(\T^d)} \leq M_\Phi(\|\eta\|_{H^{s+\tdm}(\T^d)})\|\chi\|_{H^{s-\mez}}\quad \forall \chi\in B_{H^{s-\mez}(\T^d)}(0,\delta^0),\\
   \label{bound for Xi}
        \|\Xi_\gamma[\eta](h)\|_{H^{s-\mez}(\T^d)} \leq M_\Xi(\|\eta\|_{H^{s+\tdm}(\T^d)})\|h\|_{H^{s+\mez}(\T^d)}\quad\forall h\in  B_{\rH^{s+\mez}(\T^d)}(0,\tilde{\delta}^0).
        \end{align}
        Moreover, $\Phi_\gamma[\eta]: B_{\rH^{s-\mez}(\T^d)}(0, \delta^0)\to \mathring{H}^{s+\mez}(\T^d)$ is injective. 
        \item[(ii)]  There exists $\delta^\dag>0$ depending only on $(\|\eta\|_{H^{s+\tdm}(\T^d)}, d,b,c^0, s)$ such that  the image of $B_{\rH^{s-\mez}(\T^d)}(0, \delta^0)$ under $\Phi_\gamma[\eta]$ contains $B_{\rH^{s+\mez}(\T^d)}(0, \delta^\dag)$. Moreover, the inverse of $\Phi_\gamma[\eta]$ is given by 
        \bq\label{form:Phiinverse}
        \Phi_\gamma[\eta]^{-1}(h)=\Xi_\gamma[\eta](h)-\int_{\T^d} \Xi_\gamma[\eta](h)dx\quad\forall h\in B_{\rH^{s+\mez}(\T^d)}(0, \delta^\dag). 
        \eq
%      In particular,  there exist $M_{\Phi^{-1}}: \Rr^+ \to \Rr^+$ depending only on $(d,b,c^0,s)$ such that 
%    \bq
%      \|\Phi_\gamma[\eta]^{-1}(h)\|_{H^{s-\mez}(\T^d)} \leq M_{\Phi^{-1}}(\|\eta\|_{H^{s+\tdm}(\T^d)})\|h\|_{H^{s+\mez}}\quad\forall h\in  B_{H^{s+\mez}(\T^d)}(0,\delta^\dag).
%      \eq
          \end{enumerate}
\end{prop}
\begin{proof} 
Let $\delta^*$ and $\tilde{\delta}^*$ be as in \cref{T has a fixed point} and \cref{well-posedness for navier stokes with small data with Dirichlet}, both depends only on $(\|\eta\|_{H^{s+\tdm}(\T^d)}, d,b,c^0, s)$.

(i) By virtue of \cref{well-posedness for navier stokes with small data with Dirichlet}, for $\delta=\tilde{\delta}^*$ there exists $\tilde{\delta}'>0$ such that the problem \eqref{sys:navier stokes with Dirichlet} has a unique solution  $(v, p)\in H^{s+1}(\Omega_\eta)\times \rH^s(\Omega_\eta)$ if $f=g=l=0$ and $h\in \rH^{s+\mez}(\T^d)$ with  $\| h\|_{H^{s+\mez}(\T^d)}\le \tilde{\delta}'$. Then $\Xi_\gamma[\eta](h)=\chi\equiv (pI-\Dd v)\cN\cdot \frac{\cN}{|\cN|^2}$ is well-defined in $H^{s-\mez}(\T^d)$ and obeys the bound  \eqref{bound for Xi} since $H^{s-\mez}(\T^d)$ is an algebra. Thus we can choose any  $\tilde{\delta}^0\le \tilde{\delta}'$.  Next, by \cref{T has a fixed point}, for $\delta=\delta^*$ there exists $\delta'$ such that \eqref{sys:navier stokes with stress}  with $f=g=0$ has a unique solution $(v, p)\in H^{s+1}(\Omega_\eta)\cap H^s(\Omega_\eta)$ if $\| k\|_{H^{s-\mez}(\T^d)}\le \delta'$. Since 
\[
\| \chi \cN\|_{H^{s-\mez}}\le C_1(\| \eta\|_{H^{s+\tdm}})\| \chi\|_{H^{s-\mez}},
\]
if $\| \chi\|_{H^{s-\mez}}\le \delta'/C_1(\| \eta\|_{H^{s+\tdm}})$  then we can choose $k=\chi \cN$. Thus \eqref{bound for Phi} holds with any choice of  $\delta^0\le  \delta'/C_1(\| \eta\|_{H^{s+\tdm}})$. Moreover, $\Phi_\gamma[\eta](\chi)=v\vert_{\Sigma_\eta}\cdot \cN$ has mean zero in view of the divergence theorem.

  To prove the injectivity of $\Phi_\gamma[\eta]$ on $B_{\mathring{H}^{s-\mez}(\T^d)}(0,\delta^0)$, suppose that $\Phi_\gamma[\eta](\chi_1)=\Phi_\gamma[\eta](\chi_2)$ for some $\chi_i\in B_{\mathring{H}^{s-\mez}(\T^d)}(0,\delta^0)$. Let $(v_i,p_i)\in H^{s+1}(\Omega_\eta)\times H^s(\Omega_\eta)$ be the corresponding unique solutions of   \eqref{sys:navier stokes with stress}  with $f=g=0$ and $k=\chi_i\cN$. Set $v=v_1-v_2$ and $p=p_1-p_2$. Then $v\vert_{\Sigma_\eta}\cdot \cN=\Phi_\gamma[\eta]\chi_1-\Phi_\gamma[\eta]\chi_2=0$, and hence the weak form \eqref{weak stokes with stress} with $u=v$ gives 
    $$ B(v, v)\equiv \int_{\Omega_\eta} \mez |\Dd v|^2 - \gamma \partial_1 v \cdot v =  -\int_{\Omega_\eta} [v_1\cdot \nabla v + v\cdot \nabla v_2 ]v -\int_{\Sigma_\eta} (\chi_1-\chi_2)v\cdot \cN = -\int_{\Omega_\eta} [v_1\cdot \nabla v + v\cdot \nabla v_2 ]v.$$
      Since $v_i\in H^{s+1}(\Omega_\eta)\subset W^{1, \infty}(\Omega_\eta)$ for $s>(d+1)/2$, H\"older's and Poincar\'e's inequalities imply 
    \[
   \Big|\int_{\Omega_\eta} [v_1\cdot \nabla v + v\cdot \nabla v_2 ]v\Big|\le C_2(\|v_1\|_{H^{s+1}}+\|v_2\|_{H^{s+1}})\|v\|_{H^1}^2\leq C_3\delta^0 \|v\|_{H^1}^2.
   \]
    Then invoking the coercivity \eqref{def:gamma*} of $B$, we deduce that $v=0$ if we further restrict $\delta^0<\frac{1}{4c_*C_3}$. It follows that  $p=c$ is a constant, and hence $(\chi_1-\chi_2)\cN=(pI - \Dd v)\cN=c\cN$. Since the $\chi_i$'s have mean zero, we deduce that $\chi_1-\chi_2=0$. Thus $\Phi_\gamma[\eta]$ is injective on $B_{\rH^{s-\mez}(\T^d)}(0, \delta^0)$. 
    
  (ii) From  (i) we have the following: if $h\in \rH^{s+\mez}(\T^d)$ with $\| h\|_{H^{s+\mez}(\T^d)}\le \tilde{\delta}^0$ then, in view of \eqref{bound for Xi},
   \[
   \|\Xi_\gamma[\eta](h)\|_{H^{s-\mez}(\T^d)} \leq M_\Xi(\|\eta\|_{H^{s+\tdm}(\T^d)})\|h\|_{H^{s+\mez}(\T^d)},
   \]
   and hence 
   \[
    \|\Xi_\gamma[\eta](h)-\int_{\T^d}\Xi_\gamma[\eta](h)\|_{H^{s-\mez}(\T^d)} \leq M_\Xi(\|\eta\|_{H^{s+\tdm}(\T^d)})\|h\|_{H^{s+\mez}(\T^d)}
    \]
   upon increasing  $M_\Xi$.  Choosing 
    \[
    \|h\|_{H^{s+\mez}}\le \frac{\delta^0}{M_{\Xi}(\|\eta\|_{H^{s+\tdm}})}=:\delta^\dag,
    \]
     we obtain that $\Xi_\gamma[\eta](h)-\int_{\T^d}\Xi_\gamma[\eta](h)\in B_{\rH^{s-\mez}(\T^d)}(0, \delta^0)$ - the domain of $\Phi_\gamma[\eta]$,  and 
    \[
    \Phi_\gamma[\eta]\big(\Xi_\gamma[\eta](h)-\int_{\T^d}\Xi_\gamma[\eta](h)\big)= \Phi_\gamma[\eta]\Xi_\gamma[\eta](h)=h. 
    \]
    This completes the proof of (ii). 
\end{proof}
\begin{prop}\label{prop:contrPhi}
    Let $(d+1)/2<s\in \Nn$. The following assertions hold. 
    \begin{enumerate}
        \item[(i)] Let $\eta\in H^{s+\tdm}(\T^d)$ satisfy $\inf_{x\in\T^d}(\eta +b)\geq c^0>0$, and let $|\gamma|\le \gamma^*=\gamma^*(d, b, c^0, \|\eta\|_{W^{1, \infty}})$, given by  \eqref{def:gamma*}. There exist $0<\delta^0_\sharp\le \delta^0$ and $0<\tilde{\delta}^0_\sharp\le \tilde{\delta}^0$ depending only on $(\|\eta\|_{H^{s+\tdm}(\T^d)}, d,b,c^0, s)$  such that 
        \bq\label{contraction estimate for Phi in chi}
        \|\Phi_\gamma[\eta](\chi_1)-\Phi_\gamma[\eta](\chi_2) \|_{H^{s+\mez}(\T^d)} \leq N_\Phi(\|\eta\|_{H^{s+\tdm}(\T^d)})\|\chi_1-\chi_2\|_{H^{s-\mez}(\T^d)}\quad\forall \chi_1,~\chi_2\in B_{H^{s-\mez}(\T^d)}(0, \delta^0_\sharp)
    \eq
    and
    \bq\label{contraction estimate for Xi in h}
        \|\Xi_\gamma[\eta](h_1)-\Xi_\gamma[\eta](h_2) \|_{H^{s-\mez}(\T^d)} \leq N_\Xi(\|\eta\|_{H^{s+\tdm}(\T^d)})\|h_1-h_2\|_{H^{s+\mez}(\T^d)}\quad\forall h_1,~h_2\in B_{\rH^{s-\mez}(\T^d)}(0, \tilde{\delta}^0_\sharp).
    \eq
where $N_\Phi$ and $N_\Xi$  depend  only on $(d,b,c^0,s)$.
        \item[(ii)] Let $\eta_i\in  H^{s+\tdm}(\T^d)$,  $i\in\{1,2\}$ satisfy $\inf_{x\in\T^d}(\eta_i +b)\geq c^0_i>0$. Let 
           \[
   |\gamma|\le \min\left\{ \gamma^*(d, b, c_1^0, \|\eta_1\|_{W^{1, \infty}}), \gamma^*(d,b,c_2^0, \|\eta_2\|_{W^{1, \infty}})\right\},
    \]
     There exists  $\delta^1_\sharp>0$ depending only on $(\|\eta_1\|_{H^{s+\tdm}(\T^d)}, \|\eta_2\|_{H^{s+\tdm}(\T^d)}, d,b,c^0_1, c_2^0, s)$  such that such that for all $\chi\in B_{H^{s-\mez}(\T^d)}(0,\delta^1_\sharp)$ we have 
    \begin{multline}\label{contraction estimate for Phi in eta}
        \|\Phi_\gamma[\eta_1](\chi)-\Phi_\gamma[\eta_2](\chi) \|_{H^{s+\mez}(\T^d)} \leq \tilde{N}_\Phi(\|\eta_1\|_{H^{s+\tdm}(\T^d)},\|\eta_2\|_{H^{s+\tdm}(\T^d)})\|\chi\|_{H^{s-\mez}(\T^d)}\|\eta_1-\eta_2\|_{H^{s+\tdm}(\T^d)},
    \end{multline}
    where $\tilde{N}_\Phi: \Rr^+ \to \Rr^+$ depends only on $(d,b,c_1^0,c_2^0,s)$.

    \end{enumerate}
\end{prop}
\begin{proof}
  (i)  We first restrict $\delta^0_\sharp\le \delta^0$ and consider  $\|\chi_i\|_{H^{s-\mez}}<\delta^0_\sharp$.   By \cref{prop of Phi and Xi}, the problem \eqref{sys:navier stokes with stress} with $f=g=0$ and $k=\chi_i\cN$ has a unique solution $(v_i, p_i)\in H^{s+1}(\Omega_\eta)\times H^s(\Omega_\eta)$ . Then $v:=v_1-v_2$ and $p:=p_1-p_2$ solve  the Stokes system \eqref{sys:stokes with stress} with $f= -v_2\cdot \nabla v - v\cdot \nabla v_1$, $g=0$, and $k=(\chi_1-\chi_2)\cN$. The estimate \eqref{high reg estimate for stokes with stress} implies 
    \begin{align*}
   \|v\|_{H^{s+1}(\Omega_\eta)} + \|p\|_{H^s(\T^d)}&\leq C_1(\etanorm)\left(\|v_2\cdot \nabla v + v\cdot \nabla v_1\|_{H^{s-1}(\Omega_\eta)}+ \|\chi_1-\chi_2\|_{H^{s-\mez}(\T^d)}\right)\\
   &\le C_2(\etanorm)\left(\big( \|v_1\|_{H^{s+1}(\Omega_\eta)}+\|v_2\|_{H^{s+1}(\Omega_\eta)}\big)\|v\|_{H^{s+1}(\Omega_\eta)}\right.\\
   &\qquad\left.+ \|\chi_1-\chi_2\|_{H^{s-\mez}(\T^d)}\right)\\
   &\le C_3(\etanorm)\left(\delta^0_\sharp\|v\|_{H^{s+1}(\Omega_\eta)}+ \|\chi_1-\chi_2\|_{H^{s-\mez}(\T^d)}\right).
    \end{align*}
    Thus, by further restricting  $\delta^0_\sharp\le (2C_3(\etanorm))^{-1}$, we obtain the contraction estimate \eqref{contraction estimate for Phi in chi}. An analogous argument yields \eqref{contraction estimate for Xi in h}.

 (ii) can be proven analogously to \cref{linearization and contraction for Psi}: we use the flattened system \eqref{sys: flattened stokes with stress} for $(\tilde{v}_i, \tilde{p_i}):=(v_1\circ \cF_{\eta_i}, p_i\circ \cF_{\eta_i})$ to obtain the system for $(\tilde{v}, \tilde{p}):=(\tilde{v}_1-\tilde{v_2}, \tilde{p}_1-\tilde{p}_2)$, and then apply Corollary \cref{coro:regularityinflat}. The difference is that the forcing term $\tilde{f}$ in \eqref{sys: flattened stokes with stress} now involves $\tilde{v}_1\cdot\na \tilde{v}_1-\tilde{v}_2\cdot \na \tilde{v}_2$ which can be absorbed  by the left-hand side of \eqref{energy est for flattened Stokes with stress}  if $\| \chi\|_{H^{s-\mez}}$ is sufficiently small. 
\end{proof}
\subsection{Further properties of $\Psi_\gamma[\eta]$}
In preparation for the stability analysis of traveling waves, we establish in this section coercive and commutator estimates for  $\Psi_\gamma[\eta]$.
\subsubsection{Coercive estimates}
\begin{lemm} \label{coercivity for Psi eta}
    Let $(d+1)/2<s\in \Nn$, and let $\eta_*\in H^{s+\tdm}(\T^d)$ satisfy \eqref{eta lower bound}. Consider $|\gamma|\le \gamma^*$ as  in \eqref{def:gamma*}. Then there exists $c_{\Psi_\gamma[\eta_*]}=c_{\Psi_\gamma[\eta_*]}(d, c^0, \eta_*)>0$ such that for any $\chi\in H^{-\mez}(\T^d)$ we have the coercive estimate
    \bq\label{coercive:Psi}
    \langle -\Psi_\gamma[\eta_*] \chi , \chi \rangle_{H^\mez,H^{-\mez}}  \geq c_{\Psi_\gamma[\eta_*]}\|\chi\|^2_{H^{-\mez}}.
    \eq
\end{lemm}
\begin{proof}
Since $\Psi_\gamma[\eta_*](\chi+c)=\Psi_\gamma[\eta_*]\chi$ for all $c\in \Rr$ and $\Psi_\gamma[\eta_*]\chi$ has mean zero, it suffices to prove \eqref{coercive:Psi} for $\chi\in \rH^{-\mez}(\T^d)$. 
Let  $(v, p)\in H^1(\Omega_{\eta_*})\times L^2(\Omega_{\eta_*})$ be the weak solution of  \eqref{sys: hom stokes with normal stress}. The weak formulation \eqref{weak stokes with stress}  yields
    \[
        \langle -\Psi_\gamma[\eta_*] \chi , \chi \rangle_{H^\mez,H^{-\mez}} =- \langle v(\cdot,\eta_*(\cdot)) , \chi \cN \rangle_{H^\mez,H^{-\mez}} = \int_{\Omega_{\eta_*}} \mez |\Dd v|^2 -\gamma\p_1v\cdot v\equiv B(v, v).
    \]
    Now, we prove the lemma by contradiction. Suppose there exists a sequence of $\chi_n\in \rH^{-\mez}(\T^d)$ and the corresponding velocities $v_n\in H^1(\Omega_{\eta_*})$ such that $B(v_n, v_n) < \frac{1}{n}\|\chi_n\|_{H^{-\mez}}$.  For $|\gamma|\le\gamma^*$, we have the coercive estimate  \eqref{def:gamma*} for $B$. Combining this with  Korn's and Poincare's inequalities, we deduce
    $$\|v_n\|_{H^1(\Omega_{\eta_*})}<\frac{C}{n}\|\chi_n\|_{H^{-\mez}},$$
    where $C$ is independent of $n$.   Upon rescaling, we may assume that $\| \chi_n\|_{H^{-\mez}(\T^d)}=1$ for all $n$, so that $v_n\to 0$ strongly in $H^1(\Omega_{\eta_*})$ as $n\to \infty$. 
    In particular, $v_n\cdot \cN \to 0$ strongly in $H^{\mez}(\T^d)$. Using the boundedness of the operator $\Psi_\gamma[\eta_*]^{-1}: \rH^\mez\to \rH^{-\mez}$ given in Proposition \ref{bound for Psi and Psi inverse}, we deduce that $\chi_n = \Psi_\gamma[\eta_*]^{-1}(v_n\cdot \cN) \to 0$ strongly in $H^{-\mez}$, contradicting the assumption that  $\| \chi_n\|_{H^{-\mez}(\T^d)}=1$ for all $n$. 
\end{proof}
\begin{rema}
 \eqref{coercive:Psi} is reminiscent of the coercive estimate established in \cite{Nguyen2023-coercivity} for the Dirichlet-Neuman operator $G[\eta_*]$:
\bq\label{coercive:DN}
(G[\eta_*]\chi, \chi)_{H^{-\mez}(\T^d), H^\mez(\T^d)}\ge c_G\| \chi\|_{H^\mez(\T^d)}^2.
\eq
Since the proof of \eqref{coercive:DN} in \cite{Nguyen2023-coercivity} is direct, the constant $c_G$ is explicit and depends only on $(d, c^0, \| \eta_*\|_{W^{1, \infty}})$. On the other hand, the above proof by contradiction of \cref{coercivity for Psi eta} only yields a dependence of $c_{\Psi_\gamma[\eta_*]}$  on $\eta_*$. However, this will suffice for our stability proof in subsequent sections. 
\end{rema}
%%%%%%%%%%%%%%%%%%%%%%%%%%%%%%%%%%%%%%%%%%%%%%%%%%%%%
\subsubsection{Commutator estimates}
We recall from \cref{bound for Psi and Psi inverse} that $\Psi_\gamma[\eta]$ has order $-1$. The commutator estimate \eqref{commutator bound} below shows  a gain of a full derivative when commuting $\Psi_\gamma[\eta]$ with partial derivatives.
\begin{prop}\label{commutator estimate for Psi eta}
    Let  $(d+1)/2< s \in \Nn$, and let $\eta\in H^{s+\tdm}(\T^d)$ satisfy \eqref{eta lower bound}. Then for all  $\alpha\in \Nn^d$ and $\sigma\in \Rr$ such that  $\sigma \in [\tdm, s+\mez - |\alpha|]$, we have
    \bq \label{commutator bound}
    \|[\Psi_\gamma[\eta],\partial^\alpha]\chi\|_{H^{\sigma}(\T^d)}\leq C(\|\eta\|_{H^{s+\tdm}(\T^d)})\|\eta\|_{H^{s+\tdm}(\T^d)}\|\chi\|_{H^{\sigma+|\alpha|-2}(\T^d)}
    \eq
    whenever the right-hand side is finite, where $C: \Rr^+ \to \Rr^+$ depends only on $(d,b,c^0,s,\sigma,|\alpha|)$.
\end{prop}
\begin{proof}
    We prove by induction on $|\alpha|$. For the case of $|\alpha|=1$, we assume $\sigma \in [\tdm, s- \mez]$,  fix any  $i\in\{1,...,d\}$ and $\chi\in H^{\sigma-1}(\T^d)$. Since $\sigma -\mez \in [1, s- 1]$, by \cref{reg for stokes with stress} and \cref{rema:interpolation}, \eqref{sys: hom stokes with normal stress} has a unique solution   $(v, p)\in  H^{\sigma+\mez}(\Omega_\eta)\times H^{\sigma-\mez}(\Omega_\eta)$. Similarly, since $\sigma-\tdm \in [0, s-\tdm]$, the problem (\ref{sys: hom stokes with normal stress}) with $\chi$ replaced with $\partial_i \chi$ has a unique solution $(u, q)\in H^{\sigma-\mez}(\Omega_\eta)\times H^{\sigma-\tdm}(\Omega_\eta)$. Moreover, we have 
    \begin{align}\label{commute:vpest}
    \| v\|_{H^{\sigma+\mez}(\Omega_\eta)}+\| p\|_{H^{\sigma-\mez}(\Omega_\eta)}\le C(\|\eta\|_{H^{s+\tdm}})\| \chi\|_{H^{\sigma-1}(\T^d)},\\ \label{commute:uqest}
     \| u\|_{H^{\sigma-\mez}(\Omega_\eta)}+\| q\|_{H^{\sigma-\tdm}(\Omega_\eta)}\le C(\|\eta\|_{H^{s+\tdm}})\| \chi\|_{H^{\sigma-1}(\T^d)}.
   \end{align}
 We fix a smooth function $\ld(z): \Rr\to \Rr$ such that  $\ld(z)=0$ for $z<\frac{-c^0}{2}$  and $\ld(z)=1$ for $z>\frac{-c^0}{3}$.   Then we set 
 \[
 \psi_0(x, z)=\ld(z)e^{z|D|}\p_i\eta(x),\quad \psi(x, y) =\psi_0(x, y-\eta(x)),
 \]
 so that $\psi(x, \eta(x))=\p_i\eta(x)$. Moreover, we have
 \bq\label{est:psi:p1eta}
 \| \psi\|_{H^r(\Omega_\eta)}\le C\| \p_i\eta\|_{H^{r-\mez}(\T^d)}\quad \forall r\in [0, s+1].
 \eq
The key idea is to approximate   $(u, q)$ by 
        \bq
    \begin{split}
u_a(x,y):=& \partial_i v(x,y) +\psi(x, y)\partial_y v(x,y), \\
      q_a(x,y):=& \partial_i p(x,y) +\psi(x, y)\partial_y p(x,y).
    \end{split}
    \eq
    Then $(u_a,q_a)$ solve the problem (\ref{sys:stokes with stress}) with 
    \bq\label{vr:g}
    g=\nabla \cdot (\p_iv+\psi   \partial_y v)=\na \psi \cdot \p_y v,
    \eq
    \bq\label{vr:f}
    \begin{aligned}
        f&= -\gamma \partial_1 (\psi  \partial_y v)- \Delta (\psi \partial_y v ) +\na \dv (\psi  \partial_y v)+ \nabla (\psi  \partial_y p)\\
        &=\psi \p_y\underbrace{\big(-\gamma \partial_1 v- \Delta v  + \na p\big)}_{=0}\\
        &\qquad  -\gamma  \partial_1\psi  \partial_y v- \Delta \psi \partial_y v -2\na  \partial_yv \na \psi +\na(\na \psi\cdot \p_y v)+ \na \psi  \partial_y p\\
        &= -\gamma  \partial_1\psi  \partial_y v- \Delta \psi \partial_y v -2\na  \partial_yv \na \psi +\na(\na \psi\cdot \p_y v)+ \na \psi  \partial_y p,
    \end{aligned}
    \eq 
    and 
\bq\label{form:k:1}
            k(x)=\Bigl(\partial_i p I+ \partial_i \eta(x)\partial_y p I - \Dd\partial_i v - \partial_i \eta(x)\Dd \partial_y v-\nabla \psi \otimes \partial_y v- \partial_y v \otimes\nabla \psi\Bigr)\vert_{\Sigma_{\eta}}\cN(x). 
            \eq
           % = \Bigl( \partial_i\bigl(p(x,\eta(x)) \bigr) I - \partial_i \bigl( \Dd v(x,\eta(x))  \bigr)  \Bigr)\cN(x) + \Bigl(\nabla \partial_i \eta(x) \otimes \partial_y v(x,\eta(x)) + \partial_y v(x,\eta(x)) \otimes \nabla \partial_i \eta(x)  \Bigr)\cN(x) .   
    On the other hand, taking $\partial_i$ of $S(p,v)\cN = \chi \cN$ yields
    \[
    \Bigl(\p_i p+\p_i\eta \p_yp-\Dd\p_i v-\p_i\eta\Dd\p_y v \Bigl)\vert_{\Sigma_{\eta}}\cN+(p I-\Dd v)\p_i \cN=\p_i\chi \cN+\chi \p_i \cN.
    \]
    Comparing this with \eqref{form:k:1}, we find
    \[
    k=\p_i\chi \cN+\chi \p_i \cN-(p I-\Dd v)\p_i \cN - \bigl(\nabla \psi \otimes \partial_y v+ \partial_y v \otimes \nabla \psi \bigr)\vert_{\Sigma_{\eta}}\cN.
    \]
We set 
\bq\label{vr:k}
k_r=\chi \p_i \cN-(p I-\Dd v)\p_i \cN -\bigl(\nabla \psi \otimes \partial_y v+ \partial_y v \otimes \nabla \psi \bigr)\vert_{\Sigma_{\eta}}\cN,
\eq
so that $k_r=k-\p_i \chi \cN$ is the difference between that normal stress between $(u_a, q_a)$ and $(u,  q)$. Let $(u_r, q_r)$ be the solution of \eqref{sys:stokes with stress} with  $f$, $g$, and   the normal-stress data $k_r$  given by \eqref{vr:f}, \eqref{vr:g}, and \eqref{vr:k}. By the linearity of \eqref{sys:stokes with stress}, we have 
\bq\label{urua}
u_r=u_a-u,\quad q_r=q_a-q.
\eq
We now proceed to estimate $(u_r, q_r)$ using  \cref{reg for stokes with stress}. For $\sigma\in [\tdm,s+\mez]$ and $s>(d+1)/2$, we can apply \cref{product estimate for domain} and the bound \eqref{est:psi:p1eta} (with $r=s+1$) to have
    \bq\label{bound for the forcing f}
\begin{aligned}
        \|f\|_{H^{\sigma-\tdm}(\Omega_\eta)} &\le C \| \partial_1\psi\|_{H^{s}(\Omega_\eta)}\| \partial_y v\|_{H^{\sigma-\mez}(\Omega_\eta)}+C \| \Delta \psi\|_{H^{s-1}(\Omega_\eta)}\| \partial_y v\|_{H^{\sigma-\mez}(\Omega_\eta)}  \\
        &\qquad +C\| \na \psi\|_{H^s(\Omega_\eta)}\| \na  \partial_yv\|_{H^{\sigma-\tdm}(\Omega_\eta)}  +C\| \na \psi\|_{H^s(\Omega_\eta)}\| \p_y v\|_{H^{\sigma-\mez}(\Omega_\eta)}\\
        &\qquad+C\| \na\psi\|_{H^{s}(\Omega_\eta)}\| \partial_y p\|_{H^{\sigma-\tdm}(\Omega_\eta)} \\
        &\le C (\|\eta\|_{H^{s+\tdm}(\T^d)})\|\eta\|_{H^{s+\tdm}(\T^d)}\bigl(\| v\|_{H^{\sigma+\mez}(\Omega_\eta)} + \| p\|_{H^{\sigma-\mez}(\Omega_\eta)}\bigr)\\
        &\leq  C(\|\eta\|_{H^{s+\tdm}(\T^d)})\|\eta\|_{H^{s+\tdm}(\T^d)} \|\chi\|_{H^{\sigma-1}(\T^d)},
    \end{aligned}
    \eq
    where $C: \Rr_+\to \Rr_+$ depends only on $(d, b, c^0, s,\sigma)$. We similarly obtain 
    \bq
    \begin{aligned}
    \| g\|_{H^{\sigma-\mez}(\Omega_\eta)}&\le  C (\|\eta\|_{H^{s+\tdm}(\T^d)})\|\na \psi \|_{H^{s}(\Omega_\eta)}\| v\|_{H^{\sigma+\mez}(\Omega_\eta)} \\
    & \le  C (\|\eta\|_{H^{s+\tdm}(\T^d)})\|\eta\|_{H^{s+\tdm}(\T^d)} \|\chi\|_{H^{\sigma-1}(\T^d)}
    \end{aligned}
    \eq
  and    \bq\label{bound for the stress k bar}
    \begin{aligned}
        \|k_r\|_{H^{\sigma-1}(\T^d)}&\leq C\bigl(\| \chi\|_{H^{\sigma-1}} +\| (p I-\Dd v)\vert_{\Sigma_{\eta}}\|_{H^{\sigma-1}}\bigr) \| \p_i \cN\|_{H^{s-\mez}}\\
        &\qquad -C\| \nabla \psi\vert_{\Sigma_{\eta}}\|_{H^{s-\mez}}\| \partial_y v\vert_{\Sigma_{\eta}}\|_{H^{\sigma-1}}\|\cN \|_{H^{s+\mez}}\\
        & \le  C(\|\eta\|_{H^{s+\tdm}(\T^d)})\|\eta\|_{H^{s+\tdm}(\T^d)}\|\chi\|_{H^{\sigma-1}}.
    \end{aligned}
    \eq
    Inserting the above estimates into \eqref{high reg estimate for stokes with stress}  yields
    \bq\label{est:urqr}
       \| u_r\|_{H^{\sigma+\mez}(\Omega_\eta)}+\| q_r\|_{H^{\sigma-\mez}(\Omega_\eta)}\le  C (\|\eta\|_{H^{s+\tdm}(\T^d)})\|\eta\|_{H^{s+\tdm}(\T^d)}\|\chi\|_{H^{\sigma-1}(\T^d)}.
    \eq
    We note that by virtue of \cref{rema:interpolation},  $\sigma$ need not be an integer in \eqref{est:urqr} since $\sigma-\mez\ge 1$.  Next, using \eqref{urua} we write 
     \begin{multline*}
        [\Psi_\gamma[\eta],\partial_i]\chi = u\vert_{\Sigma_\eta}\cdot \cN - \partial_i (v\vert_{\Sigma_\eta}\cdot \cN) \\
        = u\vert_{\Sigma_\eta}\cdot \cN -(\underbrace{\p_i v+\p_i\eta\p_y v)\vert_{\Sigma_\eta}}_{=u_a\vert_{\Sigma_\eta}}\cN-v\vert_{\Sigma_\eta}\cdot \p_i\cN
         = - u_r\vert_{\Sigma_\eta} \cdot \cN - v\vert_{\Sigma_\eta}\cdot \partial_i \cN.
     \end{multline*}
The $H^\sigma(\T^d)$ norms of $v\vert_{\Sigma_\eta}$ and $u_a\vert_{\Sigma_\eta}$  are bounded by means of \eqref{commute:vpest} and \eqref{est:urqr}. Since $\p_i\cN \in H^{s-\mez}(\T^d)$,  for $\sigma \in [\tdm, s-\mez]$ and $s>(d+1)/2$, \cref{product estimate on torus}  yields
    \bq\label{cmt:Psi:base}
    \begin{split}
        \|[\Psi_\gamma[\eta],\partial_i]\chi\|_{H^{\sigma}(\T^d)}%&\leq  c(d,\sigma)\Bigl[(\|\eta\|_{W^{\lceil \sigma +1 \rceil,\infty}(\T^d)}+1)\|v_r\|_{H^{\sigma+\mez}(\Omega_\eta)} +  \|\eta\|_{W^{\lceil \sigma +2 \rceil,\infty}(\T^d)}  \|v\|_{H^{\sigma+\mez}(\Omega_\eta)} \Bigr] \\
        &\leq  C (\|\eta\|_{H^{s+\tdm}(\T^d)})\|\eta\|_{H^{s+\tdm}(\T^d)}\|\chi\|_{H^{\sigma-1}(\T^d)}.
    \end{split}
    \eq
Here is the place where  the most stringent restriction  $\sigma\le s-\mez$ is needed. We have finished the proof of \eqref{commutator bound} for $|\alpha|=1$.

    Next, we assume  that for some $k\ge 2$,  \eqref{commutator bound} holds  for all $\alpha\in  \Nn^d$ with $|\alpha|\le k-1$ and for all $\sigma \in [\tdm, s+\mez-|\alpha|]$. To prove  \eqref{commutator bound} for  $|\alpha| = k$ and $\sigma \in [\tdm, s+\mez-|\alpha|]$, we write  $\partial^\alpha = \partial^\beta \partial_i$ for some $i\in \{1,\dots, d\}$ and $|\beta|=k-1$, so that
        \[
        [\Psi_\gamma[\eta],\partial^\alpha]\chi =   [\Psi_\gamma[\eta],\partial_i ]\partial^\beta\chi + \partial_i[\Psi_\gamma[\eta],\partial^\beta]\chi.
    \]
   Since $\sigma+1\in [\frac{5}{2}, s+\mez-|\beta|]$,  the induction   hypothesis implies
    \begin{align*}
        \|[\Psi_\gamma[\eta],\partial^\alpha ]\chi\|_{H^{\sigma}(\T^d)} &\leq   \|[\Psi_\gamma[\eta],\partial_i ]\partial^\beta\chi\|_{H^{\sigma}(\T^d)}  + \|[\Psi_\gamma[\eta],\partial^\beta]\chi\|_{H^{\sigma+1}(\T^d)} \\
       &  \leq  C(\|\eta\|_{H^{s+\tdm}(\T^d)})\|\eta\|_{H^{s+\tdm}(\T^d)}\Bigl(\|\partial^\beta\chi\|_{{H}^{\sigma-1}(\T^d)} + \|\chi\|_{{H}^{\sigma+1+|\beta|-2}(\T^d)},  \Bigr) \\
       & \le C(\|\eta\|_{H^{s+\tdm}(\T^d)})\|\eta\|_{H^{s+\tdm}(\T^d)}\|\chi\|_{{H}^{\sigma+|\alpha|-2}(\T^d)}
    \end{align*}
    which completes the proof of the induction.
\end{proof}

\section{Existence of large traveling waves}\label{sec:construction}
\subsection{The capillary-gravity operator}
Recalling the mean curvature of operator $\cH(f) = -\dv(\frac{\nabla f}{\sqrt{1+|\nabla f|^2}})$, we consider the capillary-gravity operator $\sigma \cH+gI$, where $\sigma$, $g>0$. Estimates for the linearized operator and the incurred remainder are given the following Proposition. 
\begin{prop}\label{prop: T eta bijection}
     Let $s>(d-1)/2$ and $\eta\in H^{s+\tdm}(\T^d)$. Let $T_\eta$ be defined by   
    \bq\label{form:Teta}
    T_\eta f =g f-\sigma \dv (\frac{1}{(1+|\nabla \eta|^2)^\mez} \nabla f - \frac{\langle\nabla \eta , \nabla f\rangle}{(1+|\nabla \eta|^2)^\tdm} \nabla \eta).
    \eq
    Then for any $\mu\in [1,s+\tdm]$, $T_\eta : H^{\mu}(\T^d) \to H^{\mu- 2}(\T^d)$ is an isomorphism with 
    \bq\label{bound:Teta}
    \|T_{\eta}\|_{H^{\mu}(\T^d)\to H^{\mu-2}(\T^d)} \leq C(\|\nabla \eta\|_{H^{s+\mez}(\T^d)}),
    \eq
     and
     \bq\label{bound:Teta:inv}
     \|T^{-1}_{\eta}\|_{H^{\mu-2}(\T^d)\to H^{\mu}(\T^d)} \leq C(\|\eta\|_{H^{s+\tdm}(\T^d)}),
     \eq
     where $C:\Rr^+ \to \Rr^+$ depends only on $(d,s,\mu, g, \sigma)$. 
\end{prop}
\begin{proof}
     $T_\eta$ is the  linear second-order operator 
    \bq\label{def: a_ij}
   T_\eta f= gf- \sigma\dv(a_{ij}\partial_j f),\quad a_{ij} = \frac{\delta_{ij}}{(1+|\nabla \eta|^2)^\mez} - \frac{\partial_i \eta \partial_j \eta}{(1+|\nabla \eta|^2)^\tdm}
    \eq
which is elliptic since 
    \bq\label{ellipticity of aij}
    (1+|\nabla \eta|^2)^{-\tdm}|\xi|^2 \leq a_{ij}\xi_i\xi_j \leq (1+|\nabla \eta|^2)^{-\mez}|\xi|^2\quad  \forall \xi\in \Rr^n.
    \eq 
    By \cref{prop: composition regularity} and the fact that $H^{s+\mez}(\T^d)$ is an algebra  embedding into $L^\infty(\T^d)$, we have
    \bq\label{bound:a_ij}
        \|a_{ij}\|_{H^{s+\mez}(\T^d)}\leq C(\|\nabla \eta\|_{L^\infty(\T^d)}) (\|\nabla \eta\|_{H^{s+\mez}(\T^d)}+1).
    \eq
    Then applying \cref{product estimate on torus} to \eqref{def: a_ij} yields \eqref{bound:Teta}.

    Next, we prove that  $T_\eta$ has a continuous inverse from $H^{\mu}(\T^d)$ to  $H^{\mu-2}(\T^d)$.  The case $\mu =1$ and the existence of the inverse  follow from Lax-Milgram's theorem and  \eqref{bound:a_ij}.  To propagate to higher regularity, we shall prove that for any $\eps\in (0,s- \frac{d-1}{2})$ with $\eps\leq \mez$,  if \eqref{bound:Teta:inv} holds for some $\mu \in [1,s+\tdm - \eps]$, then \eqref{bound:Teta:inv} also holds for $\mu + \eps$. To this end, first suppose $d\geq 2$; we will treat the case $d=1$ separately at the end of the proof. For any $F\in H^{\mu + \eps- 2}(\T^d)$, we will prove regularity for the elliptic problem $T_\eta f = F$ for $f$ via paralinearization. 
    Assuming \eqref{bound:Teta:inv} holds for $\mu$, let $f = T_{\eta}^{-1} F \in H^{\mu}(\T^d)$. We have
    \[
       F= T_\eta f = gf - \sigma \bigl(a_{ij} \partial_i \partial_j f + \partial_i a_{ij} \partial_j f \bigr) =  gf + \sigma T_m f -\sigma  (a_{ij} - T_{a_{ij}})\partial_i \partial_j f -\sigma \partial_i a_{ij} \partial_j f,
    \]
    where $T_m$ is the paradifferential operator (see \eqref{eq.para}) associated with the symbol 
    \[
        m(x,\xi) = a_{ij}\xi_i\xi_j,
    \]
    where we have used the Einstein notation for summation. 
    This allows us to rewrite the elliptic problem as
    \bq\label{T eta paralinearized}
        \sigma T_m f = F', \qquad F' = F - gf + \sigma (a_{ij} - T_{a_{ij}})\partial_i \partial_j f +\sigma \partial_i a_{ij} \partial_j f.
    \eq
    For $s>(d-1)/2$, $d\ge 2$, and $\mu \in [1,s+\tdm - \eps]$, we have $s>\mez$, $\mu+\eps -2 \leq s +\mez$, $s+\mez + \mu -2 > 0$, and $s+\mez + \mu -2 > \mu+\eps -2 + d/2$, where the last inequality uses the fact that $\eps<s-\frac{d-1}{2}$.  Hence, we can apply  \eqref{lemPa1} with $s_0 = \mu+\eps-2$, $s_1= s+\mez$, and $s_2 = \mu-2$ to obtain
    \[
       \| (a_{ij} - T_{a_{ij}})\partial_i \partial_j f \|_{H^{\mu+\eps - 2}}\leq C\|a_{ij}\|_{H^{s+\mez}} \|\partial_i\partial_j f\|_{H^{\mu-2}} \leq C(\|\nabla \eta\|_{H^{s+\mez}(\T^d)})\|F\|_{H^{\mu-2}(\T^d)},
    \]
    where in the last inequality we appeal to \eqref{bound:a_ij} and \eqref{bound:Teta:inv}. Similarly, we can apply  \cref{product estimate on torus} to have
    \[
        \|\partial_i a_{ij} \partial_j f\|_{H^{\mu+\eps-2}} \leq C\|\partial_i a_{ij} \|_{H^{s-\mez}}\|\partial_j f\|_{H^{\mu-1}} \leq C(\|\nabla \eta\|_{H^{s+\mez}(\T^d)})\|F\|_{H^{\mu-2}(\T^d)}.
    \]
    The preceding estimates imply that $F' \in H^{\mu +\eps-2}(\T^d)$  and 
        \bq
        \|F'\|_{H^{\mu+\eps -2}} \leq C(\|\nabla \eta\|_{H^{s+\mez}(\T^d)})\|F\|_{H^{\mu+\eps-2}(\T^d)}.
    \eq
    Since $T_1 f=f-(2\pi)^{-d}\hat{f}(0)$, it follows from  \eqref{T eta paralinearized}  that 
    \bq\label{paraeq:f}
    \begin{aligned}
        \sigma f &= \sigma T_{m^{-1}m} f +\sigma(2\pi)^{-d}\hat{f}(0)\\
        &= \sigma (T_{m^{-1}m} - T_{m^{-1}} T_m) f + \sigma T_{m^{-1}} T_m f+\sigma(2\pi)^{-d}\hat{f}(0)\\
        & = \sigma (T_{m^{-1}m} - T_{m^{-1}} T_m) f  + T_{m^{-1}} F'+\sigma(2\pi)^{-d}\hat{f}(0).
    \end{aligned}
    \eq
    In the notation given by \cref{defiGmrho}, we have $m\in \Gamma^2_{\eps}(\T^d)$, and using the ellipticity \eqref{ellipticity of aij}, $m^{-1}\in \Gamma^{-2}_\eps(\T^d)$. Moreover, we have the bounds
    \bq
        M^2_\eps(m), M^{-2}_\eps(m^{-1})\leq C(\|\nabla \eta\|_{L^\infty(\T^d)}) (\|\nabla \eta\|_{H^{s+\mez}(\T^d)}+1).
    \eq
  Using the symbolic calculus in \cref{theo:sc0}, we deduce from \eqref{paraeq:f} that 
    \begin{multline}
        \|f\|_{H^{\mu +\eps}} \leq C\|T_{m^{-1}m} - T_{m^{-1}} T_m\|_{H^{\mu}\to H^{\mu + \eps}} \|f\|_{H^{\mu}} + C\|T_{m^{-1}}\|_{H^{\mu+\eps -2} \to H^{\mu +\eps}}\|F'\|_{H^{\mu+\eps -2}} +C\|f\|_{L^2} \\
        \leq C(\|\nabla \eta\|_{H^{s+\mez}(\T^d)})\|F\|_{H^{\mu+\eps-2}(\T^d)}.
    \end{multline}
    This finishes the proof for the case $d\geq 2$.

    Finally, we turn to the case $d=1$. Note that in this case, we have
    \bq
        T_\eta f = gf  - \sigma (af')', \qquad a = \frac{1}{(1 + |\eta'|^2)^{\mez}} - \frac{|\eta'|^2}{(1+|\eta'|^2)^{\tdm}} = \frac{1}{(1+|\eta'|^2)^\tdm}.
    \eq
Thus the  elliptic problem becomes    $ -\sigma(af')' = F - gf =:F_1$. Denoting by $\p_x^{-1} F_1$ the mean-zero antiderivative of $F_1$, we obtain  $-f' = (\sigma a)^{-1}\p_x^{-1}  F_1$. Then applying  \cref{product estimate on torus}  yields 
    \bq
        \|f'\|_{H^{\mu+\eps - 1}}  \le \sigma^{-1}\| a^{-1}\|_{H^{s+\mez}} \|\p_x^{-1}F_1\|_{H^{\mu+\eps -1}} \leq C(\|\eta'\|_{H^{s+\mez}})\|F_1\|_{H^{\mu+\eps -2}},
    \eq
    which in turn combined with \eqref{bound:Teta:inv} yields the result for $\mu +\eps$. This completes the proof.
\end{proof}

\begin{prop}\label{linearizing mean curvature at eta}
    Let $s>(d-1)/2$ and $\eta\in H^{s+\tdm}(\T^d)$. We have
    \bq\label{expand:cH}
    (\sigma \cH+gI)(\eta+f) -(\sigma \cH+gI)(\eta) = T_\eta f + R_\eta(f),
    \eq
    where  $T_\eta$ is defined in \eqref{form:Teta}, and we have the following estimates:
    \bq\label{bound:Teta tame}
    \|T_{\eta}f\|_{H^{s-\mez}}\leq C(\|\na\eta\|_{L^\infty})\big(\| \na f\|_{H^{s+\mez}}+\| \na \eta\|_{H^{s+\mez}}\| \na f\|_{L^\infty}\big),
     \eq
      \begin{multline}\label{contraction for T in eta}
  \|T_{\eta_1}f - T_{\eta_2}f\|_{H^{s-\mez}}\leq C(\|\na \eta_1\|_{L^\infty},\|\na\eta_2\|_{L^\infty})\Big(\|\na\eta_1-\na\eta_2\|_{L^\infty}\|\na f\|_{H^{s+\mez}}\\
 +\|\na\eta_1-\na\eta_2\|_{H^{s+\mez}}\|\na f\|_{L^\infty}\Big),
    \end{multline}
     \begin{multline}\label{bound for R eta f}
 \|R_\eta(f)\|_{H^{s-\mez}(\T^d)}\leq C(\|\na \eta\|_{L^\infty},\| \na f\|_{L^\infty}) \Big(\|\na f\|_{L^\infty}\| \na f\|_{H^{s+\mez}} +\| \na \eta\|_{H^{s+\mez}}\| \na f\|_{L^\infty}^2\Big),
    \end{multline} 
    and
  \begin{multline}\label{contraction for R eta f}
     \|R_{\eta}(f_1) - R_{\eta}(f_2)\|_{H^{s-\mez}}\leq C\big(\|\na \eta\|_{L^\infty},\| \na f_1\|_{L^\infty},\| \na f_2\|_{L^\infty}\big) \\
    \cdot \Bigl((\|\na f_1\|_{H^{s+\mez}}+\|\na f_2\|_{H^{s+\mez}})\|\na(f_1-f_2)\|_{L^\infty}  + \|\na (f_1-f_2)\|_{H^{s+\mez}}\Bigr).
    \end{multline} 
    The above functions $C$'s depend only on $(d, s, g, \sigma)$. 
\end{prop}
\begin{proof}
   We write  $\cH(\eta)=-\dv F(\na \eta)$, where  $F(z)=\frac{z}{\sqrt{1+|z|^2}}$. Since $(\na F(z))_{ij}=\frac{\delta_{ij}}{\sqrt{1+|z|^2}} -\frac{z_iz_j}{(1+|z|^2)^\tdm}$, we have
    \[
   \cH(\eta+f)-\cH(\eta)=-\dv\big(\na F(\na\eta)\na f\big)-\dv\int_0^1(1-t)D^2F(\na \eta+t\na f)\na f \cdot \na f dt.  
    \]
    This implies the formula \eqref{form:Teta} for $T_\eta$ in \eqref{expand:cH}. Moreover, since $D^2F(0)=0$ and  $H^{s+\mez}(\T^d)$ is an algebra for $s>(d-1)/2$, the $H^{s-\mez}$ estimates \eqref{bound:Teta tame}-\eqref{contraction for R eta f} for $T_\eta$ and $R_\eta$ then follow from \cref{prop: composition regularity}.
\end{proof}

Next, we prove commutator estimates for the linearized capillary-gravity operator $T_\eta$ with partial derivatives.
\begin{prop}\label{commutator estimate for T eta}
Let $\eta \in H^{s+\tdm}(\T^d)$ with $s>(d+1)/2$. Then for all $\alpha\in \Nn^d$ and $\sigma \in \Rr$ satisfying 
\bq\label{cd:cmt:Teta}
\sigma+|\alpha|\le s-\mez\quad \text{and}\quad s+\mez+\sigma>0,
\eq
 we have
    \bq \label{commutator bound for T eta}
    \|[T_\eta,\partial^\alpha]f\|_{H^{\sigma}(\T^d)}\leq C(\|\eta\|_{H^{s+\tdm}(\T^d)})\|\eta\|_{H^{s+\tdm}(\T^d)}\|f\|_{H^{\sigma+|\alpha|+1}(\T^d)}
    \eq
provided  the right-hand side is finite. Here $C: \Rr^+ \to \Rr^+$ depends only on $(d,s,\sigma,|\alpha|)$.
\end{prop}
\begin{proof}
Since
\[
 [\partial^\alpha, T_\eta]f=\sum_{0<\beta\le \alpha} \dv\big(\p^\beta a_{ij}\p_j \p^{\alpha-\beta}f\big),
 \]
 we have
 \[
 \|  [\partial^\alpha, T_\eta]f\|_{H^{\sigma}}\le \sum_{0<\beta\le \alpha} \| \p^\beta a_{ij}\p_j \p^{\alpha-\beta}f\|_{H^{\sigma+1}}.
 \]
Under the condition \eqref{cd:cmt:Teta},  \cref{product estimate on torus} is applicable with $s_0=\sigma+1$, $s_1=s+\mez-|\beta|$, $s_2=\sigma+1+|\beta|-1$, implying 
 \[
 \| \p^\beta a_{ij}\p_j \p^{\alpha-\beta}f\|_{H^{\sigma+1}}\le C \| \p^\beta a_{ij}\|_{H^{s+\mez-|\beta|}}\|\p_j \p^{\alpha-\beta}f\|_{H^{\sigma+1+|\beta|-1}} \le C\| a_{ij}\|_{H^{s+\mez}}\| f\|_{H^{\sigma+|\alpha|+1}}.
 \]
 Finally, applying the nonlinear estimate \eqref{nonl:est} to $a_{ij}$ we obtain \eqref{commutator bound for T eta}. 
\end{proof}

Finally, we prove the invertibility of the  gravity-capillary operator in Sobolev spaces.
\begin{prop}\label{prop:wp:gc}
    Let $s> (d+1)/2$. For any $\phi\in H^{s-\mez}(\T^d)\cap C^{1}(\T^d)$, there exists a unique $\eta\in H^{s+\tdm}(\T^d) $ such that  $(gI + \sigma \cH) \eta = \phi$, and $\eta$ satisfies
        \bq\label{est:eta phi}
        \|\eta\|_{H^{s+\tdm}}\leq C(\|\phi\|_{H^{s-\mez}} + \|\phi\|_{C^{1}} )
    \eq
 for some $C: \Rr_+\to \Rr_+$ depending only on $(d,s,\sigma,g)$. 
\end{prop}
\begin{proof}
Let $\eps\in (0, \min\{\mez, s-(d+1)/2\})$. Appealing to invertibility of the capillary-gravity operator in H\"older spaces (see Proposition 3.1 in \cite{Nguyen2023-capillary}), and since $\phi \in C^{1}(\T^d)$, \eqref{est:eta phi} has a unique solution $\eta \in C^{2, \eps}(\T^d)$ which satisfies
    \bq\label{est: eta phi base}
        \|\eta\|_{C^{2,\eps}} \leq C(\|\phi\|_{C^{1}})\|\phi\|_{C^{1}},
    \eq
    where $C:\Rr_+ \to \Rr_+$ depends only on $(d, s, \sigma,g)$. Thus we need to upgrade the regularity of $\eta$ from $C^{2, \eps}(\T^d)$ to  $ H^{s+\tdm}(\T^d)$. This  will be achieved  by an induction on the Sobolev index.  Precisely, assuming that 
     \bq\label{est: eta phi induction}
        \|\eta\|_{H^\mu} \leq C(\|\phi\|_{H^{s-\mez}\cap C^{1}})\|\phi\|_{H^{s-\mez}\cap C^{1}},
    \eq
    for some  $\mu\in[2, s+\tdm -\eps]$,  we will prove that
    \bq\label{induction:gc}
        \|\eta\|_{H^{\mu+\eps}} \leq C(\|\phi\|_{H^{s-\mez}\cap C^{1}})\|\phi\|_{H^{s-\mez}\cap C^{1}},
    \eq
    
    Note that \eqref{est: eta phi induction} holds for $\mu=2$ in view of \eqref{est: eta phi base}. We will prove \eqref{induction:gc} using a paradifferential approach. Let us assume for simplicity that $\sigma=1$. Setting $a(z)= (1+|z|^2)^{-\mez}$, we write
    \begin{align*}
        (gI +  \cH)\eta &= g\eta - \dv(a(\nabla \eta) \nabla \eta) \\
        &=  g \eta - a(\nabla \eta)\Delta \eta - \nabla [a(\nabla \eta)] \cdot \nabla \eta \\
        &= g\eta + T_{a(\nabla \eta) |\xi|^2} \eta + [T_{a(\nabla \eta)} - a(\nabla \eta)]\Delta \eta - T_{\nabla \eta} \cdot \nabla [a(\nabla \eta)] + [T_{\nabla \eta} - \nabla \eta] \cdot \nabla [a(\nabla \eta)].
    \end{align*}
 Next, we paralinearize  $T_{\nabla \eta} \cdot \nabla [a(\nabla \eta)] $, which together with $T_{a(\nabla \eta) |\xi|^2} \eta$ is the highest-order term. Since $\na a(z)=-a(z)^3z$, we have
    \begin{align*}
        T_{\nabla \eta} \cdot \nabla [a(\nabla \eta)] &= T_{\nabla \eta} \cdot (\nabla a)\circ (\nabla \eta)\Delta \eta \\
        &= -  T_{\nabla \eta} \cdot a(\nabla \eta)^3\Delta \eta \nabla \eta \\
        &=   T_{\nabla \eta} \cdot T_{a(\nabla \eta)^3 \nabla \eta |\xi|^2} \eta  +  T_{\nabla \eta} \cdot[T_{a(\nabla \eta)^3 \nabla \eta} - a(\nabla \eta)^3 \nabla \eta]\Delta \eta \\
        &= T_{a(\nabla \eta)^3 |\nabla \eta|^2 |\xi|^2} \eta + [ T_{\nabla \eta} \cdot T_{a(\nabla \eta)^3 \nabla \eta |\xi|^2} -  T_{a(\nabla \eta)^3 |\nabla \eta|^2 |\xi|^2}] \eta \\
        & \qquad \qquad \qquad \qquad \qquad  + T_{\nabla \eta} \cdot[T_{a(\nabla \eta)^3 \nabla \eta} - a(\nabla \eta)^3 \nabla \eta]\Delta \eta.
    \end{align*}
    Combining the preceding paralinearizations yields
    \bq\label{eq:Tm + F}
        (gI + \cH)\eta = T_{m(x,\xi)} \eta + F,
    \eq
    where $m(x,\xi)$ is the paradifferential symbol given by
    \bq
        m(x,\xi) =\bigl[a(\nabla \eta(x))  - a(\nabla \eta(x))^3 |\nabla \eta(x)|^2\bigr] |\xi|^2 = a(\nabla \eta(x))^3 |\xi|^2,
    \eq
    and $F$  encompasses all the lower order and low frequency terms
    \begin{multline*}
        F = g \eta  + [T_{a(\nabla \eta)} - a(\nabla \eta)]\Delta \eta + [T_{\nabla \eta} - \nabla \eta] \cdot \nabla [a(\nabla \eta)]  \\ 
        - [ T_{\nabla \eta} \cdot T_{a(\nabla \eta)^3 \nabla \eta |\xi|^2} -  T_{a(\nabla \eta)^3 |\nabla \eta|^2 |\xi|^2}] \eta - T_{\nabla \eta} \cdot[T_{a(\nabla \eta)^3 \nabla \eta} - a(\nabla \eta)^3 \nabla \eta]\Delta \eta. 
    \end{multline*}
Applying  \eqref{nonl:est}, \eqref{Bony3}, and \eqref{niS},  we obtain the following estimates:
   \begin{align*}
        \|[T_{a(\nabla \eta)} - a(\nabla \eta)]\Delta \eta\|_{H^{\mu+\eps-2}} & \leq K \|a(\nabla \eta)\|_{H^{\mu-1}}\|\Delta \eta\|_{C_*^{\eps -1}}\\
        &\leq C(\|\nabla \eta\|_{L^\infty})(\|\nabla \eta\|_{H^{\mu-1}}+1)\|\eta\|_{C^{1,\eps}} 
       % &\leq C(\|\phi\|_{H^{s-\mez}})\|\phi\|_{H^{s-\mez}},
    \end{align*}
    \begin{align*}
        \| [T_{\nabla \eta} - \nabla \eta] \cdot \nabla [a(\nabla \eta)]\|_{H^{\mu+\eps-2}} &\leq K\|\nabla \eta\|_{C_*^{\eps }}\|\nabla[a(\nabla \eta)]\|_{H^{\mu-2}} \\
        &\leq C(\|\nabla \eta\|_{L^\infty})(\|\nabla \eta\|_{H^{\mu-1}}+1)\| \eta\|_{C^{1,\eps}},
        %&\leq C(\|\phi\|_{H^{s-\mez}})\|\phi\|_{H^{s-\mez}},
    \end{align*}
    and 
    \begin{align*}
        \|T_{\nabla \eta} \cdot[T_{a(\nabla \eta)^3 \nabla \eta} - a(\nabla \eta)^3 \nabla \eta]\Delta \eta\|_{H^{\mu+\eps -2}} &\leq K\| \nabla \eta\|_{L^\infty}\|[T_{a(\nabla \eta)^3 \nabla \eta} - a(\nabla \eta)^3 \nabla \eta]\Delta \eta\|_{H^{\mu+\eps -2}} \\
        & \leq  K\| \nabla \eta\|_{L^\infty} \|a(\nabla \eta)^3 \nabla \eta\|_{H^{\mu-1}} \|\Delta \eta\|_{C_*^{\eps -1}}  \\
        &\leq C(\|\nabla \eta\|_{L^\infty})\|\nabla \eta\|_{H^{\mu-1}}\|\eta\|_{C^{1,\eps}}.
       % &\leq C(\|\phi\|_{H^{s-\mez}})\|\phi\|_{H^{s-\mez}}.
    \end{align*}
On the other hand, appealing to  \eqref{esti:quant2} yields     \begin{align*}
        \|[ T_{\nabla \eta} \cdot T_{a(\nabla \eta)^3 \nabla \eta |\xi|^2} -  T_{a(\nabla \eta)^3 |\nabla \eta|^2 |\xi|^2}] \eta\|_{H^{\mu+\eps -2}} \leq C(\|\eta\|_{C^{0, \eps}})\|\eta\|_{H^\mu}.
    \end{align*}
 It follows  from the above estimates, \eqref{est: eta phi base}, and \eqref{est: eta phi induction} that
    \bq\label{esti: F gaining epsilon derivative}
    \|F\|_{H^{\mu+\eps -2}} \leq C(\|\phi\|_{H^{s-\mez}\cap C^{1}})\|\phi\|_{H^{s-\mez}\cap C^{1}}.
    \eq
It is readily seen that  $m(x, \xi)$ is a second-order elliptic symbol
    and satisfies
    $$ M^2_\eps(m)+M^{-2}_\eps(m^{-1})\le C(\|\eta\|_{C^{1, \eps}}) \leq C(\|\phi\|_{C^{1}}). $$ 
    Therefore, by applying $T_{m^{-1}}$ to \eqref{eq:Tm + F} and then appealing to \eqref{esti:quant2} and \eqref{esti: F gaining epsilon derivative}, we deduce that
    \begin{align*}
        \|\eta\|_{H^{\mu+\eps}} = \|T_{m^{-1}m}\eta\|_{H^{\mu+\eps}} +C\| \eta\|_{L^2}&\leq \|[T_{m^{-1}m} - T_{m^{-1}}T_m]\eta\|_{H^{\mu+\eps}} + \|T_{m^{-1}}(\phi - F)\|_{H^{\mu+\eps}} +C\| \eta\|_{L^2}\\
        &\leq C(\|\phi\|_{H^{s-\mez}})(\|\eta\|_{H^\mu}   + \|\phi\|_{H^{\mu+\eps -2}}) \\
        &\leq C(\|\phi\|_{H^{s-\mez}\cap C^{1}})\|\phi\|_{H^{s-\mez}\cap C^{1}} .
    \end{align*}
This completes the proof of the desired estimate \eqref{induction:gc}.
    
\end{proof}

%%%%%%%%%%%%%%%%%%%%%%%%%%%%%%%%%%%%%%%%%%%%%%%%%%%%%%
\subsection{Existence of large traveling waves for the Stokes problem}\label{sec:tw:Stokes}
We consider the system \eqref{sys:main} with $\alpha=0$. We fix $s>(d+1)/2$ and suppose \bq
\phi \in H^{s-\mez}(\T^d)\cap C^{1}(\T^d).
\eq
By virtue of \cref{prop:wp:gc},  there exists a unique  $\eta_*\in H^{s+\tdm}(\T^d) $ such that
\bq\label{eq:eta*}
g\eta_*+\sigma \cH(\eta_*)=-\phi. 
\eq
If  $\eta_*$ also satisfies \eqref{eta lower bound}, then $(\gamma,v,p,\eta)= (0,0,0,\eta_*)$ is a trivial solution to \eqref{sys:main}. To guarantee \eqref{eta lower bound} we assume that 
\bq\label{lowerbound:phi}
 \min_{\T^d}-\phi> -gb. 
\eq
Since $\eta_*\in H^{s+\tdm}(\T^d)\subset C^2(\T^d)$, $\eta_*$ satisfies the following equation  pointwise:
      $$g\eta_*  = -\sigma\cH(\eta_*)-\phi=\sigma \Delta \eta_*(1+|\nabla\eta_*|^2)^{-\mez} +\sigma \nabla(1+|\nabla\eta_*|^2)^{-\mez} \cdot \nabla \eta_*-\phi. $$
  At any minimum point $x_0$ of $\eta_*$, we have  $\na \eta_*(x_0)=0$ and $\Delta \eta_*(x_0)\ge 0$, whence $g\eta_*(x_0)\ge -\phi(x_0)> -gb$ in view of \eqref{lowerbound:phi}.  Therefore, $\eta$ satisfies \eqref{eta lower bound} for some $c^0>0$.

We shall construct a solution  $(\gamma,v,p,\eta)$ of  \eqref{sys:main}  with  $\gamma$ small and $\eta= \eta_* +f$ a perturbation of $\eta_*$, where 
\bq\label{cd:f:c0}
\|f\|_{H^{s+\tdm}(\T^d)} < r(d, c^0, s)<1
\eq
 is sufficiently small so that $\min_{\T^d}(\eta+b)\ge \frac{c^0}{2}$.

In terms of the linear normal-stress to normal-Dirichlet operator $\Psi_\gamma$,  we observe that   \eqref{sys:main}  is equivalent to 
\bq\label{eq:tw:10}
-\gamma \partial_1 (f+ \eta_*)= \Psi_0[\eta]\big((\sigma\cH+gI)(\eta_*+f) - (\sigma \cH+gI)\eta_* \big),
\eq
where we have used \eqref{eq:eta*} to replace $\phi$. Using the expansion \eqref{expand:cH} for $\sigma \cH+gI$, we find 
\bq
\begin{aligned}
\eqref{eq:tw:10}& \iff -\gamma \partial_1 (f+ \eta_*)=\Psi_0[\eta_*]\big(T_{\eta_*} f+R_{\eta_*}(f))+ \big(\Psi_0[\eta_*+f]-\Psi_0[\eta_*]\big)\big(T_{\eta_*} f+R_{\eta_*}(f))\\
& \iff f= G_\gamma(f),
\end{aligned}
\eq
where 
\bq\label{def:Ggamma}
G_\gamma(f):=T_{\eta_*}^{-1} \Psi_0[\eta_*]^{-1}\Bigl(-\gamma \partial_1(\eta_*+f) - (\Psi_0[\eta_* +f] - \Psi_0[\eta_*])\bigl(T_{\eta_*}(f)+R_{\eta_*}(f)\bigr) \Bigr)-T_{\eta_*}^{-1}R_{\eta_*}(f).
\eq
 We have used the invertibility of $\Psi_\gamma[\eta_*]$ and $T_{\eta_*}$, established \cref{bound for Psi and Psi inverse} and \cref{prop: T eta bijection}. We are thus led to proving the existence of a fixed point of the nonlinear map $G_\gamma$. 
\begin{theo}\label{large traveling wave for Stokes}
    Let $(d+1)/2<s\in \Nn$. There exists $\delta_S>0$ depending on $(d,b,c^0, s, g, \sigma,\|\eta_*\|_{H^{s+\tdm}(\T^d)})$ such that for any $0< \delta< \delta_S$, there exists $\gamma_\delta>0$ such that for all $|\gamma|< \gamma_\delta$, $G_\gamma$ maps $\overline{B_{H^{s+\tdm}(\T^d)}(0,\delta)}$ to itself and  is a contraction.
\end{theo}
\begin{proof} We appeal to the Banach fixed point theorem in  $\overline{B_{H^{s+\tdm}(\T^d}(0, \delta)}$, where $\delta \le r(d, c^0, s)$ is a small number to be determined. We denote 
\[
h:=-\gamma \partial_1(\eta_*+f) - (\Psi_0[\eta_* +f] - \Psi_0[\eta_*])\big(T_{\eta_*}(f)+R_{\eta_*}(f)\big).
\]
Using \cref{linearization and contraction for Psi} for the contraction of $\Psi_0[\eta_*+f]-\Psi_0[\eta_*])$ (with $\sigma=s$)  and using the bounds \eqref{bound:Teta} and \eqref{bound for R eta f} for $T_{\eta_*}f$ and $R_{\eta_*}(f)$, we find that $h\in  \rH^{s+\mez}(\T^d)$ and
\bq
\| h\|_{H^{s+\mez}(\T^d)}\le C(|\gamma|+\| f\|_{H^{s+\tdm}}^2),
\eq
where $C=C(\| \eta_*\|_{H^{s+\tdm}}, d,b,c^0,s)$ since $\| f\|_{H^{s+\tdm}}\le \delta<1$.  Consequently, \cref{bound for Psi and Psi inverse} (with $\sigma=s$)  implies that $ \Psi_0[\eta_*]^{-1}h\in \rH^{s-\mez}(\T^d)$. Combining this with the bounds  \eqref{bound:Teta:inv} and \eqref{bound for R eta f}, we obtain 
    \bq\label{bound:Ggamma:S}
    \begin{aligned}
        \|G_\gamma(f)\|_{H^{s+\tdm}}& \leq \|T_{\eta_*}^{-1}\|_{\rH^{s-\mez}\to \rH^{s+\tdm}}\|\Psi_0[\eta_*]^{-1}\|_{H^{s+\mez}\to H^{s-\mez}} \| h\|_{H^{s+\mez}}  + \|T_{\eta_*}^{-1}\|_{H^{s-\mez}\to H^{s+\tdm}}\| R_{\eta_*}(f)\|_{H^{s-\mez}}      \\
       & \leq  C_1(|\gamma|+\| f\|_{H^{s+\tdm}}^2),
    \end{aligned}
    \eq
    where $C_1=C_1(\| \eta_*\|_{H^{s+\tdm}}, d,b,c^0,s, g, \sigma)$.
 
  To obtain the contraction of $G_\gamma$ we  write
    \begin{multline*}
    G_\gamma(f_1) - G_\gamma(f_2)  = -T_{\eta_*}^{-1} \Psi_0[\eta_*]^{-1}\Bigl( \gamma \partial_1(f_1-f_2) + (\Psi_0[\eta_* +f_1] - \Psi_0[\eta_*+f_2])\bigl(T_{\eta_*}(f_1)+R_{\eta_*}(f_1)\bigr) \\ +   (\Psi_0[\eta_* +f_2] - \Psi_0[\eta_*])\bigl(T_{\eta_*}(f_1-f_2)+R_{\eta_*}(f_1)-R_{\eta_*}(f_2)\bigr)  \Bigr)\\
    -T_{\eta_*}^{-1}\big(R_{\eta_*}(f_1)-R_{\eta_*}(f_2)\big).
    \end{multline*}
Using the contraction estimates for $\Psi_\gamma$, $T_{\eta_*}$ and $R_{\eta_*}$ obtained in \cref{linearization and contraction for Psi} and \cref{linearizing mean curvature at eta}, we deduce 
    \bq\label{contra:Ggamma:S}
        \|G_\gamma(f_1) - G_\gamma(f_2)\|_{H^{s+\tdm}} \leq C_2\Bigl(|\gamma|+\big(\|f_1\|_{H^{s+\tdm}} +\|f_2\|_{H^{s+\tdm}}\big)\Bigr)\|f_1-f_2\|_{H^{s+\tdm}} 
\eq
  for all $f_1$, $f_2$ in  $\overline{B_{H^{s+\tdm}}(0,\delta)}$, where  $C_2=C_2(\| \eta_*\|_{H^{s+\tdm}}, d,b,c^0,s, g, \sigma)$. Now we  set 
  \[
  \delta_S=\min\left\{r(d, c^0, s), (4C_1)^{-1}, (4C_2)^{-1}\right\},
  \]
and for any $\delta \in (0, \delta_S)$ we choose 
  \[
  |\gamma|<\gamma_\delta:= \delta \min\left\{ (4C_1)^{-1}, (4C_2)^{-1} \right\}.
  \]
In view of  the estimates \eqref{bound:Ggamma:S} and \eqref{contra:Ggamma:S}, we conclude that  $G_\gamma$ maps $\overline{B_{H^{s+\tdm}(\T^d)}(0,\delta)}$ to itself and is a contraction for all $\delta<\delta_S$ and $|\gamma|<\gamma_\delta$. 
\end{proof}
\subsection{Existence of large traveling waves for the Navier-Stokes problem}
We turn to construct large traveling wave solutions for the Navier-Stokes problem, i.e., the system \eqref{sys:main} with $\alpha=1$. We consider $\phi$ and $\eta_*$ as in Section \ref{sec:tw:Stokes} and seek a solution $\eta=\eta_*+f$, where $f$ is a small perturbation in $H^{s+\tdm}(\T^d)$. In particular, $f$ must satisfy \eqref{cd:f:c0} so that $\min_{\T^d}(\eta+b)\ge \frac{c^0}{2}>0$. In terms of the nonlinear normal-stress to normal-Dirichlet operator $\Phi_\gamma$, \eqref{sys:main} is equivalent to 
\bq\label{eq:NS:Phi}
-\gamma \partial_1 (f+ \eta_*)= \Phi_\gamma[f+\eta_*]\big((\sigma\cH+gI)(\eta_*+f) - (\sigma \cH+gI)\eta_* \big).
\eq
Similarly to \eqref{def:Ggamma}, \eqref{eq:NS:Phi} is equivalent to the fixed point equation
\bq
f=F_\gamma(f),
\eq
where 
\bq
F_\gamma(f):= T_{\eta_*}^{-1} \Phi_\gamma[\eta_*]^{-1}\Bigl(-\gamma \partial_1(\eta_*+f) - (\Phi_\gamma[\eta_* +f] - \Phi_\gamma[\eta_*])\bigl(T_{\eta_*}(f)+R_{\eta_*}(f)\bigr) \Bigr)-T_{\eta_*}^{-1}R_{\eta_*}(f).
\eq
We recall  that the inverse $\Psi_\gamma[\eta_*]^{-1}$ is given in terms of the operator $\Xi_\gamma[\eta_*]$ as in \eqref{form:Phiinverse} and is well-defined on $B_{\rH^{s+\mez}(\T^d)}(0, \delta^\dag)$, where $\delta^\dag$ depends only on $(\|\eta\|_{H^{s+\tdm}(\T^d)}, d,b,c^0, s)$. We establish the existence and uniqueness of fixed points of $F_\gamma$ in the next Proposition. 
\begin{theo}\label{large traveling wave for NS} 
     Let $(d+1)/2<s\in \Nn$. There exists $\delta_{NS}>0$ depending on $(d,b,c^0,s , g, \sigma,\|\eta_*\|_{H^{s+\tdm}(\T^d)})$ such that for any $0< \delta< \delta_{NS}$, there is some $\gamma_\delta>0$ such that for all $|\gamma|< \gamma_\delta$, $F_\gamma$ maps $\overline{B_{H^{s+\tdm}(\T^d)}(0,\delta)}$ to itself and  is a contraction.\end{theo}
\begin{proof}
We appeal to the Banach fixed point theorem in  $\overline{B_{H^{s+\tdm}(\T^d}(0, \delta)}$, where $\delta \le r(d, c^0, s)<1$ is a small number to be determined. We denote 
\[
h=-\gamma \partial_1(\eta_*+f) - (\Phi_\gamma[\eta_* +f] - \Phi_\gamma[\eta_*])\bigl(T_{\eta_*}(f)+R_{\eta_*}(f)\bigr).
\]
The bounds  \eqref{bound:Teta} and \eqref{bound for R eta f} imply 
\bq\label{est:TR:NS}
\| T_{\eta_*}(f)+R_{\eta_*}(f)\|_{H^{s-\mez}(\T^d)}\le C_1\| f\|_{H^{s+\tdm}(\T^d)},
\eq
where $C_1=C_1(\| \eta_*\|_{H^{s+\tdm}}, d, s , g, \sigma)$ since $\| f\|_{H^{s+\tdm}}\le \delta<1$. The difference  $\Phi_\gamma[\eta_* +f] - \Phi_\gamma[\eta_*]$ will be dealt with  using \cref{prop:contrPhi} (ii). To this end, we first restrict $|\gamma|<\gamma_\delta$ with 
\[
\gamma_\delta\le \inf_{\| f\|_{H^{s+\tdm}}<r(d, c^0, s)}\gamma^*(d, b, c_1^0, \|\eta_*+f\|_{W^{1, \infty}})
\]
Then we have the contraction estimate \eqref{contraction estimate for Phi in eta} provided $\| \chi\|_{H^{s+\tdm}}<\delta^1_\sharp= \delta^1_\sharp(\|\eta_*\|_{H^{s+\tdm}(\T^d)},  d,b,c^0,  s)$. In particular, by restricting $\delta<\delta_\sharp^1(C_1)^{-1}$ and recalling \eqref{est:TR:NS}, we have
\[
\|  (\Phi_\gamma[\eta_* +f] - \Phi_\gamma[\eta_*])\bigl(T_{\eta_*}(f)+R_{\eta_*}(f)\bigr) \|_{H^{s+\mez}(\T^d)}\le  C_2\| f\|^2_{H^{s+\tdm}(\T^d)},
\]
where $C_2=C_2(\| \eta_*\|_{H^{s+\tdm}}, d,b,c^0,s, g, \sigma)$. This implies $h\in \rH^{s+\mez}(\T^d)$ and
\[
\| h\|_{H^{s+\mez}(\T^d)}\le C_3(|\gamma|+\| f\|^2_{H^{s+\tdm}(\T^d)}),\quad C_3=C_3(\| \eta_*\|_{H^{s+\tdm}}, d,b,c^0,s, g, \sigma).
\]
 By virtue of \cref{prop of Phi and Xi} (ii), the application $\Phi_\gamma[\eta_*]^{-1}h$ is well-defined if we further restrict 
\[
\delta<\delta^\dag (4C_3)^{-1},\quad \gamma_\delta<\delta^\dag (4C_3)^{-1},
\]
where $\delta^\dag= \delta^1_\sharp(\|\eta_*\|_{H^{s+\tdm}(\T^d)},  d,b,c^0,  s)$. Moreover, by combing the formula \eqref{form:Phiinverse} for $\Phi_\gamma[\eta_*]^{-1}$ with the estimate \eqref{bound for Xi} for $\Xi_\gamma[\eta_*]$, we find 
\[
\|\Phi_\gamma[\eta_*]^{-1}h \|_{H^{s-\mez}(\T^d)}\le   C_4(|\gamma|+\| f\|^2_{H^{s+\tdm}(\T^d)}),\quad C_4=C_4(\| \eta_*\|_{H^{s+\tdm}}, d,b,c^0,s, g, \sigma).
\]
Hence, invoking  the bounds  \eqref{bound:Teta:inv} and \eqref{bound for R eta f} for $T_{\eta_*}^{-1}$ and $R_{\eta_*}$ yields
\bq
\| F_\gamma(f)\|_{H^{s+\tdm}(\T^d)}\le C_5(|\gamma|+\| f\|^2_{H^{s+\tdm}(\T^d)}),\quad C_5=C_5(\| \eta_*\|_{H^{s+\tdm}}, d,b,c^0,s , g, \sigma).
\eq
At this point, we choose 
\[
\delta_{NS}<\min\left\{r(d, c^0, s), \delta_\sharp^1C_1^{-1}, \delta^\dag (4C_3)^{-1}, (4C_5)^{-1}\right\},
\]
and for any $\delta\in (0, \delta_{NS})$ we choose 
\[
\gamma_\delta< \min\left\{\inf_{\| f\|_{H^{s+\tdm}}<r(d, c^0, s)}\gamma^*(d, b, c_1^0, \|\eta_*+f\|_{W^{1, \infty}}),  \delta^\dag (4C_3)^{-1},  \delta (4C_5)^{-1} \right\}.
\]
 Then, for any $\delta\in (0, \delta_{NS})$ and $|\gamma|<\gamma_\delta$, $F_\gamma$ maps  $\overline{B_{H^{s+\tdm}(\T^d}(0, \delta)}$ to itself.  The contraction estimate for $G_\gamma$ in $H^{s+\tdm}(\T^d)$ can be obtained analogously by using the contraction estimates for $\Phi_\gamma$, $\Xi_\gamma$,$T_{\eta_*}$, and $R_{\eta_*}$ in \cref{prop:contrPhi} and \cref{linearizing mean curvature at eta}. Then, by further decreasing $\delta_{NS}$ and $\gamma_\delta$, we obtain that $F_\gamma$ is a contraction on 
 $\overline{B_{H^{s+\tdm}(\T^d}(0, \delta)}$ for any $\delta\in (0, \delta_{NS})$ and $|\gamma|<\gamma_\delta$.
\end{proof}

%%%%%%%%%%%%%%%%%%%%%%%%%%%%%%%%%%%%%%%%%%%%%%%%%%%%%%%%%%%%%%%%%%%%%%%%%%%%%%%%%%%%%%%%%%%%%%%%%%%%%%%%%%

\section{Stability of large traveling waves for the Stokes problem}\label{sec:stab}
We consider the time-dependent free boundary Stokes equations written in the moving frame $(x, y,  t)\mapsto (x-\gamma e_1t, y,  t)$:
\bq\label{dynamic Stokes}
\begin{cases}
    - \Delta v + \nabla p = 0&\quad\text{in~} \Omega_{\eta(t,\cdot)}, \\
    \nabla\cdot v= 0&\quad\text{on~}  \Omega_{\eta(t,\cdot)}, \\
    (pI - \mathbb{D}v)(t,\cdot, \eta(t,\cdot))\mathcal{N}(t,\cdot) =(\sigma \cH(\eta(t,\cdot))+g\eta(t,\cdot)+\phi(t,\cdot)) \mathcal{N}(t,\cdot) &\quad\text{on~} \T^d,  \\
    \partial_t\eta(t,\cdot)-\gamma\p_1\eta(x, t)=v(t,x,\eta(\cdot))\cdot \cN(t,\cdot) &\quad\text{on~}  \T^d, \\
    v=0 &\quad\text{on~}  \Sigma_{-b}.
\end{cases}
\eq
Using linear normal-stress to normal-Dirichlet operator $\Psi_\gamma$, we can rewrite \eqref{dynamic Stokes} as a single equation for the free boundary $\eta$:
\bq \label{dynamic formulation for eta}
\partial_t\eta= \gamma \partial_1 \eta + \Psi_0[\eta](\sigma\cH(\eta)+g\eta +\phi).
\eq
We established in Section \ref{sec:tw:Stokes} the existence of time-independent solutions to \eqref{dynamic formulation for eta}. They are traveling wave solutions with small speed $\gamma$ and can have large amplitude. Our goal in this section is to prove that these traveling wave solutions are asymptotically stable for the dynamic problem \eqref{dynamic formulation for eta}.

%%%%%%%%%%%%%%%%%%%%%%%%%%%%%%%%%%%%%%%%%%%%%%%%%%%%%

%%%%%%%%%%%%%%%%%%%%%%%%%%%%%%%%%%%%%%%%%%%%%%%%%%%%%

\subsection{Linear dynamics}
Let $\eta_*: \T^d\to \Rr$ be the solution of \eqref{eq:eta*}. We consider  the following linear problem
\bq \label{linearized dynamic around eta*}
\partial_t f = \cL f := \gamma \partial_1 f + \Psi_0[\eta_*]T_{\eta_*}f,
\eq
where  $\cL$ is of order $1$ by virtue of \cref{bound for Psi and Psi inverse} and \cref{prop: T eta bijection}. The well-posedness of \eqref{linearized dynamic around eta*} will be established using the method of vanishing viscosity. The first step concerns the regularized problem in which $\cL$ is replaced with $\cL+\eps\Delta$.
\begin{lemm}\label{lem: regularized evolution}
    Let $(d+1)/2< s\in \Nn$, $ \sigma \in [\mez,  s+\mez]$, and $\eta_*\in H^{s+\tdm}(\T^d)$ satisfy \eqref{eta lower bound}. Let  $T>0$ and  $F\in L^2([0,T];\mathring{H}^\sigma(\T^d))$. For $\gamma \in \Rr$ and $\eps\in (0, 1)$, we consider the regularized operator
    \bq
    L_\eps:= \gamma \partial_1 + \eps \Delta. 
    \eq
For all $f_0\in \mathring{H}^\sigma(\T^d)$, there exists a unique solution $f^\eps\in C([0,T]; \mathring{H}^\sigma(\T^d))\cap L^2([0,T]; H^{\sigma  +1}(\T^d))$ to
    \bq \label{regularized evolution eq}
        \partial_t f^\eps= L_\eps f^\eps + \Psi_0[\eta_*](T_{\eta_*}f^\eps) + F, \quad f^\eps\vert_{t=0} = f_0.
    \eq
\end{lemm}
\begin{proof}
We define the Banach space $Y_T^\sigma =C([0,T]; \mathring{H}^\sigma(\T^d))\cap L^2([0,T]; H^{\sigma  +1}(\T^d))$
    with the norm
    \bq\label{norm Y_T}
    \|f\|_{Y^\sigma_T}= \|f\|_{C([0,T]; \mathring{H}^\sigma)} + \eps^{\mez}\|f\|_{ L^2([0,T]; H^{\sigma  +1})}.
    \eq
     Let $\bar{f}$ solve 
    \bq
        \partial_t \bar{f} = L_\eps \bar{f}, \quad \bar{f}\vert_{t=0} =f_0, 
    \eq
    and for any $g\in Y^\sigma_T$ and any $\tau\in[0,T]$  let $g^\tau$ solve
    \bq
    \partial_t g^\tau = L_\eps g^\tau, \quad g^\tau\vert_{t=\tau} =  \Psi_0[\eta_*]T_{\eta_*}g(\tau) + F(\tau).
    \eq
    Then, we define
    $$\cG(g)(t) = \bar{f}(t) + \int_0^t g^\tau(t)d\tau.$$
    Clearly,  fixed points of $\cG$ are solutions of \eqref{regularized evolution eq}. We shall prove that $\cG$ has a unique fixed point in $Y_{T_0}^\sigma$ for sufficiently small $T_0$.
    
   We have $g(\tau)\in \mathring{H}^{\sigma+1}$ and  $F(\tau) \in \mathring{H}^\sigma$  for a.e. $\tau\in[0,T]$. Since  $\Psi_0[\eta_*]\circ T_{\eta_*} : H^{\sigma+1}\to \mathring{H}^\sigma$ in view of the estimates \eqref{bound for Psi} and \eqref{bound:Teta}, we deduce that $g^\tau(\tau)\in \rH^\sigma$ for a.e. $\tau\in [0,T]$. By  standard Sobolev theory for the heat equation, we have the existence and uniqueness of $\bar{f}\in Y_T^\sigma$ and $g^\tau \in Y^\sigma([\tau,T]):= C([\tau,T]; \mathring{H}^\sigma)\cap L^2([0,T]; H^{\sigma  +1})$ for a.e. $\tau\in [0,T]$, with the bounds
    \bq\label{g tau and f bar bounds}
    \begin{split}
        \|\bar{f}\|_{Y_T^\sigma} &\leq M\|f_0\|_{H^\sigma} \\
        \|g^\tau\|_{Y^\sigma([\tau,T])} &\leq M \|\Psi_0[\eta_*]T_{\eta_*}g(\tau) + F(\tau)\|_{H^\sigma} \\
        &\leq C(\|\eta_*\|_{H^{s+\tdm}})\|g(\tau)\|_{H^{\sigma +1}} + M\|F(\tau)\|_{H^\sigma}\qquad \text{a.e. }\tau\in[0,T],
    \end{split}
    \eq
    where $M=M(\sigma,d)$; the independence of the constant $M$ from $\eps$ is due to our scaling choice for the norm of $Y_T^\sigma$ in (\ref{norm Y_T}). The preceding estimates for $\|g^\tau\|_{Y^\sigma([\tau,T])}$  requires $\sigma-\mez \in [0, s]$, i.e. $\sigma \in [\mez, s+\mez]$.  For any $t\in[0,T]$, using  (\ref{g tau and f bar bounds}) gives
    \begin{multline*}
        \|\int_0^t g^\tau(t)d\tau\|_{H^{\sigma}} \leq \int_0^t\|g^\tau(t)\|_{H^{\sigma}} d\tau 
        \leq \int_0^t\|g^\tau\|_{Y^\sigma([\tau,T])} d\tau \\
        \leq \int_0^t C(\|\eta_*\|_{H^{s+\tdm}})\|g(\tau)\|_{H^{\sigma +1}} + M\|F(\tau)\|_{H^\sigma} d\tau \\
        \leq  C(\|\eta_*\|_{H^{s+\tdm}})\sqrt{T} \|g\|_{L^2([0,T];H^{\sigma+1})} + M\sqrt{T}\|F\|_{L^2([0,T];H^{\sigma})}.
    \end{multline*}
    This implies  $\int_0^t g^\tau(t)d\tau \in C_t([0,T];\mathring{H}^\sigma)$. Also, appealing to Minkowski's inequality and (\ref{g tau and f bar bounds}) we have
    \begin{multline*}
            \|\int_0^t g^\tau(t)d\tau\|_{L^2_t([0,T];H^{\sigma+1})}  \leq \|\int_0^t \|g^\tau(t)\|_{H^{\sigma+1}}d\tau\|_{L^2_t([0,T])}  
            \leq \int_0^T \|g^\tau\|_{L^2_t([\tau,T];H^{\sigma+1})} d\tau \\
            \leq \int_0^T \|g^\tau\|_{Y^\sigma([\tau,T])} d\tau 
            \leq C(\|\eta_*\|_{H^{s+\tdm}})\sqrt{T} \|g\|_{L^2([0,T];H^{\sigma+1})} + M\sqrt{T}\|F\|_{L^2([0,T];H^{\sigma})}.
        \end{multline*}
    Thus   $\int_0^t g^\tau(t)d\tau \in L^2([0,T];H^{\sigma+1})$, and hence  $\cG$ maps $Y^\sigma_T$ to itself. More precisely, combining the preceding inequalities we have
    \bq \label{cG bound}
        \|\cG(g)\|_{Y^\sigma_T} \leq M\|f_0\|_{H^\sigma} + C(\|\eta_*\|_{H^{s+\tdm}})\sqrt{T}(\eps^{-\mez}+1)\|g\|_{Y^\sigma_T} + M\sqrt{T}(1+\eps^\mez)\|F\|_{L^2([0,T];H^{\sigma})},
    \eq
    where $M=M(\sigma,d)$ can be assumed to be greater than $1$. Setting $R= 2M\|f_0\|_{H^\sigma}$ and choosing 
    \bq \label{T_0 choice}
        T_0\leq \min \left\{\frac{1}{\bigl( 4C(\|\eta_*\|_{H^{s+\tdm}})(\eps^{-\mez}+1)\bigr)^2}, \frac{R^2}{\bigl( 4M(1+\eps^\mez)\|F\|_{L^2([0,T];H^{\sigma})}\bigr)^2}  \right\},
    \eq
    we deduce from \eqref{cG bound} that $\cG: \overline{B_{Y^\sigma_{T_0}}(0, R)} \to  \overline{B_{Y^\sigma_{T_0}}(0, R)} $.

    To prove $\cG$ is a contraction, let $g_1,g_2\in Y^\sigma_T$, then we have
    $$\cG(g_1)(t) - \cG(g_2)(t) = \int_0^t \bar{g}^\tau(t)d\tau,$$
    where $\bar{g}^\tau$ solves
    $$\partial_t \bar{g}^\tau = L_\eps \bar{g}^\tau,\quad \bar{g}^\tau\vert_{t=\tau} =  \Psi_0[\eta_*]T_{\eta_*}(g_1-g_2)(\tau).$$
    Arguing as before, we obtain 
    \[
    \|\bar{g}^\tau\|_{Y^\sigma([\tau,T])} \leq C(\|\eta_*\|_{H^{s+\tdm}})\|(g_1-g_2)(\tau)\|_{H^{\sigma +1}}
    \]
and 
    \[
    \begin{split}
        \|\int_0^t \bar{g}^\tau(t)d\tau\|_{C_t([0,T];\mathring{H}^\sigma)} +  \|\int_0^t \bar{g}^\tau(t)d\tau\|_{L^2_t([0,T];H^{\sigma+1})} &\leq C(\|\eta_*\|_{H^{s+\tdm}})\sqrt{T} \|g_1-g_2\|_{L^2([0,T];H^{\sigma+1})}.
      %&\leq C(\|\eta_*\|_{H^{s+\tdm}})\sqrt{T} \|g_1-g_2\|_{L^2([0,T];H^{\sigma+1})}.
    \end{split}
    \]
It follows that
    \bq\label{contra:cG} \|\cG(g_1) - \cG(g_2)\|_{Y^\sigma_T} \leq C(\|\eta_*\|_{H^{s+\tdm}})\sqrt{T}(1+\eps^{-\mez}) \|g_1-g_2\|_{Y^\sigma_T}. 
    \eq
    Thus, our choice for $T_0$ in (\ref{T_0 choice}) implies that  $\cG$ is a contraction on $ \overline{B_{Y^\sigma_{T_0}}(0, R)}$.  The unique fixed point  $f^\eps \in Y^\sigma_{T_0}$ of $\cG$ solves \eqref{regularized evolution eq} on $[0,T_0]$. Since  $T_0$ only depends on the given $\eps$ and $\eta_*$, we can extend $f^\eps$ to a  unique solution in $Y^\sigma_T$. 
    \end{proof}
Next, we take the limit of $f^\eps$ as $\eps$ goes to zero to establish the well-posedness of \eqref{linearized dynamic around eta*}.
\begin{prop}\label{prop: wellposedness of evolution for L}
    Let $(d+1)/2<s\in \Nn$, and $\sigma\in \Nn$ with 
    $1\leq \sigma\leq s+1$. Suppose $\eta_*\in H^{s+\fdm}(\T^d)$. There exists 
    \bq\label{def:gammadag}
    \gamma^\dag=\gamma^\dag(\|\eta_*\|_{H^{s+\tdm}}, c_{\Psi_0[\eta_*]}, d, s, g, \sigma)
    \eq such that the following holds. For any $|\gamma|\leq  \gamma^\dag$, $T>0$, and $F\in L^2([0,T]; \rH^{\sigma-\mez}(\T^d))$,  the problem 
    \bq \label{linear evolution equation for f}
    \partial_t f = \cL f + F, \qquad f\vert_{t=0} = f_0
    \eq
    has a unique solution 
    \bq\label{def:Xspace}
    f\in C([0,T];\mathring{H}^\sigma(\T^d))\cap  L^2([0,T]; H^{\sigma+\mez}(\T^d))=: X_T^\sigma.
    \eq
    Moreover, $f$ satisfies 
    \bq\label{bound for sol of evolution of L}
    \|f\|_{X^\sigma_T} \leq C(\|\eta_*\|_{H^{s+\fdm}}, c^{-1}_{\Psi_0[\eta_*]})\Bigl( \|f_0\|_{H^\sigma} + \|F\|_{L^2([0,T]; H^{\sigma-\mez})} \Bigr),
    \eq
    where $C: (\Rr_+)^3\to \Rr_+$ depends only on $(d,b,c^0,s,\sigma)$ and is increasing in each argument. 
\end{prop}
\begin{proof}
    We approximate $F$ with $F^\eps\in L^2([0,T]; \rH^{\sigma}(\T^d))$ converging to $F$ in $L^2([0,T]; \rH^{\sigma-\mez}(\T^d))$. We apply  Lemma \ref{lem: regularized evolution} to obtain a unique solution $f^\eps\in Y^\sigma_T$ solving the problem 
    \bq\label{f epsilon evolution}
        \partial_t f^\eps = \gamma \partial_1 f^\eps + \eps \Delta f^\eps + \Psi_0[\eta_*](T_{\eta_*}f^\eps)+ F^\eps, \qquad f^\eps\vert_{t=0} = f_0.
    \eq
    This is applicable as $\eta_*\in H^{s+\fdm}(\T^d) = H^{(s+1)+\tdm}(\T^d)$ and $\sigma \in [1,s+1] \subset [1,(s+1)+\mez]$.
    The remainder of the proof is divided into two steps.
    
    First, we establish a uniform in $\eps$ bound for $f^\eps$ in $X^\sigma_T$. In view of \eqref{def: a_ij},  $T_{\eta_*}$ is a  self-adjoint operator on $L^2$. It follows that
\begin{align*}
\mez \frac{d}{dt} \langle f^\eps, T_{\eta_*}f^\eps \rangle_{L^2,L^2} &= \langle  \partial_t f^\eps, T_{\eta_*}f^\eps\rangle_{L^2,L^2}=\langle \gamma \partial_1 f^\eps + \eps \Delta f^\eps + \Psi_0[\eta_*](T_{\eta_*} f^\eps) + F^\eps, T_{\eta_*}f^\eps \rangle_{L^2,L^2}.
\end{align*}
Using \eqref{def: a_ij} gives
$$
    \langle \Delta f^\eps , T_{\eta_*}f^\eps\rangle_{L^2,L^2} = \langle \Delta f^\eps , gf^\eps 
 - \sigma \dv(A\nabla f^\eps)\rangle_{L^2,L^2} = -g\|\nabla f^\eps\|_{L^2}^2 +\sigma \langle \Delta \nabla f^\eps , A\nabla f^\eps \rangle_{L^2,L^2},
$$
where $A=(a_{ij})_{i,j=1}^d$ is a  positive-definite symmetric matrix. Let $B$ be the unique symmetric positive-definite square root of $A$, i.e. $A=B^2$. The matrix $B$ can be obtained from $A$ via the Cholesky algorithm which involves finitely many elementary operations (products, inverses, square roots) involving coefficients of $A$. In particular, using \eqref{nonl:est} and \cref{product estimate on torus} successively we have   
\bq\label{est:matrixB}
\|B\|_{H^{s+\tdm}(\T^d)} \leq C(\|\eta_*\|_{H^{s+\fdm}(\T^d)}).
\eq
We observe that
\begin{multline*}
  \langle \Delta \nabla f^\eps , A\nabla f^\eps \rangle_{L^2,L^2} = \langle B\Delta \nabla f^\eps , B\nabla f^\eps \rangle_{L^2,L^2} \\
  = \langle \Delta (B \nabla f^\eps) , B\nabla f^\eps \rangle_{L^2,L^2} + \langle [B,\Delta] \nabla f^\eps , B\nabla f^\eps \rangle_{L^2,L^2} \\
  = - \| \na(B\nabla f^\eps)\|_{L^2}^2 + \langle [B,\Delta] \nabla f^\eps , B\nabla f^\eps \rangle_{L^2,L^2},
\end{multline*}
where $[B,\Delta]g=B\Delta g - \Delta (Bg)$ is a first-order operator. Combining \cref{product estimate on torus}  with \eqref{est:matrixB} yields 
 \bq\label{est: commutator for B}
    \|[B,\Delta]\|_{H^{r+1}\to H^r} \leq C(\|\eta_*\|_{H^{s+\tdm}}),\quad \mez -s<r\le  s-\mez.
\eq 
Putting the above conditions together leads to 
\begin{align*}
\mez \frac{d}{dt} \langle f^\eps, T_{\eta_*}f^\eps \rangle_{L^2,L^2} &= \langle \Psi_0[\eta_*](T_{\eta_*} f^\eps),  T_{\eta_*}f^\eps \rangle_{L^2,L^2}  -g\eps\|\nabla f^\eps\|_{L^2}^2 -\sigma \eps \|\na(B\nabla f^\eps)\|_{L^2}^2\\
&\qquad  +\sigma\eps \langle [B,\Delta] \nabla f^\eps , B\nabla f^\eps \rangle_{L^2,L^2}+\langle \gamma \partial_1 f^\eps, T_{\eta_*}f^\eps \rangle_{L^2,L^2}+  \langle F^\eps, T_{\eta_*}f^\eps \rangle_{L^2,L^2}.
\end{align*}
Invoking the coercive estimate \eqref{coercive:Psi}  for $-\Psi_0[\eta_*]$,  the estimates  \eqref{bound:Teta} and \eqref{bound:Teta:inv} for $T_{\eta_*}$ and $T_{\eta_*}^{-1}$, and  \eqref{est: commutator for B} with ($r=-\mez$), we obtain 
\begin{align*}
    \mez \frac{d}{dt} \langle f^\eps, T_{\eta_*}f^\eps \rangle_{L^2,L^2} &\le  -c_{\Psi_0[\eta_*]}\|T_{\eta_*}f\|_{H^{-\mez}}^2 +\sigma \eps\|  [B,\Delta] \nabla f^\eps\|_{H^{-\mez}}\| B\nabla f^\eps \|_{H^\mez}\\ 
    &\qquad +|\gamma| \| \partial_1 f^\eps\|_{H^\mez} \| T_{\eta_*}f^\eps \|_{H^{-\mez}}+  \| F^\eps\|_{H^\mez}\| T_{\eta_*}f^\eps \|_{H^{-\mez}}\\
&\le  -\frac{c_{\Psi_0[\eta_*]}}{C_1}\| f\|_{H^\tdm}^2 +\sigma \eps C_2\| f^\eps \|_{H^\tdm}^2 +|\gamma| C_1\| f^\eps \|_{H^\tdm}^2+  C_1\| F^\eps\|_{H^\mez}\| f\|_{H^\tdm},
\end{align*}
where $c_{\Psi_0[\eta_*]}=c_{\Psi_0[\eta_*]}(d, c^0, \eta_*)>0$ and  $C_j=C_j(\|\eta_*\|_{H^{s+\tdm}})$, $j=1, 2$. We choose $\eps$ and $\gamma$ such that 
\bq\label{cd:epsgamma:1}
\eps \sigma C_2\le \frac{c_{\Psi_0[\eta_*]}}{4C_1},\quad |\gamma| \le   \frac{c_{\Psi_0[\eta_*]}}{4C_1^2}. 
\eq
Then we obtain 
\bq\label{energy:fesp:low}
 \mez \frac{d}{dt} \langle f^\eps, T_{\eta_*}f^\eps \rangle_{L^2,L^2} \le -\frac{c_{\Psi_0[\eta_*]}}{2C_1}\| f\|_{H^\tdm}^2+C_3\| F^\eps\|_{H^\mez}^2,\quad C_3=C_3(\|\eta_*\|_{H^{s+\tdm}}).
 \eq
Let $\alpha\in \Nn^d$ be a multiindex of order $\sigma-1$. Commuting \eqref{f epsilon evolution} with $\p^\alpha$ yields
\begin{align*}
    \mez \frac{d}{dt}\langle \p^\alpha f^\eps, T_{\eta_*}\p^\alpha f^\eps \rangle_{L^2,L^2} &= \langle \gamma \partial_1\p^\alpha  f^\eps + \eps \Delta \p^\alpha f^\eps + \p^\alpha \Psi_0[\eta_*](T_{\eta_*} f^\eps) + \p^\alpha F^\eps, T_{\eta_*}\p^\alpha f^\eps \rangle_{L^2,L^2} \\
    &= \underbrace{\langle \gamma \partial_1\p^\alpha  f^\eps - \eps \Delta \p^\alpha f^\eps +  \Psi_0[\eta_*](T_{\eta_*} \p^\alpha f^\eps) + \p^\alpha F^\eps, T_{\eta_*}\p^\alpha f^\eps \rangle_{L^2,L^2}}_{I}\\
    &\qquad +\underbrace{\langle \Psi_0[\eta_*]([\p^\alpha, T_{\eta_*} ]f^\eps) + \bigl[\p^\alpha ,\Psi_0[\eta_*]\bigr] (T_{\eta_*} f^\eps), T_{\eta_*}\p^\alpha f^\eps \rangle_{L^2,L^2}}_{II}.
\end{align*}
Under the conditions in \eqref{cd:epsgamma:1}, we have in view of \eqref{energy:fesp:low} that
\bq\label{est:I:flin}
|I|\le -\frac{c_{\Psi_0[\eta_*]}}{2C_1}\| \p^\alpha f\|_{H^\tdm}^2+C_3\| \p^\alpha F^\eps\|_{H^\mez}^2.
\eq
Next, we estimate 
\begin{align*}
|II| %&\le \left(\|\Psi_0[\eta_*]([\p^\alpha, T_{\eta_*} ]f^\eps)\|_{H^{\tdm}}+\| \bigl[\p^\alpha ,\Psi_0[\eta_*]\bigr] (T_{\eta_*} f^\eps)\|_{H^{\tdm}}\right)\| T_{\eta_*}\p^\alpha f^\eps \|_{H^{-\tdm}}\\
&\le \left(\|\Psi_0[\eta_*]\|_{H^\mez\to H^\tdm}\| [\p^\alpha, T_{\eta_*} ]f^\eps\|_{H^\mez}+\| \bigl[\p^\alpha ,\Psi_0[\eta_*]\bigr] (T_{\eta_*} f^\eps)\|_{H^{\tdm}}\right)\| T_{\eta_*}\p^\alpha f^\eps \|_{H^{-\tdm}}.
\end{align*}
Combining \eqref{bound for Psi} (with $\sigma=1\le s$) with the commutator estimate \eqref{commutator bound for T eta} (with $\sigma=\mez$ and $s$ replaced with $s+1$), we obtain 
\[
\|\Psi_0[\eta_*]\|_{H^\mez\to H^\tdm}\| [\p^\alpha, T_{\eta_*} ]f^\eps\|_{H^\mez}\le C_4^{(1)}(\| \eta_*\|_{H^{s+\fdm}}) \| f^\eps\|_{H^{|\alpha|+\tdm}}.
\]
On the other hand, applying the commutator estimate \eqref{commutator bound} (with $\sigma=\tdm$) yields 
\[
\| \bigl[\p^\alpha ,\Psi_0[\eta_*]\bigr] (T_{\eta_*} f^\eps)\|_{H^{\tdm}}\le C_4^{(2)}(\| \eta_*\|_{H^{s+\tdm}})\| T_{\eta_*} f^\eps\|_{H^{|\alpha|-\mez}}.
\]
Then, invoking the estimate \eqref{bound:Teta} for $T_{\eta_*}$ then applying Young's inequality, we  deduce 
\bq\label{est:II:flin}
\begin{aligned}
|II|&\le  C^{(3)}_4(\| \eta_*\|_{H^{s+\fdm}})  \| f^\eps\|_{H^{|\alpha|+\tdm}}\| f^\eps\|_{H^{|\alpha|+\mez}}\\
&\le\frac{c_{\Psi_0[\eta_*]}}{4C_1} \| f^\eps\|_{H^{\sigma+\mez}}^2+C_4(\| \eta_*\|_{H^{s+\fdm}})  \| f^\eps\|_{H^\tdm}^2.
\end{aligned}
\eq
where we recall that $|\alpha|=\sigma-1$.  A combination of  \eqref{est:I:flin} and \eqref{est:II:flin} leads to 
\bq\label{energy:feps:high}
    \mez \frac{d}{dt}\sum_{|\alpha|=\sigma-1}\langle \p^\alpha f^\eps, T_{\eta_*}\p^\alpha f^\eps \rangle_{L^2,L^2} \le -\frac{c_{\Psi_0[\eta_*]}}{4C_1}\sum_{|\alpha|=\sigma-1}\| \p^\alpha f\|_{H^\tdm}^2+C_3\|  F^\eps\|_{H^{\sigma-\mez}}^2+ C_4 \| f^\eps\|_{H^\tdm}^2
\eq
If we choose $A=A(\| \eta_*\|_{H^{s+\fdm}}, c_{\Psi_0[\eta_*]})>0$ satisfying 
\[
(2A-1)\frac{c_{\Psi_0[\eta_*]}}{4C_1}>C_4,
\]
 then it follows from \eqref{energy:fesp:low} and \eqref{energy:feps:high} that the energy 
\bq\label{def:energy}
E(t):= \mez A  \langle f^\eps, T_{\eta_*}f^\eps \rangle_{L^2,L^2}+\mez \sum_{|\alpha|=\sigma-1} \langle \p^\alpha f^\eps, T_{\eta_*}\p^\alpha f^\eps \rangle_{L^2,L^2} 
\eq
satisfies 
\bq\label{energyineq:feps}
\begin{aligned}
E'(t)&\le -\frac{c_{\Psi_0[\eta_*]}}{4C_1}\Bigl(\|  f\|_{H^\tdm}^2+\sum_{|\alpha|=\sigma-1}\| \p^\alpha f\|_{H^\tdm}^2\Bigl)+2C_3\|  F^\eps\|_{H^{\sigma-\mez}}^2+\Big(C_4-(2A-1)\frac{c_\Psi C_1}{4}\Big) \| f^\eps\|_{H^\tdm}^2\\
&\le -\frac{c_{\Psi_0[\eta_*]}}{4C_1}\Bigl(\|  f\|_{H^\tdm}^2+\sum_{|\alpha|=\sigma-1}\| \p^\alpha f\|_{H^\tdm}^2\Bigl)+2C_3\|  F^\eps\|_{H^{\sigma-\mez}}^2.
\end{aligned}
\eq
Since 
\[
    \langle g, T_{\eta_*}g \rangle_{L^2,L^2} = \int_{\T^d} g^2 + \int_{\T^d} a_{ij}\partial_i g \partial_j g,
\]
 the ellipticity bounds (\ref{ellipticity of aij}) imply 
\bq\label{E and Sobolev energy comparable}
\frac{1}{(1+\| \na \eta_*\|^2_{L^\infty})^\tdm} \|g\|_{H^1}^2 \leq \langle g, T_{\eta_*}g \rangle_{L^2,L^2} \leq   \|g\|_{H^1}^2 \quad\forall g\in H^1(\T^d).
\eq
Consequently, $E(t)$ is equivalent to $\| f(t)\|_{H^\sigma}^2$ up to multiplicative constants depending only on $\| \eta_*\|_{H^{s+\fdm}}$ and $c_{\Psi_0[\eta_*]}$. Therefore, integrating \eqref{energyineq:feps} in $t\in [0, T]$ yields 
\bq\label{f epsilon bound}
\|f^\eps\|_{X_T^\sigma}\leq C(\|\eta_*\|_{H^{s+\fdm}}, c^{-1}_{\Psi_0[\eta_*]}) \Bigl( \|f_0\|_{H^\sigma} + \|F^\eps\|_{L^2([0,T];H^{\sigma-\mez})} \Bigr),
\eq
where $C: (\Rr_+)^3\to \Rr_+$ is increasing in each argument.  Note that $\|F^\eps\|_{L^2([0,T];H^{\sigma-\mez})}$ is uniformly bounded in $\eps$, hence so is $\|f^\eps\|_{X_T^\sigma}$.

The uniform bound (\ref{f epsilon bound}) implies that up to taking a subsequence, there exists $f: \T^d\times [0, T]\to \Rr$ such that 
\[
f^\eps \overset{\ast}{\rightharpoonup} f~ \text{in~} L^\infty([0, T]; \rH^\sigma)\quad\text{and}\quad f^\eps \rightharpoonup f~ \text{in~} L^2([0, T]; H^{\sigma+\mez})
\]
as $\eps \to 0$.  Then $f$ satisfies the bound \eqref{bound for sol of evolution of L} in view of \eqref{f epsilon bound}. Moreover, we have 
\[
\gamma \partial_1 f^\eps \wc  \gamma \partial_1 f \quad \text{in  } L^2([0,T];H^{\sigma-\mez}), \quad \qquad   \eps \Delta f^\eps \rightharpoonup 0\quad \text{in  } L^2([0,T];H^{\sigma-\tdm}).
\]
Using the fact that  $\Psi_0[\eta_*]$ and $T_{\eta_*}$ are self-adjoint operators (see \cref{lemm: Psi_0 is self-adjoint} and \cref{prop: T eta bijection}) and $\Psi_0[\eta_*]T_{\eta_*}: H^{\sigma+\mez}(\T^d)\to H^{\sigma-\mez}(\T^d)$ is bounded for $1\le \sigma\le s+1$ (see \cref{bound for Psi and Psi inverse}), we deduce 

$$\Psi_0[\eta_*](T_{\eta_*} f^\eps) \rightharpoonup \Psi_0[\eta_*](T_{\eta_*} f) \quad \text{in  } L^2([0,T];H^{\sigma-\mez}). $$
In combination with the strong convergence $F^\eps \to F$ is $L^2([0,T];H^{\sigma -\mez})$, we can conclude that 
\[
\p_tf=\gamma \partial_1 f +  \Psi_0[\eta_*](T_{\eta_*} f) + F \quad \text{in  } L^2([0,T];H^{\sigma-\tdm}).
\]
It remains to show that $f$ achieves the initial data $f_0$. To this end,  we note that $\sigma -\mez\ge \mez$ and $\sigma+\mez\le s+\tdm$, so that \cref{bound for Psi and Psi inverse} and \cref{prop: T eta bijection} imply
\[
\| \Psi_0[\eta_*]T_{\eta_*} f^\eps\|_{H^{\sigma-\mez}}\le C_5(\| \eta_*\|_{H^{s+\tdm}})\| T_{\eta_*} f^\eps\|_{H^{\sigma-\tdm}} \le C_6(\| \eta_*\|_{H^{s+\tdm}})\|  f^\eps\|_{H^{\sigma+\mez}}.
\]
Since $f^\eps$ is uniformly bounded in $L^2([0, T]; H^{\sigma+\mez})$, it follows that $\Psi_0[\eta_*]T_{\eta_*} f^\eps$  and $\gamma \partial_1 f^\eps + \eps \Delta f^\eps$  are uniformly bounded in $L^2([0, T]; H^{\sigma-\mez})$ and  $L^2([0, T]; H^{\sigma-\tdm})$, respectively.  Hence, $\|\partial_t f^\eps\|_{X^{\sigma-2}_T}$ is uniformly bounded in  $L^2([0, T]; H^{\sigma-\tdm})$.  Then we can apply the  Aubin-Lions lemma for $H^\sigma \subset H^{\sigma -1} \subset H^{\sigma-\tdm}$ to deduce that  $\{f^\eps\}_{0<\eps\ll 1}$ is compactly embedded into $ C([0,T];H^{\sigma-1})$.  But $f^\eps\vert_{t=0}=f_0$ for all $\eps$, so $f\vert_{t=0} = f_0$. Finally, since  $f\in L^2([0, T]; H^{\sigma+\mez})$ and $\p_t f\in  L^2([0, T]; H^{\sigma-\mez})$, we deduce using interpolation that $ f\in C([0, T]; H^\sigma)$, whence $f\in X^\sigma_T$. 
\end{proof}

%%%%%%%%%%%%%%%%%%%%%%%%%%%%%%%%%%%%%%%%%%%%%%%%%%%%%
\subsection{Asymptotic stability}

Suppose that $\eta_w(x)$ is a traveling wave with speed $\gamma$ to the free boundary Stokes  problem, i.e., 
\bq
\gamma \partial_1 \eta_w + \Psi_0[\eta_w]\big(\sigma\cH(\eta_w)+g\eta_w +\phi\big)=0.
\eq
To investigate the stability of $\eta_w$ with respect to the  dynamic free boundary Stokes problem \eqref{dynamic formulation for eta} , we denote the perturbation by $f(x, t)=\eta(x, t)-\eta_w(x)$. Using \eqref{expand:cH}, we write  
\begin{align*}
\partial_t f - \gamma \partial_1 f &= \Psi_0[\eta_w+f]\Big((\sigma\cH+g\Id)(\eta_w+f)+\phi \Big)- \Psi_0[\eta_w]\Big((\sigma\cH+g\Id)(\eta_w)+\phi \Big)\\
&=  \Psi_0[\eta_w]T_{\eta_w}f+ \Psi_0[\eta_w+f]\Big((\sigma\cH+g\Id)(\eta_w)+T_{\eta_w}f+R_{\eta_w}(f)+\phi \Big)\\
&\qquad-  \Psi_0[\eta_w]\Big((\sigma\cH+g\Id)(\eta_w)+\phi +T_{\eta_w}f \Big)\\
&= \Psi_0[\eta_w]T_{\eta_w}f+ \Psi_0[\eta_w+f]R_{\eta_w}(f)\\
&\qquad+ \big(\Psi_0[\eta_*+f]-\Psi_0[\eta_w]\big)\Big((\sigma\cH+g\Id)(\eta_w)+T_{\eta_w}f+\phi \Big).
\end{align*}
Then replacing $\eta_w$ with $\eta_*$ in the term  $\Psi_0[\eta_w]T_{\eta_w}f$, and replacing $\phi$ with $-(\sigma\cH+g\Id)(\eta_*)$, we obtain 
\bq\label{eq: perturbation evolution}
\partial_t f = \gamma \partial_1 f + \Psi_0[\eta_*](T_{\eta_*} f) +  N(f),
\eq
where 
\begin{multline}\label{def: N(f)}
    N(f) = (\Psi_0[\eta_w]-\Psi_0[\eta_*])(T_{\eta_w} f) + \Psi_0[\eta_*](T_{\eta_w}f - T_{\eta_*}f)+ \Psi_0[\eta_w+f]R_{\eta_w}(f) \\ + \big(\Psi_0[\eta_w+f]-\Psi_0[\eta_w]\big)\Big((\sigma\cH+g\Id)(\eta_w)-(\sigma\cH+g\Id)(\eta_*)+T_{\eta_w}f \Big).
    \end{multline}
Viewing  $\eta_*$, and $\eta_w$ in \eqref{def: N(f)} as  independent given functions, we prove the following asymptotic stability of the zero solution of \eqref{eq: perturbation evolution}.

\begin{theo}[\textbf{Aymptotic Stability}]\label{theo:stab}
    Let $(d+1)/2+1<s\in \Nn$, $\eta_*\in  H^{s+\fdm}(\T^d)$ and $\eta_w\in H^{s+\tdm}(\T^d)$ such that 
    \[
    \inf_{\T^d}(\eta_*+b)\ge \frac{c^0}{4},\quad  \inf_{\T^d}(\eta_w+b)\ge \frac{c^0}{4},\quad c^0>0.
    \]
Consider $|\gamma|\le \gamma^\dag=\gamma^\dag(\|\eta_*\|_{H^{s+\tdm}}, c_{\Psi_0[\eta_*]}, d, s, g, \sigma)$ as in \eqref{def:gammadag}. Then there exists a  function $\omega: (\Rr_+)^2\to \Rr_+$ depending only on $(d, b, c^0, s, \sigma, g)$ and increasing in each argument such that if 
    \bq\label{stab:cd:thm}
    \| \eta_w-\eta_*\|_{H^{s+\tdm}}<\omega\big(\|\eta_*\|_{H^{s+\fdm}}, \|\eta_w\|_{H^{s+\tdm}}, c^{-1}_{\Psi_0[\eta_*]}\big)=:\overline{\om}
    \eq
   then the following holds. For each number $0<\delta< \overline{\om}$,  if  $f_0\in \mathring{H}^{s+1}(\T^d)$ satisfies $\|f_0\|_{H^{s+1}}\leq \overline{\om}\delta$, then 
    the problem
    \bq\label{evolution for f}
    \begin{cases}
        \partial_t f = \gamma \partial_1 f + \Psi_0[\eta_*](T_{\eta_*} f) +  N(f), \\
        f\vert_{t=0} = f_0,
    \end{cases}
    \eq
    has a unique solution in $X^{s+1}_T$ for all $T>0$, where  $X^{s+1}_T$ and $N(f)$ are  defined by \eqref{def:Xspace} and \eqref{def: N(f)}, respectively. Moreover, $f$ obeys the global bounds 
    \bq\label{global in time estimate for f}
        \|f\|_{X^{s+1}_T} \leq M_1\|f_0\|_{H^{s+1}}, \qquad T>0,
    \eq
    \bq\label{decayest}
        \|f(t)\|_{H^{s+1}} \leq M_2\|f_0\|_{H^{s+1}}e^{-c_0t}, \qquad t>0.
    \eq
   Here the constants  $M_1$, $M_2$, and $c_0$ depend only on $(\|\eta_*\|_{H^{s+\fdm}}, \|\eta_w\|_{H^{s+\tdm}}, c_{\Psi_0[\eta_*]}, d, b, c^0, s, \sigma, g)$,  where $M_2$ and $c_0$ are independent of $\|\eta_w\|_{H^{s+\tdm}}$. 
\end{theo}
\begin{proof}
The proof is divided into two steps.

1.  Global well-posedness.  Our goal in this step is to prove that \eqref{evolution for f} has a unique solution $f$ in $X^{s+1}_T$ for all $T>0$ if $\| f_0\|_{\rH^{s+1}}$ is sufficiently small in  $f$ via a fixed-point argument.  We fix an arbitrary number $T>0$.  For any  $f\in X^{s+1}_T$, we let $\bar{f}$ and $\tilde{f}$ be solutions of the problems
    \bq
        \partial_t \bar{f} = \cL \bar{f}, \quad \bar{f}\vert_{t=0} = f_0
    \eq
    and
    \bq\label{eq:tildef}
        \partial_t \tilde{f} = \cL \tilde{f} + N(f), \quad \tilde{f}\vert_{t=0}=0,
    \eq
    where 
    \[
    \cL=\gamma \partial_1  + \Psi_0[\eta_*]T_{\eta_*}.
    \]
    Then $f$ is a  solution of \eqref{evolution for f}  if $f$ is   the fixed point of the map 
    $$f\mapsto \cF(f):= \bar{f}+ \tilde{f}.$$ 
We start by  showing that $\cF: X^{s+1}_T \to X^{s+1}_T$ is well-defined.  Applying \cref{prop: wellposedness of evolution for L} with $\sigma=s+1$ and $F=0$, we deduce that  $\bar{f}$ exists and belongs to $X^{s+1}_T$. As for the existence of $\tilde{f}$ in $X_T^{s+1}$,  by appealing to \cref{prop: wellposedness of evolution for L} again,  it suffices to show that $N(f)\in L^2([0, T]; H^{s+\mez})$. To this end, we will use the assumption that $(d+1)/2<s-1\in \Nn$ to apply the tame estimates \eqref{tame:contPsi} and \eqref{tame contraction estimate for R eta} with $\sigma=s$. We will also invoke the embedding $H^{s-\mez}(\T^d)\subset L^\infty(\T^d)$. First, a combination of \eqref{tame:contPsi} and \eqref{contraction for T in eta} yields
\bq\label{estN:1}
\begin{aligned}
\|  \Psi_0[\eta_*](T_{\eta_w}f-T_{\eta_*}f)\|_{H^{s+\mez}}& \le C(\| \eta_*\|_{H^{s+\mez}})\| T_{\eta_w}f-T_{\eta_*}f\|_{H^{s-\mez}}\\
&\le  C(\|\eta_*\|_{H^{s+\mez}},\|\eta_w\|_{H^{s+\mez}}) \|\eta_w-\eta_*\|_{H^{s+\tdm}(\T^d)}\|f\|_{H^{s+\tdm}}.
\end{aligned}
\eq
Next, it follows from \eqref{tame:contPsi} and \eqref{bound for R eta f} that
\bq\label{estN:2}
\begin{aligned}
\| \Psi_0[\eta_w+f]R_{\eta_w}(f)\|_{H^{s+\mez}}&\le C(\| \eta_w+f\|_{H^{s+\mez}}) \| R_{\eta_w}(f)\|_{H^{s-\mez}}\\
&\le C(\| \eta_w\|_{H^{s+\tdm}}, \| f\|_{H^{s+\mez}}) \| f\|_{H^{s+\mez}}\| f\|_{H^{s+\tdm}}.
\end{aligned}
\eq
Using \eqref{tame contraction estimate for R eta} and \eqref{bound:Teta}, we find 
\bq\label{estN:3}
\begin{aligned}
\| (\Psi_0[\eta_w]-\Psi_0[\eta_*])(T_{\eta_w} f)\|_{H^{s+\mez}} &  \leq C(\|\eta_*\|_{H^{s+\tdm}},\|\eta_w\|_{H^{s+\tdm}})\|\eta_w-\eta_*\|_{H^{s+\tdm}}\|T_{\eta_w} f\|_{H^{s-\mez}} \\
&  \leq C(\|\eta_*\|_{H^{s+\tdm}},\|\eta_w\|_{H^{s+\tdm}})\|\eta_w-\eta_*\|_{H^{s+\tdm}}\| f\|_{H^{s+\tdm}}.
     \end{aligned}
\eq
Set $\chi =(\sigma\cH+g\Id)(\eta_w)-(\sigma\cH+g\Id)(\eta_*)+T_{\eta_w}f$. The estimate \eqref{tame contraction estimate for R eta}  implies
\begin{multline}\label{estN:10}
\left\| \big(\Psi_0[\eta_w+f]-\Psi_0[\eta_w]\big)\chi \right\|_{H^{s+\mez}} \leq C(\|\eta_w\|_{H^{s+\mez}}, \| f\|_{H^{s+\mez}})\Bigl\{\|f\|_{H^{s+\mez}}\|\chi\|_{H^{s-\frac12}}  
 + \|f\|_{H^{s+\tdm}}\|\chi\|_{H^{s-\frac52}} \\
       + \|f\|_{H^{s+\mez}}\|\chi\|_{H^{s-\frac52}}\big(\|\eta_w\|_{H^{s+\tdm}} + \|f\|_{H^{s+\tdm}}\big) 
        \Bigr\}. 
\end{multline}
We now estimate $\chi$ in $H^{r-\mez}$, $r\in [s-2, s]$. An application of \cref{prop: composition regularity} yields 
\bq\label{estN:11}
\| \cH(\eta_w)-\cH(\eta_*)\|_{H^{r-\mez}}\le C(\| \eta_*\|_{H^{s-\mez}}, \| \eta_w\|_{H^{s-\mez}})\| \eta_w-\eta_*\|_{H^{r+\mez}}. 
\eq
As for $T_{\eta_w}f$, we apply \eqref{bound:Teta} with $s$ replaced with $r>(d-1)/2$ for $r\in [s-2, s]$: 
\bq\label{estN:12}
\| T_{\eta_w} f\|_{H^{r-\mez}}\le C(\| \eta_w\|_{H^{s-\mez}})\big(\| f\|_{H^{r+\tdm}}+ \| \eta\|_{H^{r+\tdm}}\| f\|_{H^{s-\mez}}\big)
\eq
It follows from \eqref{estN:10}, \eqref{estN:11}, and \eqref{estN:12} that 
\bq\label{estN:13}
\left\| \big(\Psi_0[\eta_w+f]-\Psi_0[\eta_w]\big)\chi \right\|_{H^{s+\mez}} \le C(\|\eta_w\|_{H^{s+\tdm}}, \| f\|_{H^{s+\mez}})\Big(\| \eta_w-\eta_*\|_{H^{s+\tdm}}+\| f\|_{H^{s+\mez}}\Big)\| f\|_{H^{s+\tdm}}
\eq
Combining \eqref{estN:1}, \eqref{estN:2}, \eqref{estN:3}, and \eqref{estN:13}, we deduce 
\bq\label{est:N:final}
\| N(f)\|_{H^{s+\mez}} \le C(\|\eta_*\|_{H^{s+\tdm}}, \|\eta_w\|_{H^{s+\tdm}}, \| f\|_{H^{s+\mez}})\Big(\| \eta_w-\eta_*\|_{H^{s+\tdm}}+\| f\|_{H^{s+\mez}}\Big)\| f\|_{H^{s+\tdm}}.
\eq
Taking the $L^2$ norm in $t$ leads to 
\begin{multline}\label{estN:L2}
\| N(f)\|_{L^2([0, T]; H^{s+\mez})}\le C\Big(\|\eta_*\|_{H^{s+\tdm}}, \|\eta_w\|_{H^{s+\tdm}}, \| f\|_{L^\infty([0, T]; H^{s+\mez})}\Big)\\
\cdot\Big(\| \eta_w-\eta_*\|_{H^{s+\tdm}}+ \| f\|_{L^\infty([0, T]; H^{s+\mez})}\Big)\| f\|_{L^2([0, T]; H^{s+\tdm})}.
\end{multline}
Using \eqref{estN:L2}, we deduce from the estimate \eqref{bound for sol of evolution of L} with $\sigma=s+1$ that
\bq\label{est:cF(f)}
\begin{aligned}
\| \cF(f)\|_{X^{s+1}_T}&\le \| \overline{f}\|_{X^{s+1}_T}+\| \tilde{f}\|_{X^{s+1}_T} \\
&\le  C(\|\eta_*\|_{H^{s+\fdm}})\Bigl( \|f_0\|_{H^{s+1}} + \|F\|_{L^2([0,T]; H^{s+\mez})} \Bigr)\\
&\le   \tilde{C_1}\Big(\|\eta_*\|_{H^{s+\fdm}}, \|\eta_w\|_{H^{s+\tdm}}, c^{-1}_{\Psi_0[\eta_*]}, \| f\|_{X^{s+1}_T}\Big)\Bigl\{\|f_0\|_{H^{s+1}}+\big(\| \eta_w-\eta_*\|_{H^{s+\tdm}}+\| f\|_{X^{s+1}_T}\big)\| f\|_{X^{s+1}_T} \Bigr\}
\end{aligned}
\eq
for some $\tilde{C_1}: (\Rr_+)^4\to \Rr_+$ depending only on $(d, b, c^0, s, \sigma, g)$ and increasing in each argument.

Now we consider $f\in \overline{B(0, \delta)}\subset X_T^{s+1}$ for some $\delta \in (0, 1)$. Set 
\[
C_1=\tilde{C_1}\Big(\|\eta_*\|_{H^{s+\fdm}}, \|\eta_w\|_{H^{s+\tdm}}, c^{-1}_{\Psi_0[\eta_*]}, 1\Big).
\]
 Then \eqref{est:cF(f)} implies that $\cF: \overline{B(0, \delta)}\to \overline{B(0, \delta)}$ upon choosing 
\bq\label{stab:small1}
\| \eta_w-\eta_*\|_{H^{s+\tdm}}<\frac{1}{3C_1},\quad \delta<\frac{1}{3C_1},\quad \|f_0\|_{H^{s+1}}<\frac{\delta}{3C_1}. 
\eq

Next, we prove the contraction of $\cF$. Consider $f_j\in  X_T^{s+1}$, $j=1, 2$. Then $\cF(f_1)-\cF(f_2)=h_1-h_2$, where $h_j\in X^{s+1}_T$ is the solution of \eqref{eq:tildef} with $f=f_j$. Since $h:=h_1-h_2\in X^{s+1}_T$ is satisfies 
\[
\p_t h=\cL f+N(f_1)-N(f_2),\quad h\vert_{t=0}=0,
\]  
appealing to the estimate  \eqref{bound for sol of evolution of L} again gives 
\bq\label{contrah:1}
\| h\|_{X^{s+1}_T}\le C(\|\eta_*\|_{H^{s+\fdm}}, c_{\Psi_0[\eta_*]})\| N(f_1)-N(f_2)\|_{L^2([0, T]; H^{s+\mez})}.
\eq
By appealing to  \cref{bound for Psi and Psi inverse}, \cref{linearization and contraction for Psi}, and \cref{linearizing mean curvature at eta} as in the proof of \eqref{estN:L2}, one can show 
     \begin{multline}
    \|N(f_1)-N(f_2)\|_{L^2([0,T];H^{s+\mez})} \leq C(\|\eta_*\|_{H^{s+\tdm}},\|\eta_w\|_{H^{s+\tdm}},\|f_1\|_{L^\infty([0,T];H^{s+\mez})},\|f_2\|_{L^\infty([0,T];H^{s+\mez})}) \\ 
    \cdot \Bigl[ \bigl(\|f_1\|_{L^2([0,T];H^{s+\tdm})}+\|f_2\|_{L^2([0,T];H^{s+\tdm})} \bigr)\|f_1-f_2\|_{L^\infty([0,T];H^{s+\mez})} + \|\eta_w-\eta_*\|_{H^{s+\tdm}} \|f_1-f_2\|_{L^2([0,T];H^{s+\tdm})} \Bigr].
\end{multline}
Then \eqref{contrah:1} implies 
\begin{multline}
  \| h\|_{X^{s+1}_T}\le  \tilde{C}_2(\|\eta_*\|_{H^{s+\fdm}},\|\eta_w\|_{H^{s+\tdm}}, c^{-1}_{\Psi_0[\eta_*]}, \|f_1\|_{X^{s+1}_T},\|f_2\|_{X^{s+1}_T}) \\ 
    \cdot \bigl(\|f_1\|_{X^{s+1}_T}+\|f_2\|_{X^{s+1}_T} + \|\eta_w-\eta_*\|_{H^{s+\tdm}} \bigr)\|f_1-f_2\|_{X^{s+1}_T}
\end{multline}
for some $\tilde{C_2}: (\Rr_+)^5\to \Rr_+$ depending only on $(d, b, c^0, s, \sigma, g)$ and increasing in each argument. Setting 
\[
C_2=\tilde{C}_2(\|\eta_*\|_{H^{s+\fdm}},\|\eta_w\|_{H^{s+\tdm}}, c^{-1}_{\Psi_0[\eta_*]}, 1, 1),
\]
 we strengthen \eqref{stab:small1} to 
\bq\label{stab:small2}
\| \eta_w-\eta_*\|_{H^{s+\tdm}}<\frac{1}{3\max\{C_1, C_2\}},\quad \delta<\frac{1}{3\max\{C_1, C_2\}},\quad \|f_0\|_{H^{s+1}}<\frac{\delta}{3C_1}.
\eq
Then $\cF: \overline{B(0, \delta)}\to \overline{B(0, \delta)}$ is a contraction. Therefore, $\cF$ has a unique fixed point $f$ in $\overline{B(0, \delta)}$, which is a solution of \eqref{evolution for f} on $[0, T]$. Since the smallness conditions in \eqref{stab:small2} are independent of $T>0$, obtain a global solution $f\in X^{s+1}_T$ for all $T>0$. Moreover, it follows from \eqref{est:cF(f)} and \eqref{stab:small1} that 
\[
\| f\|_{X^{s+1}_T}=\| \cF(f)\|_{X^{s+1}_T}\le C_1\big( \|f_0\|_{H^{s+1}}+\frac{2}{3C_1}\| f\|_{X^{s+1}_T}\big),
\]
whence $\| f\|_{X^{s+1}_T}\le 3C_1\| f_0\|_{H^{s+1}}$ for all $T>0$. This proves the global bound \eqref{global in time estimate for f} with $M_1=3C_1$.

2. Exponential decay. We note that \eqref{eq: perturbation evolution} is of the form \eqref{linear evolution equation for f} with $F=N(f)$. Therefore, we can apply the estimate \eqref{energyineq:feps} with $\sigma=s+1$ to have
\bq\label{dE:1}
\begin{aligned}
E'(t) &\le -c\|  f\|_{H^{s+\tdm}}^2+C_3\| N(f)\|_{H^{s-\mez}}^2,
\end{aligned}
\eq
where $c=\frac{c_{\Psi_0[\eta_*]}}{4\tilde{C_1}}$, $\tilde{C_1}=\tilde{C_1}(\|\eta_*\|_{H^{s+\tdm}})$, $C_3=C_3(\|\eta_*\|_{H^{s+\tdm}})$, and  the energy $E(t)$ is given by \eqref{def:energy}, i.e.,
\[
E(t):= \mez A  \langle f, T_{\eta_*}f \rangle_{L^2,L^2}+\mez \sum_{|\alpha|=s+1} \langle \p^\alpha f, T_{\eta_*}\p^\alpha f \rangle_{L^2,L^2},\quad A=A(\| \eta_*\|_{W^{s+3, \infty}}, c_\Psi).
\]
It follows from \eqref{dE:1} and \eqref{est:N:final} that 
\bq
E'(t) \le -c\|  f\|_{H^{s+\tdm}}^2+C_4\big(\| \eta_w-\eta_*\|_{H^{s+\tdm}}^2+\| f\|_{H^{s+1}}^2\big)\| f\|_{H^{s+\tdm}}^2,
\eq
$C_4=C_4(\|\eta_*\|_{H^{s+\tdm}}, \|\eta_w\|_{H^{s+\tdm}})$. Now we assume in addition to \eqref{stab:small2} that 
\bq\label{stab:small3}
\| \eta_w-\eta_*\|^2_{H^{s+\tdm}}< \frac{c}{3C_4},\quad \delta^2<\frac{c}{3C_4},
\eq
so that $E'(t) \le -\frac{c}{3}\|  f\|_{H^{s+\tdm}}^2$.  We recall that, in view of \eqref{E and Sobolev energy comparable}, $E(t)$ is equivalent to $\| f(t)\|_{H^{s+1}}^2$ up to multiplicative constants depending only on $(\| \eta_*\|_{H^{s+\fdm}}, c_{\Psi_0[\eta_*]})$. Therefore, we obtain the exponential decay 
\bq\label{stab:small4}
\| f(t)\|_{H^{s+1}}^2\le C_5\| f_0\|_{H^{s+1}}^2e^{-c_0t}
\eq
for some $C_5=C_5(\| \eta_*\|_{H^{s+\fdm}}, c_{\Psi_0[\eta_*]})$ and $c_0=c_0(\| \eta_*\|_{H^{s+\fdm}}, c_{\Psi_0[\eta_*]})$. This proves \eqref{decayest}.  Finally,  the smallness conditions stated in \cref{theo:stab} follow from \eqref{stab:small1} and \eqref{stab:small3}. 
\end{proof}
Finally, we recall that, as stated in \cref{coro:intro},  \cref{large traveling wave for Stokes} provides a class of large traveling wave solutions which satisfy the stability condition \eqref{stab:cd:thm} and are thus asymptotically stable. 
%%%%%%%%%%%%%%%%%%%%%%%%%%%%%%%%%%%%%%%%%%%%%%%%%%%%%%%%%%%%%%%%%%%%%%%%%%%%%%%%%%%%%%%%%%%%%%%%%%%%%%%%%%%%
\appendix
\section{}
In this appendix we record various analytic  tools used throughout the paper.
\subsection{Liftings and Lagrange multipliers}\label{appendix:lifting}
\begin{lemm}\label{solving div problem}
    Let $\eta\in L^{\infty}(\T^d)$ satisfy (\ref{eta lower bound}). Then for any $g\in L^2(\Omega_\eta)$, there exists $v\in {}_0H^1(\Omega_\eta)$ with $\dv v = g$, and 
    \bq \label{div problem estimate}
    \|v\|_{H^1(\Omega_\eta)} \leq c(\|\eta\|_{L^{\infty}(\T^d)})\|g\|_{L^2(\Omega_\eta)},
    \eq
    where $c:\Rr^+\to \Rr^+$ depends only on $(d,b,c^0)$. 
\end{lemm}
\begin{proof}
    Let $U = \T^d \times (-2\|\eta\|_{L^\infty}-3b,\|\eta\|_{L^\infty})\supset \Omega_\eta$, and extend $g$ by zero to  $U$. 
There exists  a unique  $\phi \in H^1_0(U)\cap H^2(U)$ satisfying $\Delta \phi=\bar{g}$ in $U$, with
    \bq \label{bound on phi in H2}
    \|\phi\|_{H^2(U)}\leq c(d,b,\|\eta\|_{L^\infty}) \|g\|_{L^2(\Omega_\eta)}.
    \eq
We define $v= (v',v_{d+1})(x,y):\Omega_\eta \to \Rr^{d+1}$ by
    \bq
    \begin{split}
        v'(x,y)&= \nabla'\phi(x,y) + 3 \nabla' \phi (x,-(y+b)-b) - 4\nabla' \phi(x,-2(y+b)-b),\quad \na':=\na_x, \\
        v_{d+1}(x,y) &= \partial_y \phi(x,y) -3 \partial_y\phi(x,-(y+b)-b) + 2 \partial_y\phi(x,-2(y+b)-b).
    \end{split}
    \eq
We have $v(x,-b)=0$.  Since $\bar{g}$ is zero on $\T^d\times (-2\|\eta\|_\infty-3b,-b)$, it follows that
    \bq
    \dv v (x,y) = \Delta \phi (x,y) + 3 \Delta \phi (x,-(y+b)-b) - 4 \Delta\phi (x,-2(y+b)-b) = g(x,y)
    \eq
    for any $(x,y)\in \Omega_\eta$. Thus by (\ref{bound on phi in H2}) we have $v\in {}_0H^1(\Omega_\eta)$ and 
    (\ref{div problem estimate}) is satisfied.
\end{proof}
\begin{prop}\label{coro: lagrange multiplier for pressure}
    Suppose $\Lambda\in ({}_0H^1(\Omega_\eta))^*$ is such that $\Lambda\vert_{{}_0H^1_\sigma(\Omega_\eta)} =0$. Then there exists a unique $p\in L^2(\Omega_\eta)$ such that
    $$\langle \Lambda , u \rangle_{\Omega_\eta} = \int_{\Omega_\eta} p\dv u,$$
    for all $u\in {}_0H^1(\Omega_\eta) $. Moreover, we have the bound
    $$\|p\|_{L^2}\leq c(\|\eta\|_{L^{\infty}}) \|\Lambda\|_{({}_0H^1(\Omega_\eta))^*}, $$
    where $c:\Rr^+\to \Rr^+$ depends only on $(d,b,c^0)$.
\end{prop}
\begin{proof}
For any $p\in L^2(\Omega_\eta)$, the mapping
$$u\mapsto \int_{\Omega_\eta} p\dv u$$
defines a bounded linear functional on ${}_0H^1(\Omega_\eta)$, so by Riesz's representation theorem, there exists a unique $w_p\in {}_0H^1(\Omega_\eta) $ such that 
\bq \label{wp definition}
\int_{\Omega_\eta} p \dv u = (w_p,u)_{{}_0H^1(\Omega_\eta)}, \quad \forall u\in {}_0H^1(\Omega_\eta).
\eq
Choosing $u=w_p$ yields $\|w_p\|_{H^1}\leq c(d) \|p\|_{L^2}$.  On the other hand, by  \cref{solving div problem}, there exists $v\in {}_0H^1(\Omega_\eta)$ such that $\dv v = p$, and $\|v\|_{H^1} \leq c(\|\eta\|_{L^{\infty}})\|p\|_{L^2}$. Upon choosing $u=v$ in (\ref{wp definition}), we obtain $\|p\|_{L^2}\leq c(\|\eta\|_{L^{\infty}}) \|w_p\|_{H^1}$. Thus the norms of $p$ and $w_p$ are comparable, implying that the map   $Q: L^2(\Omega_\eta) \to {}_0H^1(\Omega_\eta)$ that maps  $p$ to $w_p$  has a closed range. Consequently  ${}_0H^1(\Omega_\eta)=\text{Ran}(Q)\oplus (\text{Ran}(Q))^\perp$.  But it is clear from \eqref{wp definition} that $(\text{Ran}(Q))^\perp= {}_0H^1_\sigma(\Omega_\eta)$. Therefore, $\text{Ran}(Q)={}_0H^1_\sigma(\Omega_\eta)^\perp$ which concludes the proof.
\end{proof}
\begin{lemm}\label{lifting neumann}
    Let $\eta\in W^{2,\infty}(\T^d)$ satisfy (\ref{eta lower bound}). Then for any $h\in H^\mez(\T^d)$, there exists $v_h \in {}_0H^1(\Omega_\eta)$ such that $v_h(\cdot, \eta(\cdot)) \cdot \cN(\cdot) = h(\cdot) $, and also 
    \bq\label{vh bound}
    \|v_h\|_{H^1} \leq c(\|\eta\|_{W^{2,\infty)}}) \|h\|_{H^\mez},
    \eq
    with $c:\Rr^+ \to \Rr^+$ depending only on $(d,b,c^0)$.
\end{lemm}
\begin{proof}
  For $a=\frac{1}{|\T^d|}\int_{\T^d} h$, the Neumann problem 
\[\label{Neumann}
\begin{cases}
    \Delta \theta = 0 \qquad &\text{in  }\Omega_\eta \\
    \nabla \theta (\cdot,\eta(\cdot))\cdot \cN(\cdot) =h(\cdot) \qquad &\text{on  }\T^d,\\
    \p_y\tt(\cdot, -b)=a
\end{cases}
\]
 has a solution $\theta\in H^2(\Omega_\eta)$ with mean zero and 
\[ \label{theta bound}
\|\theta\|_{H^2}\leq c(\|\eta\|_{W^{2,\infty}(\T^d)}) \|h\|_{H^{\mez}(\T^d)}.
\]
Let $\zeta: \T^d\times \Rr$ be a smooth  function that is identically one near $\Sigma_\eta$ and identically zero near $\Sigma_{-b}$.  Then 
$v_h = \nabla (\zeta \theta)$ is a desired vector field. 
\end{proof}
\begin{prop} \label{Lambda representation by k}
    Let $\eta\in W^{2,\infty}(\T^d)$ satisfy (\ref{eta lower bound}). Suppose $\Lambda\in ({}_0H^1(\Omega_\eta))^*$ is such that $\Lambda\vert_{{}_0H^1_\cN(\Omega_\eta)} =0$. Then there exists a unique scalar  distribution $\chi\in H^{-\mez}(\T^d)$ such that
    $$\langle \Lambda , u \rangle_{\Omega_\eta} = \langle \chi, u(\cdot,\eta(\cdot))\cdot \cN(\cdot)\rangle_{H^{-\mez},H^\mez} ,$$
    for all $u\in {}_0H^1(\Omega_\eta) $. Moreover, we have the bound
    $$\|\chi\|_{H^{-\mez}}\leq c(\|\eta\|_{W^{2,\infty}}) \|\Lambda\|_{({}_0H^1(\Omega_\eta))^*}, $$
    where $c:\Rr^+\to \Rr^+$ depends only on $(d,b,c^0)$.
\end{prop}
\begin{proof}
We have $\cN\in W^{1, \infty}(\T^d)$ for $\eta\in W^{2, \infty}$. Hence, for any $\chi\in H^{-\mez}(\T^d)$, the mapping
$$u\mapsto \langle \chi, u(\cdot,\eta(\cdot))\cdot \cN(\cdot)\rangle_{H^{-\mez},H^\mez}$$
defines a bounded linear functional on ${}_0H^1(\Omega_\eta)$. By Riesz's representation theorem, there exists a unique $w_\chi\in {}_0H^1(\Omega_\eta) $ such that 
\bq \label{wk definition}
\langle \chi, u(\cdot,\eta(\cdot))\cdot \cN(\cdot)\rangle_{H^{-\mez},H^\mez} = (w_\chi, u)_{{}_0H^1(\Omega_\eta)}, \quad \forall u\in {}_0H^1(\Omega_\eta).
\eq
Choosing $u=w_\chi$ yields 
$$\|w_\chi\|_{H^1}\leq c(d,\|\eta\|_{W^{2,\infty}}) \|\chi\|_{H^{-\mez}}.$$

Appealing to \cref{lifting neumann}, for any $h\in H^\mez(\T^d)$, there exists $v_h \in {}_0H^1(\Omega_\eta)$ such that $v_h(\cdot,\eta(\cdot))\cdot \cN(\cdot) = h(\cdot)$, with the bound (\ref{vh bound}). Setting $u=v_h$ in (\ref{wk definition}) then yields
$$\|\chi\|_{H^{-\mez}}\leq c(\|\eta\|_{W^{2,\infty}}) \|w_\chi\|_{H^1}.$$
The preceding bounds show that the map $S:H^{-\mez}(\T^d)\to {}_0H^1(\Omega_\eta)$ sending $\chi$ to $w_\chi$ has a closed range. Then, using  \eqref{wk definition} we deduce that  $\text{Range}(S) = ({}_0H^1_\cN(\Omega_\eta))^\perp$. 
\end{proof}
\begin{lemm}\label{lifting div in HcN}
    Let $\eta \in W^{2,\infty}(\T^d)$ satisfy \eqref{eta lower bound}. Then for any $g\in \mathring{L}^2(\Omega_\eta)$, there exists $v\in {}_0H^1_\cN(\Omega_\eta)$ such that $\dv v = g$, and also
    \bq\label{bound:liftHcN}
        \|v\|_{H^1(\Omega_\eta)} \leq c(\|\eta\|_{W^{2,\infty}})\|g\|_{L^2(\Omega_\eta)}
    \eq
    for some $c:\Rr^+ \to\Rr^+$ depending only on $(d,b,c^0)$.
\end{lemm}
\begin{proof}
    From classic elliptic regularity, we know that the Neumann problem
    \bq
        \begin{cases}
         \Delta \theta = g \qquad &\text{in  }\Omega_\eta, \\
        \nabla \theta (\cdot,\eta(\cdot))\cdot \cN(\cdot) =0 \qquad &\text{on  }\T^d,\\
        \partial_y \theta(\cdot , -b)=0
        \end{cases}
    \eq
    has a solution $\theta \in H^2(U)$ with mean zero and 
    $$\|\theta\|_{H^2(U)}\leq c(\|\eta\|_{W^{2,\infty}(\T^d)}) \|g\|_{L^2(\Omega_\eta)}.$$
    Let $\beta\in C^{\infty}(\Rr)$ be such that $\beta(-b) =1$, $\beta(y) = 0$ for $y\geq -b + \frac{c^0}{2}$, and $\int_{-b}^{-b+\frac{c^0}{2}} \beta(y)dy = 0 $. We define $w=(w', w_{d+1}): \Omega_\eta\to \Rr^{d+1}$ by $w'(x, y)=-\na_x\tt(x, -b)\beta(y)$, $w_{d+1}= \Delta_x\tt(x, -b)\int_{-b}^y\beta(z)dz$.     We have $w'(x,-b) = -\na_x\tt(x, -b) $, $w_{d+1}(x,-b) = 0$, $w(x,y) = 0$ for $y\geq -b + \frac{c^0}{2}$, and  $\dv w = 0$. Then $v= \nabla\theta  + w $ is the desired vector field provided $w\in H^1(\Omega_\eta)$. From the definition of $w$, we see that  $w\in H^1(\Omega_\eta)$ if $\tt(\cdot, -b)\in H^3(\T^d)$. This is the case if $g\in H^2(\Omega_\eta)$. Therefore, for $g\in \mathring{L}^2(\Omega_\eta)$, we approximate $g$ by $g_n\in \rH^2(\Omega_\eta)$ and obtain a $v_n\in {}_0H^1_\cN(\Omega_\eta)$ satisfying  $\dv v_n=g_n$ and \eqref{bound:liftHcN}. In particular, $\{v_n\}$ is bounded in ${}_0H^1_\cN(\Omega_\eta)$. Since $\Omega_\eta$ is bounded, any weak $H^1$ limit of $v_n$ gives a desired vector field. 
\end{proof}
\begin{prop}\label{Lagrange for pressure in HcN}
    Suppose $\Lambda\in ({}_0H^1_\cN(\Omega_\eta))^*$ is such that $\Lambda\vert_{{}_0H^1_\cN(\Omega_\eta)\cap {}_0H^1_\sigma(\Omega_\eta)} =0$. Then there exists a unique $p\in \mathring{L}^2(\Omega_\eta)$ such that
    $$\langle \Lambda , u \rangle_{\Omega_\eta} = \int_{\Omega_\eta} p\dv u,$$
    for all $u\in {}_0H^1_\cN(\Omega_\eta) $. Moreover, we have the bound
    $$\|p\|_{L^2}\leq c(\|\eta\|_{W^{2,\infty}}) \|\Lambda\|_{({}_0H^1_\cN(\Omega_\eta))^*}, $$
    where $c:\Rr^+\to \Rr^+$ depends only on $(d,b,c^0)$.
\end{prop}
\begin{proof}
    The proof is similar to that of \cref{coro: lagrange multiplier for pressure}, using the lifting provided in \cref{lifting div in HcN}.
\end{proof}
%%%%%%%%%%%%%%%%%%%%%%%%%%%%%%%%%%%%%%%%%%%%%%%%%%%%%%%
\subsection{Product estimates}\label{appendix:productestimates}

\begin{prop}\label{prop: composition regularity}
Let  $X$ be either $\Rr^d$, or $\T^d$, or $\Omega_\eta$ with $\eta\in W^{1, \infty}(\T^d)$ satisfy \eqref{eta lower bound}. For any $s\ge 0$, there exists $C: \Rr_+\to \Rr_+$ such that the following assertions hold. When $X=\Omega_\eta$, $C$ depends only on $(d, s, c^0, \| \na \eta\|_{L^\infty(\T^d)})$.

(i) For any $F\in C^\infty(\Rr)$ with $F(0)=0$ and $u\in H^s(X)\cap L^\infty(X)$, we have
 \bq\label{nonl:est}
        \|F(u)\|_{H^s(X)} \leq C(\|u\|_{L^\infty(X)}) \|u\|_{H^s(X)}.
    \eq
    (ii)  For any $u,v\in  H^s(X)\cap L^\infty(X)$, we have
    \bq\label{tame:product}
        \|uv\|_{H^s(X)}\leq C\bigl(\|u\|_{H^s(X)}\|v\|_{L^\infty(X)} + \|v\|_{H^s(X)}\|u\|_{L^\infty(X)} \bigr).
    \eq
\end{prop}
\begin{proof}
The cases $X=\Rr^d$ and $X=\T^d$  can be found in \cite{BCD}. The case $X=\Omega_\eta$  follows from the case $X=\Rr^{d+1}$ and the use of Stein's extension operator for bounded Lipschitz domains  \cite{Stein}.
\end{proof}
In $\Rr^d$ and $\T^d$, we have the following Sobolev product estimates, whose proof can be found in Corollary 2.11, \cite{ABZ3}. 
\begin{lemm}\label{product estimate on torus}
   Let  $X=\Rr^d$ or $X=\T^d$. Consider  $(s_0, s_1, s_2)\in \Rr^3$  satisfying $s_1+s_2>\max\{0, s_0+\frac{d}{2}\}$, and $s_0\le s_j$, $j=1,2$. There exists $C=C(d, s_0, s_1, s_2)>0$ such that 
    \bq\label{productest:Sobolev}
    \|uv\|_{H^{s_0}(X)} \leq C\|u\|_{H^{s_1}(X)}\|v\|_{H^{s_2}(X)}\quad\forall u\in H^{s_1}(X),\; v\in H^{s_2}(X).
    \eq
\end{lemm}

By invoking Stein's extension operator \cite{Stein} for bounded Lipschitz domains, we deduce the following from Lemma \ref{product estimate on torus}.
\begin{lemm}\label{product estimate for domain}
Let   $\eta\in W^{1, \infty}(\T^d)$ satisfy \eqref{eta lower bound}. Consider  $(s_0, s_1, s_2)\in [0, \infty)$  satisfying $s_1+s_2>\max\{0, s_0+\frac{d+1}{2}\}$, and $s_0\le s_j$, $j=1,2$. 
 Then for all $F\in H^{s_1}(\Omega_\eta)$ and $G\in H^{s_2}(\Omega_\eta)$, we have
    \bq
    \|FG\|_{H^{s_0}(\Omega_\eta)} \leq C(\|\na \eta\|_{L^\infty(\T^d)})\|F\|_{H^{s_1}(\Omega_\eta)}\|G\|_{H^{s_2}(\Omega_\eta)},
    \eq
    where $C: \Rr^+\to \Rr^+$ depends only on $(d,b, s_0, s_1, s_2, c^0)$. 
\end{lemm}
The next lemma is a direct consequence of  \cref{prop: composition regularity} and \cref{product estimate for domain}.
\begin{lemm}\label{product estimate for inverse}
Let  $s>(d+1)/2$ and $\eta\in W^{1, \infty}(\T^d)$ satisfy \eqref{eta lower bound}. Moreover, let $G\in H^{s}(\Omega_\eta)$ be such that $\inf G \geq M >0$. Then for all $ \sigma\in [0, s]$, $F\in H^{\sigma}(\Omega_\eta)$, and $\l\in \Nn$, we have
    \bq\label{est: product with inverse}
    \|FG^{-l}\|_{H^\sigma(\Omega_\eta)} \leq C(\|G\|_{H^{s}(\Omega_\eta)})\|F\|_{H^\sigma(\Omega_\eta)},
    \eq
    where $C$ depends only on $(d,b,s,\sigma,c^0,M,l,\|\na\eta\|_{L^\infty})$.
\end{lemm}

\subsection{Paradifferential calculus}\label{appendix:para}
\begin{defi}
Given~$\rho\in [0, 1]$ and~$m\in\Rr$,~$\Gamma_{\rho}^{m}({\T}^d)$ denotes the space of
locally bounded functions~$a(x,\xi)$
on~$\T^d\times({\Rr}^d\setminus 0)$,
which are~$C^\infty$ with respect to~$\xi$ for~$\xi\neq 0$ and
such that, for all~$\alpha\in\Nn^d$ and all~$\xi\neq 0$, the function
$x\mapsto \partial_\xi^\alpha a(x,\xi)$ belongs to~$W^{\rho,\infty}(\T^d)$ and there exists a constant
$C_\alpha$ such that
\begin{equation*}%\label{para:10}
\forall\la \xi\ra\ge \frac18,\quad 
\lA \partial_\xi^\alpha a(\cdot,\xi)\rA_{W^{\rho,\infty}(\T^d)}\le C_\alpha
(1+\la\xi\ra)^{m-\la\alpha\ra}.
\end{equation*}
\end{defi}
Given a symbol~$a$, 
the paradifferential operator~$T_a$ acting on a periodic distribution $u: \T^d\to \Rr$ is define  by
\begin{equation}\label{eq.para}
\widehat{T_a u}(n)=(2\pi)^{-d}\sum_{m\in \Zz} \chi(n-m, m)\widehat{a}(n-m, m)\psi(m)\widehat{u}(m),
\end{equation}
where
$\widehat{a}(n, \xi)=\int_{\T^d}e^{-ix\cdot n}a(x,\xi)\, dx$
is the Fourier transform of~$a$ with respect to the first variable; 
$\chi$ and~$\psi$ are two fixed~$C^\infty$ functions such that:
\begin{equation}\label{cond.psi}
\psi(\eta)=0\quad \text{for } \la\eta\ra\le \frac{1}{5},\qquad
\psi(\eta)=1\quad \text{for }\la\eta\ra\geq \frac{1}{4},
\end{equation}
and~$\chi(\theta,\eta)$ 
satisfies, for~$0<\eps_1<\eps_2$ small enough,
$$
\chi(\theta,\eta)=1 \quad \text{if}\quad \la\theta\ra\le \eps_1(1+\la \eta\ra),\qquad
\chi(\theta,\eta)=0 \quad \text{if}\quad \la\theta\ra\geq \eps_2(1+\la\eta\ra),
$$
and such that
$$
\forall (\theta,\eta)\,:\qquad \la \partial_\theta^\alpha \partial_\eta^\beta \chi(\theta,\eta)\ra\le 
C_{\alpha,\beta}(1+\la \eta\ra)^{-\la \alpha\ra-\la \beta\ra}.
$$
\begin{defi}\label{defiGmrho}
For~$m\in\Rr$,~$\rho\in [0,1]$ and~$a\in \Gamma^m_{\rho}(\T^d)$, we set
\begin{equation}\label{defi:norms}
M_{\rho}^{m}(a)= 
\sup_{\la\alpha\ra\le 2(d+2) +\rho ~}\sup_{\la\xi\ra \ge 1/2~}
\lA (1+\la\xi\ra)^{\la\alpha\ra-m}\partial_\xi^\alpha a(\cdot,\xi)\rA_{W^{\rho,\infty}(\T^d)}.
\end{equation}
\end{defi}
\begin{theo}\label{theo:sc0}
Let~$m\in\Rr$ and~$\rho\in [0,1]$.

$(i)$ If~$a \in \Gamma^m_0(\T^d)$, then~$T_a$ is of order~$ m$. 
Moreover, for all~$\mu\in\Rr$ there exists a constant~$K$ such that
\begin{equation}\label{esti:quant1}\lA T_a \rA_{H^{\mu}\rightarrow H^{\mu-m}}\le K M_{0}^{m}(a).
\end{equation}
$(ii)$ If~$a\in \Gamma^{m}_{\rho}(\T^d), b\in \Gamma^{m'}_{\rho}(\T^d)$ then 
$T_a T_b -T_{a b}$ is of order~$ m+m'-\rho$. 
Moreover, for all~$\mu\in\Rr$ there exists a constant~$K$ such that
\begin{equation}\label{esti:quant2}
\lA T_a T_b  - T_{a b}   \rA_{H^{\mu}\rightarrow H^{\mu-m-m'+\rho}}%&
\le 
K M_{\rho}^{m}(a)M_{0}^{m'}(b)+K M_{0}^{m}(a)M_{\rho}^{m'}(b).
\end{equation}
\end{theo}

\begin{defi}[Zygmund spaces] Consider a Littlewood-Paley decomposition: 
$u=\sum_{q=0}^\infty \Delta_q u$. 
If~$s$ is any real number, we define the Zygmund class~$C^{s}_*(\T^d)$ as the 
space of tempered distributions~$u$ such that
$$
\lA u\rA_{C^{s}_*}:= \sup_q 2^{qs}\lA \Delta_q u\rA_{L^\infty}<+\infty.
$$
\end{defi}
\begin{rema}
Recall that $C^{s}_*(\T^d)$ is the 
H\"older 
space~$W^{s,\infty}(\T^d)$ if~$s\in (0,+\infty)\setminus \Nn$. 
\end{rema}

\begin{defi}
Given two functions~$a,b$ defined on~$\T^d$ we define the remainder 
$$
R(a,u)=au-T_a u-T_u a.
$$
\end{defi}
\begin{theo} 
(i) For any $\alpha,\beta\in \Rr$, if~$\alpha+\beta>0$ then
\begin{align}
&\lA R(a,u) \rA _{H^{\alpha + \beta-\frac{d}{2}}(\T^d)}
\leq K \lA a \rA _{H^{\alpha}(\T^d)}\lA u\rA _{H^{\beta}(\T^d)},\label{Bony} \\ 
&\lA R(a,u) \rA _{H^{\alpha + \beta}(\T^d)} \leq K \lA a \rA _{C^{\alpha}_*(\T^d)}\lA u\rA _{H^{\beta}(\T^d)}.\label{Bony3}
\end{align}
(ii) For any $m>0$ and~$s\in \Rr$, 
\begin{equation}
\lA T_a u\rA_{H^{s-m}}\le K \lA a\rA_{C^{-m}_*}\lA u\rA_{H^{s}}.\label{niS}%\\
\end{equation}
(iii) Let $s_0,s_1,s_2$ be such that 
$s_0\le s_2$ and $s_0 < s_1 +s_2 -\frac{d}{2}$. Then
\begin{equation}\label{boundpara}
\lA T_a u\rA_{H^{s_0}}\le K \lA a\rA_{H^{s_1}}\lA u\rA_{H^{s_2}}.
\end{equation}
(iv)  Let $s_0, s_1,s_2$ be such that $s_1+s_2>0$, $s_0\le s_1$, and  $s_0 < s_1+s_2-\frac{d}{2}$. Then 
\bq\label{lemPa1}
\lA au - T_a u\rA_{H^{s_0}}\le K \lA a\rA_{H^{s_1}}\lA u\rA_{H^{s_2}}.
\eq
\end{theo}

%%%%%%%%%%%%%%%%%%%%%%%%%%%%%%%%%%%%%%%%%%%%%%%%%%%%%%%%%%%%%%%%%%%%%%%%%%%%%%%%%%%%%%%%%%%%%%%%%%%%%%%%%

\vspace{.1in}
{\noindent{\bf{Acknowledgment.}}   The authors were partially supported by NSF grant DMS-2205710. The authors thank I. Tice for a question on the solvability  of the capillary-gravity operator. 
}

\end{document}